\documentclass{article}
\usepackage{arxiv}

\usepackage[utf8]{inputenc} % allow utf-8 input
\usepackage[T1]{fontenc}    % use 8-bit T1 fonts
\usepackage{hyperref}       % hyperlinks
\usepackage{url}            % simple URL typesetting
\usepackage{booktabs}       % professional-quality tables
\usepackage{amsfonts}       % blackboard math symbols
\usepackage{nicefrac}       % compact symbols for 1/2, etc.
\usepackage{microtype}      % microtypography
\usepackage{doi}
\usepackage{xcolor}

\usepackage{amsthm,amsmath,amssymb}
\usepackage{graphicx,xcolor}
\usepackage[normalem]{ulem}
\usepackage{setspace}
\usepackage{accents}
\usepackage{tikz}
\usetikzlibrary{calc}
\usetikzlibrary{arrows}
\usepackage{multicol}
\usepackage{cite}
\usepackage{comment}

\usepackage{cite,hyperref}
\usepackage{amsmath,amssymb,amsfonts,amsthm}
\usepackage{algorithmic}
\usepackage{graphicx}
\usepackage{textcomp}
\usepackage{pifont}% http://ctan.org/pkg/pifont
\newcommand{\cmark}{\ding{51}}%
\newcommand{\xmark}{\ding{55}}%
\usepackage{subcaption}
\usepackage{amsmath,bm}
\DeclareMathOperator*{\argmin}{argmin}

\graphicspath{ {./Figures/} }
\usepackage{algorithm}
\newtheorem{theorem}{Theorem}

\newtheorem{lemma}{Lemma}
\newtheorem{assumption}{Assumption}

\newtheorem{remark}{Remark}
\usepackage{hyperref}

\def\grad{\nabla}

\def\bv{\mathbf{v}}

\def\bx{\mathbf{x}}  %{\mbox{\boldmath $\lambda$}}
\def\by{\mathbf{y}}

\def\cC{\mathcal{C}}

\def\cE{\mathcal{E}}

\def\cG{\mathcal{G}}

\def\cI{\mathcal{I}}

\def\cN{\mathcal{N}}
\def\cO{\mathcal{O}}

\def\cV{\mathcal{V}}

\def\cX{\mathcal{X}}

\def\mR{\mathbb{R}}

\def\smskip{\smallskip}

\def\texitem#1{\par\smskip\noindent\hangindent 25pt
               \hbox to 25pt {\hss #1 ~}\ignorespaces}

\def\norm#1{\left\|#1\right\|}

\newcommand{\BEAS}{\begin{eqnarray*}}
\newcommand{\EEAS}{\end{eqnarray*}}
\newcommand{\BEA}{\begin{eqnarray}}
\newcommand{\EEA}{\end{eqnarray}}
\newcommand{\BEQ}{\begin{eqnarray}}
\newcommand{\EEQ}{\end{eqnarray}}
\newcommand{\BIT}{\begin{itemize}}
\newcommand{\EIT}{\end{itemize}}
\newcommand{\BNUM}{\begin{enumerate}}
\newcommand{\ENUM}{\end{enumerate}}

\newcommand{\BA}{\begin{array}}
\newcommand{\EA}{\end{array}}

\newcommand{\reals}{\mathbb{R}}
\newcommand{\integers}{\mathbb{Z}}

\newcommand{\diag}{\mathop{\bf diag}}

\newif\ifpagenumbering
\pagenumberingtrue

\pagenumberingfalse

\newsavebox{\theorembox}
\newsavebox{\lemmabox}
\newsavebox{\defnbox}
\newsavebox{\corollarybox}
\newsavebox{\assbox}
\savebox{\theorembox}{\noindent\bf Theorem}
\savebox{\lemmabox}{\noindent\bf Lemma}
\savebox{\defnbox}{\noindent\bf Definition}
\savebox{\corollarybox}{\noindent\bf Corollary}
\newtheorem{defn}{\usebox{\defnbox}}

\usepackage{todonotes}

\def\sa#1{\textcolor{black}{#1}}
\def\fin#1{\textcolor{black}{#1}}
\def\rev#1{\textcolor{black}{#1}}

\newcommand{\vertiii}[1]{{\left\vert\kern-0.25ex\left\vert\kern-0.25ex\left\vert #1
    \right\vert\kern-0.25ex\right\vert\kern-0.25ex\right\vert}}
\def\BibTeX{{\rm B\kern-.05em{\sc i\kern-.025em b}\kern-.08em
    T\kern-.1667em\lower.7ex\hbox{E}\kern-.125emX}}

\title{A Fast Row-Stochastic \fin{Decentralized Method for Distributed Optimization} Over Directed Graphs\thanks{This work was supported by PDMA (NORTE-08-5369-FSE000061), SNAP - NORTE-01-0145-FEDER-000085, funded by ERDF NORTE2020/PORTUGAL2020; RELIABLE (PTDC/EEI-AUT/3522/2020), R$\&$D Unit SYSTEC - Base (UIDB/00147/2020), Programmatic (UIDP/00147/2020) and ARISE (LA/
P/0112/2020) funded by national funds through FCT/MCTES (PIDDAC). 
The work of N. S. Aybat was supported by Office of Naval Research (ONR) through ONR research grant N00014-21-1-2271.}
\thanks{The authors are listed according to their contribution to the work, from the most to the least.}}

\author{
Diyako Ghaderyan\\
Research Center for Systems and Technologies(SYSTEC)\\
University of Porto(FEUP), Portugal\\
\texttt{dghaderyan@fe.up.pt}
\And
Necdet Serhat Aybat\\
Industrial and Manufacturing Engineering Department\\
Penn State University\\
University Park, PA 16802,\\
\texttt{nsa10@psu.edu}
\And
A. Pedro Aguiar\\
Research Center for Systems and Technologies(SYSTEC)\\
University of Porto(FEUP), Portugal\\
\texttt{pedro.aguiar@fe.up.pt}
\And
Fernando Lobo Pereira\\
Research Center for Systems and Technologies(SYSTEC)\\
University of Porto(FEUP), Portugal\\
\texttt{flp@fe.up.pt}}

\date{}

\renewcommand{\shorttitle}{A Fast Row-Stochastic Decentralized Optimization Method Over Directed Graphs}

\begin{document}

\maketitle
\begin{abstract}
In this paper, we introduce a  fast row-stochastic decentralized algorithm, referred to as FRSD, {to solve consensus optimization problems over directed communication graphs.} 
The proposed algorithm only utilizes row-stochastic weights, {leading to certain practical advantages \sa{in broadcast communication settings} over those requiring column-stochastic weights}. {Under the assumption that
each node-specific function is smooth and strongly convex,} we show that %FRSD {admits constant step-size and momentum parameters such that} 
{the FRSD iterate sequence} converges %linearly 
\rev{with a linear rate} {to the optimal consensus solution}.
\sa{In contrast to {the 
%majority of 
existing methods for directed networks}, FRSD enjoys linear convergence without 
employing a gradient tracking~(GT) technique \sa{explicitly}, rather it implements GT implicitly with the use of a novel momentum term, which leads to a significant reduction in communication and storage overhead for each node when FRSD is implemented for solving high-dimensional problems over small-to-medium scale networks.}  In the numerical {tests}, 
we compare FRSD with
{other state-of-the-art} methods, which use row-stochastic {and/or} column-stochastic weights.
\end{abstract}
\allowdisplaybreaks
\vspace*{-3mm}
\section{Introduction }
{I}{n} recent {years},  
{rapid {advances} in} artificial intelligence and communication technologies have 
led to %large-scale 
\sa{computational network systems} 
{over which one has to solve}
optimization problems with enormous, {physically distributed and/or private} data sets in order to achieve 
system level objectives 
such that every agent, represented by a node in the network, has to agree {on a common decision}. 
To reach an optimal consensus decision, 
{decentralized optimization techniques can be used to solve a consensus optimization problem in a distributed manner employing only local computations and communication among neighboring computing nodes} {that can directly communicate with each other}. The 
{classic} {consensus optimization problem has the following form:}
\begin{align}
\label{eqz1}
{{\bm{x}^{*}}\in}\argmin\limits_{{\bm{x}}{\in\reals^p}}{\bar{f}}({\bm{x}}){\triangleq}\dfrac{1}{n}\sum\limits_{i=1}^{n}f_{i}({\bm{x}}),
\end{align}
where {the objective function $\bar{f}$} is the %sum 
{average of all} {individual} cost functions {$\{f_i\}_{i=1}^n$, where} $f_{i} : \mathbb{R}^{p}\rightarrow \mathbb{R}$ {is the private function of agent $i$.} 
This problem appears in 
{a variety of applications, e.g.}, 
sensor networks~\cite{rabbat2004distributed,khan2009diland}, 
distributed control~\cite{bullo2009distributed}, \rev{large-scale machine learning~\cite{cevher2014convex,boyd2011distributed,bekkerman2011scaling,raja2015cloud,assran2019stochastic,tsianos2012push,tsianos2012consensus},} distributed estimation~\cite{yuan2018exact}. 

{Next,  we first discuss the 
{previous work} 
{focusing on} undirected networks, and then we %briefly go over the 
summarize some related methods proposed for distributed consensus optimization over \emph{directed} networks.} \sa{Let  $\cG=(\cV,\cE)$ be the network of collaborative agents, and $\tilde{\bm{x}}_i$ denote a local estimate of the optimal decision for $i\in\cV$. For merely convex objectives, we call $\tilde{\mathbf{x}}=[\tilde{\bm{x}}_i]_{i\in\cV}$ an $\epsilon$-solution if $|\frac{1}{n}\sum_{i\in\cV}f_i(\tilde{\bm{x}}_i)-f^*|\leq\epsilon$ and the consensus violation satisfies $\max\{\norm{\tilde{\bm{x}}_i-\tilde{\bm{x}}_j}:~(i,j)\in\cE\}\leq \epsilon$ for $i\in\cV$, where $f^*$ denotes the optimal value of \eqref{eqz1}. On the other hand, for strongly convex problems, we say \rev{that} $\tilde{\mathbf{x}}$ is $\epsilon$-optimal, if $\norm{\tilde{\bm{x}}_i-\bm{x}^*}\leq \epsilon$ for all $i\in\cV$.} 

{Inspired by the seminal work~\cite{tsitsiklis1986distributed}, the authors in~\cite{nedic2009distributed} proposed a distributed (sub)gradient %descent 
method for solving \eqref{eqz1}}. {
{When each $f_i$ is closed convex, the method in~\cite{nedic2009distributed} is shown to have a} sublinear convergence rate, i.e., to compute an $\epsilon$-optimal solution, one needs {to evaluate $\cO(1/\epsilon^2)$ subgradients}
 -- {in contrast to linear or $\cO(1/\epsilon)$ sublinear convergence rates, the slower rate of subgradient methods is due to 
their use of a} diminishing step-size sequence or {of} a small fixed step size $\alpha=\cO(\epsilon)$.} 
{Moreover, when agents have simple closed convex constraint sets, distributed projected subgradient methods are proposed for solving
{
\begin{align}
\label{eq:constrained_consensus}
    \min\{\bar{f}(x):\ x\in\cX\};
\end{align} 
e.g., the method in~\cite{nedic2010constrained} solves \eqref{eq:constrained_consensus}
employing exact subgradient evaluations, while the method in~\cite{ram2010distributed} can handle subgradients corrupted by stochastic noise.}} 
{In \cite{duchi2011dual},
algorithms based on dual averaging of subgradients are studied for 
{solving \eqref{eq:constrained_consensus} assuming $\mathbf{0}\in\cX$}. {In \cite{ram2010distributed}, it is shown that 
an $\epsilon$-optimal solution can be computed with $\cO(n^{3}/\epsilon^2)$ iteration complexity that is independent of the network topology, whereas the algorithm proposed in \cite{duchi2011dual}
requires iteration complexity of $\cO(n^2/\epsilon^2)$ for paths or simple cycle graphs, $\cO(n/\epsilon^2)$ for 2-$d$ grids, and $\cO(1/\epsilon^2)$ for bounded degree expander graphs.}} 

{The authors in \cite{zhu2011distributed}  
{propose a} primal-dual subgradient algorithm to solve problems with a global constraint set defined as the intersection of local %constraint sets
\sa{ones}, {i.e., $\cX=\cX_0\cap(\bigcap_{i=1}^n)\cX_i$ in \eqref{eq:constrained_consensus} such that each agent-$i$ only knows $f_i$, $\cX_i$ and $\cX_0$}}. In~\cite{jakovetic2014fast}, \sa{under the assumption that the gradients are bounded and Lipschitz,} an improved convergence rate of {$\cO(\log(k)/k^2)$} {is obtained by employing Nesterov acceleration.} 
{%For smooth convex objective functions,
\fin{In a similar setting, assuming convex functions with Lipschitz gradients, the EXTRA method~\cite{shi2015extra}, using a fixed step size and utilizing the difference of two consecutive gradients in each update, is shown to generate a sequence that converges with a $\cO(1/k)$ rate,} %the rate can be improved to a linear rate
\fin{which improves to linear convergence}
under the additional assumption of strong convexity.} 
There are also distributed %methods
\fin{algorithms} based on 
alternating direction method of multipliers (ADMM) \fin{and the augmented Lagrangian method} achieving similar rates, e.g.,~\cite{wei20131,shi2014linear, mokhtari2016dqm, mokhtari2016decentralized} \fin{--some handling composite convex functions~\cite{aybat2017distributed,aybat2015asynchronous}, and some even handling complicated constraints~\cite{aybat16,aybat2019distributed}.}

{The works that we have discussed above are designed for undirected networks; hence, they correspond to balanced graphs if we treat undirected networks as a special case of directed networks.} {However,  
directed networks may well be unbalanced; this situation arises especially for directed time-varying networks. \rev{For general directed networks, the first works to employ push-sum consensus protocol~\cite{kempe2003gossip} (for computing an average over directed networks) within the distributed optimization framework (using the dual averaging method) are \cite{tsianos2012push,tsianos2012consensus}. A follow-up work in this direction is} the subgradient-push method proposed in~\cite{nedic2014distributed} that combines the push-sum protocol with the \fin{distributed} subgradient method~\cite{nedic2009distributed} (for minimization of convex functions). More precisely, \fin{the method in~\cite{nedic2014distributed}} applies to \eqref{eqz1} when each $f_i$ is a closed convex function and the directed communication network is time-varying, \sa{and achieves} a sublinear rate of $\cO(\log(k)/\sqrt{k})$ %can be achieved 
using a column-stochastic weight matrix and a diminishing step-size sequence.} 
{\sa{On the other hand, when when each $f_i$ is smooth \fin{(i.e., $\grad f_i$ is Lipschitz continuous)} and strongly convex, the DEXTRA algorithm proposed in~\cite{xi2017dextra}, which is a distributed method for directed graphs, achieves R-linear convergence rate;} it combines EXTRA~\cite{shi2015extra} with push-sum approach \cite{kempe2003gossip}. The step-size in DEXTRA is constant and should be carefully chosen belonging to a specific interval that may be unknown to the agents. 
Compared to DEXTRA, Push-DIGing \cite{nedic2017achieving} and ADD-OPT \cite{xi2017add} have a simpler step-size rule and can achieve R-linear rate on directed graphs with time-varying and static topology, respectively, using sufficiently small constant step-size.}
These approaches employ a column-stochastic weight matrix to achieve R-linear rate over strongly connected networks. Unlike the previous methods {that} {are based on push-sum}, {there are other works achieving \sa{linear convergence for the smooth and strongly convex setting} by employing both column-stochastic and row-stochastic weights, e.g.,}  $\mathcal{AB}$, {$\mathcal{AB}$m} and Push-Pull \cite{xin2018linear,xin2019distributed,pu2020push}. \sa{Later, $\mathcal{ABN}$ method is proposed in \cite{xin2019distributedNEST} which incorporates Nesterov's momentum term into $\mathcal{AB}$.} 

{It is important to note that designing column-stochastic weights requires the knowledge of neighbors' out-degree for each node; this requirement is impractical within broadcast-based communication systems.} To address this issue, 
{in \cite{xi2018linear}}, the authors proposed a method that only uses the \emph{row-stochastic} weights. \sa{This line of research, i.e., using only \emph{row-stochastic} weights, has attracted attention, and in follow-up papers, the algorithm in~\cite{xi2018linear} is extended to handle uncoordinated step-sizes in \rev{the} FROST algorithm~\cite{xin2019frost}, and to incorporate Nesterov acceleration leading to  \rev{the} FROZEN algorithm~\cite{xin2019distributedNEST}. Finally, \rev{the} D-DNGT algorithm proposed in~\cite{lu2020nesterov} employs heavy-ball
momentum and can handle nonuniform step-sizes.} \rev{Some recent work extended the synchronous methods for directed networks to the asynchronous computation setting, in which agents asynchronously update their iterates by using the currently available (possibly old) information, and they do not wait for the other agents to update in order to proceed to the next update, i.e., there is no global clock, e.g., \cite{tian2018asy,zhang2019asyspa,assran2020asynchronous}.}

\textit{Contributions:}
In this paper, we propose a fast row-stochastic decentralized algorithm, referred to as FRSD, to solve distributed consensus optimization problems  over directed communication networks. FRSD employs only row-stochastic weights, and we show that when $\{f_i\}_{i=1}^n$ are strongly convex and smooth, FRSD iterate sequence corresponding to a constant stepsize converges to the optimal consensus decision with a linear rate. While previous \sa{row-stochastic methods~\cite{xin2019distributedNEST,xi2018linear,xin2019frost,lu2020nesterov}} crucially depend on the gradient tracking technique to establish linear rate, in this paper we achieve the same result through introducing a novel momentum term \sa{which leads to \emph{implicit} gradient tracking}, \sa{i.e., FRSD does \emph{not} employ gradient tracking \emph{explicitly} at the node level\fin{--hence, 
%through the use of 
it does not require any computation simultaneously involving $\nabla f_{i}(\bm{x}_{i}(k+1))$ and $\nabla f_{i}(\bm{x}_{i}(k))$ for any node $i\in\cV$ in the $k$-th iteration;} but, it still manages to implement gradient tracking in an \textit{implicit} manner. \rev{This new dynamics proposed in this paper leads to: (i) reduction in the data stored, and (ii) reduction in the data broadcast, for each node. More precisely, FRSD does not need to store  $\bm{x}$ iterate from the previous iteration while it is needed for all other methods explicitly using the gradient tracking term; furthermore, FRSD also eliminates the need for broadcasting a variable related to gradient tracking.
In summary, in FRSD any agent-$i$ only needs to store a $2p+n$-dimensional vector, and to broadcast $n+p$-dimensional vector. Comparing with the other row-stochastic methods Xi-\rev{r}ow, FROZEN, and D-DNGT, communication requirement decreases from $2p+n$ to $p+n$. For settings where $p\gg n$, e.g., $n\approx 100$ nodes collectively solving an image/video processing problem with $p\approx 10^6$, this reduction is significant. \rev{The} reduction in storage requirement is even more significant, see Table~\ref{table:comparison}.}}
\sa{In the numerical tests, we empirically show that FRSD 
%achieves a better convergence rate compared to 
is competitive against the other state-of-the-art methods: \fin{Xi-row~\cite{xi2018linear}, D-DNGT~\cite{lu2020nesterov}, FROZEN and  $\mathcal{ABN}$~\cite{xin2019distributedNEST}, $\mathcal{AB}$~\cite{xin2018linear}, $\mathcal{AB}$m~\cite{xin2019distributed}, Push-DIGing~\cite{nedic2017achieving} and Push-Pull~\cite{pu2020push}.}}

\textit{Notation:}
In this paper, the bold letters denote vectors, \sa{e.g.,} {$\bm{x}\in\reals^p$, and $\left[ \bm{x} \right]_{j}$ denotes the {$j$-th} element of $\bm{x}$.} The vector $\bm{0}_{n}$ and $\bm{1}_{n}$ represent the n-dimensional vectors of all zeros and ones. The uppercase of letters {are reserved for matrices; given $X\in\reals^{n\times n}$, $\diag(X)\in\reals^{n\times n}$ denotes the diagonal matrix of which diagonal is equal to that of $X\in\reals^{n\times n}$. Moreover, given $\bm{v}\in\reals^n$, $\diag(\bm{v})$ is a diagonal matrix with its diagonal equal to $\bm{v}$.} 
{\sa{$I_{n}=[\bm{e}_1, \ldots \bm{e}_n ]_{i=1}^n$} denotes the $n \times n$ identity matrix, where $\bm{e}_i$ denotes the $i$-th unit vector.} {Throughout $\Vert\cdot \Vert$ denotes the Euclidean and the spectral norms depending on whether the argument is a vector or a matrix.} 
\section{ \sa{Design, Comparison and analysis of FRSD}}
%Our goal is to solve 
{Consider the consensus optimization problem~\eqref{eqz1} over a communication network which is represented as a \emph{directed} graph $\mathcal{G}=(\mathcal{V},\mathcal{E})$, where $\mathcal{V}\triangleq\lbrace 1,2,\ldots,n\rbrace$ is the set of nodes (agents), and $\mathcal{E}$ is the set of directed communication links between the nodes. Each node $i\in\cV$ has a private cost function $f_{i} : \mathbb{R}^{p}\rightarrow \mathbb{R}$,  only known to node $i$. Furthermore,
for each node $i\in\cV$, we define \sa{its
% its out-neighbors as the set of nodes receiving information from node $i$, i.e., $\cN_{i}^{out}=\lbrace j\vert (i,j)\in\mathcal{E} \rbrace \cup \lbrace i\rbrace$, 
% %as $N_{i}^{out}$ 
% and 
in-neighbors as the set of nodes that can send information to node $i$, i.e., $\cN_{i}^{in}\triangleq\lbrace j\in\cV:~ (j,i)\in\mathcal{E} \rbrace \cup \lbrace i\rbrace$. Since FRSD is a row-stochastic method, any node $i\in\cV$ does not need to know its out-neighbors, i.e., the set of nodes receiving information from node $i$, which makes FRSD suitable for broadcast communication systems.}} 

{Throughout the paper we make the following assumptions.}
\begin{assumption}\label{assu1}
{$\cG$} is directed and strongly connected.
\end{assumption}
\begin{assumption}\label{assu2}
{For every $i\in\cV$, the local function $f_{i}$ is $L$-smooth, i.e., %such that 
it is differentiable 
with a Lipschitz 
gradient:}
\begin{equation}\label{eqz2}
\Vert \nabla f_{i}({\bm{x}})-\nabla f_{i}({\bm{x}'})\Vert \leq L \Vert {\bm{x}} -{\bm{x}'}  \Vert,\quad{\forall~{\bm{x}},\bm{x}' \in \reals^{p}.}
\end{equation}
\end{assumption}
\begin{assumption}\label{assu3}
{For all $i\in\cV$,} $f_{i}$ is $\mu$-strongly convex, i.e.,  
\begin{equation}\label{eqz3}
 f_{i}({\bm{x}'}) \geq  f_{i}({\bm{x}}) + \nabla f_{i}({\bm{x}})^\top ({\bm{x}'}-{\bm{x}}) + \dfrac{\mu}{2} \parallel {\bm{x}'} -{\bm{x}}  \Vert^2
\end{equation}
for all ${\bm{x}},\bm{x}'\in\reals^{p}$. 
\end{assumption}
\begin{remark}
\sa{Under Assumption~\ref{assu3}, the optimal solution to \eqref{eqz1} is unique, denoted by $\bm{x}^*$.}
\end{remark}

\begin{defn}
\label{def:f}
Define
  \sa{$\bm{\mathrm{x}} \triangleq \left[ \bm{x}_1^\top,\ldots,\bm{x}_n\sa{^\top}\right]^\top \in \reals^{\sa{np}}$ and 
  $\bm{\mathrm{y}} \triangleq \left[ \bm{y}_1\sa{^\top},\ldots,\bm{y}_n\sa{^\top}\right]^\top \in \reals^{\sa{np}}$,}
where $\bm{x}_i$, $\bm{y}_i \in \reals^{p} $ are the local variables of agent-$i$ for $i\in\cV\triangleq\{1,\ldots,n\}$, and in an algorithmic framework, their values at iteration \sa{$k\geq 0$} are denoted by $\bm{x}_i(k)$ and  $\bm{y}_i(k)$ for $i\in\cV$.  Let $f:\reals^{\sa{np}}\to\reals$ be a function of local variables $\{\bm{x}_i\}_{i\in\cV}$ such that  $f(\bm{\mathrm{x}})\triangleq \sum_{i\in\cV} f_i(\bm{x}_i)$ for $\bm{\mathrm{x}}\in\reals^{\sa{np} }$ and \sa{$\grad f(\bm{\mathrm{x}}) \triangleq \left[ \grad f_1(\bm{x}_1)\sa{^\top},...,\grad f_n(\bm{x}_n)\sa{^\top}\right]^\top \in \reals^{\sa{np}}$,} where $\grad f_{i}(\bm{x}_i) \in \mathbb{R}^{p}$ denotes the gradient of $f_{i}$ at $ \bm{x}_i\in\reals^p$.
\end{defn}

We next propose our decentralized optimization algorithm FRSD to solve 
{the consensus optimization problem in~\eqref{eqz1}.}
\begin{algorithm}
\caption{FRSD} 
\label{algori1}
\textbf{Input:}  
{$\bm{x}_{i}(0)\in\reals^{p},$ $\forall~i\in\cV$,
$\alpha,\beta>0$ such that $\alpha \beta < 1$, %row-stochastic 
\sa{${\overline{R}}=[r_{ij}]\in\reals^{n\times n}$} satisfies \eqref{eqz5}.}
\begin{algorithmic}[1]
\STATE {\sa{$\bm{y}_{i}(0)\gets \mathbf{0}_p$},  $\bm{v}_{i}(0)\gets {\bm{e}_{i}\in\reals^n}$ for $i\in\cV$}
\FORALL{ $k= 0, 1,...$}
%\FORALL{{$i\in\cV$}}
\STATE \sa{Each $i\in\cN$ independently performs:}
\STATE \fin{$\bar{\bm{x}}_{i}(k)\gets\sum\limits_{{j\in\cN_i^{in}}}r_{ij}\bm{x}_{j}(k)$}
\IF{$k>0$}
\STATE \sa{$\bm{y}_{i}(k) \gets \bm{y}_{i}(k-1) + \beta\left( \bm{x}_{i}(k)-\fin{\bar{\bm{x}}_{i}(k)}\right)$} \label{eqz4c}
\ENDIF
\STATE {\small $\bm{x}_{i}(k+1) \gets %\sum\limits_{{j\in\cN_i^{in}}}r_{ij}\bm{x}_{j}(k) 
\fin{\bar{\bm{x}}_{i}(k)}- \alpha \left(   \dfrac{\nabla f_{i}(\bm{x}_{i}(k))}{\left[   {\bm{v}_{i}}(k)\right]  _{i}}+\bm{y}_{i}(k)\right)$}
\STATE $\bm{v}_{i}(k+1) \gets \sum\limits_{{j\in\cN_i^{in}}}r_{ij}\bm{v}_{j}(k)$\label{eqz4a}
%\ENDFOR  
\ENDFOR 
\end{algorithmic}
\end{algorithm}
\vspace*{-3mm}
% \begin{algorithm}
% \caption{FRSD} 
% 
% \textbf{Input:}  
% {$\bm{x}_{i}(0)\in\reals^{p}$ for $i\in\cV$,
% $\alpha,\beta>0$ such that $0<\alpha \beta < 1$, {row-stochastic $ \diab{\overline{R}}=[r_{ij}]\in\reals^{n\times n}$ as in \eqref{eqz5}.}}
% \begin{algorithmic}[1]
% \STATE {$y_{i}(0)\gets \mathbf{0}$,  $\bm{v}_{i}(0)\gets {\bm{e}_{i}\in\reals^n}$ for $i\in\cV$}
% \FORALL{ $k= 0, 1,...$}
% %\FORALL{{$i\in\cV$}}
% \STATE \sa{Each $i\in\cN$ independently performs:}\\
% {\footnotesize
% \begin{align}
% \bm{x}_{i}(k+1)& \gets \sum\limits_{{j\in\cV}}r_{ij}\bm{x}_{j}(k) - \alpha \left(   \dfrac{\nabla f_{i}(\bm{x}_{i}(k))}{\left[   {\bm{v}_{i}}(k)\right]  _{i}}+\bm{y}_{i}(k)\right) \label{eqz4b} \\
% \bm{y}_{i}(k+1) & \gets \bm{y}_{i}(k) + \beta\left( \bm{x}_{i}(k+1)-\sum\limits_{{j\in\cV}}r_{ij}\bm{x}_{j}(k+1)\right) \label{eqz4c}\\
% \bm{v}_{i}(k+1)& \gets \sum\limits_{{j\in\cV}}r_{ij}\bm{v}_{j}(k) \label{eqz4a}
% \end{align}}
% %\ENDFOR  
% \ENDFOR 
% \end{algorithmic}
% \end{algorithm}
\subsection{FRSD Algorithm}
\label{sec:FRSD}
%We now describe {in detail} the distributed algorithm FRSD to solve  \eqref{eqz1}. 
{Consider FRSD displayed in Algorithm~\ref{algori1}, at each iteration $k\geq 0$,} every agent $i\in\cV$ 
{updates} three variables $\bm{x}_{i}(k)$, $\bm{y}_{i}(k)$ $\in \reals^{p}$ and $\bm{v}_{i}(k) \in \mathbb{R}^{n}$, 
%{as described in the Algorithm \ref{algori1}, 
where $\alpha,\beta>0$ and $ \sa{\overline{R}=[r_{ij}]\in\reals^{n\times n}}$  are the parameters of the algorithm: $\alpha$ is the constant step-size and $\beta$ is a momentum parameter such that $\alpha \beta < 1 $, and $\sa{\overline{R}}$ is a row-stochastic {matrix such that}
\begin{align} \label{eqz5}
     r_{ij}=\left\{
                \begin{array}{ll}
              > 0, &\quad  j \in {\cN_{i}^{in},}\\
                  0,&  \quad \text{otherwise{;}}
                \end{array}
              \right.
    \quad \sum\limits_{{j\in \cV}} r_{ij}=1, \quad \forall~i{\in\cV}.
\end{align}
\begin{remark}
\label{rem:stationary1}
{Since $\cG$ is strongly connected and has finitely many nodes, the Markov chain corresponding to the transition probability matrix $\sa{\overline{R}}$ is irreducible and positive recurrent; moreover, since $\sa{\overline{R}}$ has a positive diagonal, it is also aperiodic; therefore, {there exists a stationary distribution $\bm{\pi}\in\reals^n$, i.e., \sa{$\bm{\pi}> 0$} and $\bm{1}_{n}^\top\bm{\pi}=1$ such that $\bm{\pi}^\top \sa{\overline{R}}=\bm{\pi}^\top$.}}
\end{remark}
\begin{defn}
\label{def:Vk}
Each node \sa{$i\in\cV$}, initialized with $\bm{v}_i(0)=\sa{\bm{e}_{i}}$,  generates a sequence $\{\bm{v}_i(k)\}_{k\geq 0}$ \fin{initialized from $\bm{v}_i(0)=\sa{\bm{e}_{i}}$,}  \sa{using the recursion in \texttt{line~\ref{eqz4a}} of the FRSD algorithm.} Define \sa{$R=\overline{R}\otimes I_p$}, \sa{$\overline{V}(k)\triangleq\left[\bm{v}_{1}(k),\ldots,\bm{v}_{n}(k)\right]^\top \in \mathbb{R}^{n\times n}$, i.e., $\bm{v}_i(k)$ is the $i$-th row of $\overline{V}(k)$,} \sa{set $V(k)\triangleq\overline{V}(k)\otimes I_p$} and $\widetilde{V}(k)\triangleq\diag(V(k))$. 
\end{defn}

{Given arbitrary $\bm{\mathrm{x}}(0)\in\reals^{\sa{np}}$, we initialize $\bm{\mathrm{y}}(0)\in\reals^{\sa{np}}$ such that $\bm{y}_i(0)=\sa{\bm{0}_{p}}$ for $i\in\cV$ and $\sa{\overline{V}(0)}=I_{n}$.
We present FRSD stated in Algorithm~\ref{algori1} in a compact form as follows:}
\begin{subequations}\label{eqqz6}
\begin{align}
\bm{\mathrm{x}}(k+1)& = R\bm{\mathrm{x}}(k) - \alpha \left( \bm{\mathrm{y}}(k) + \widetilde{V}^{-1}(k) \nabla {f}(\bm{\mathrm{x}}(k)) \right), \label{eqz6b} \\
%&\sa{+(1-\alpha\beta)(I-R)\bm{x}(k),}\\
\bm{\mathrm{y}}(k+1)& = \bm{\mathrm{y}}(k) + \beta\left( I_{n} - R \right) \bm{\mathrm{x}}(k+1), \label{eqz6c}\\
 \sa{\overline{V}}(k+1)& =  \sa{\overline{R}} \hspace{.1cm} \sa{\overline{V}}(k). \label{eqz6a}
\end{align}
\end{subequations}
\begin{remark}
\sa{In the numerical section \fin{(Section~\ref{sec:numerical})}, we also considered a corrected step variant of FRSD, called FRSD-CS, which is only different from FRSD in the step size choice, i.e., FRSD-CS is obtained by replacing \eqref{eqz6b} with
\begin{equation*}
    \bm{\mathrm{x}}(k+1) = R\bm{\mathrm{x}}(k) - \alpha\widetilde{V}(k) \left( \bm{\mathrm{y}}(k) +  \widetilde{V}\rev{^{-1}}(k) \nabla {f}(\bm{\mathrm{x}}(k)) \right).
\end{equation*}}%
\vspace*{-5mm}
%FRSD-CS compared to FRSD, has a fast convergence rate in practice. In this paper, we leave the proof of the convergence rate of FRSD-CS as a future work and we only apply FRSD-CS in the numerical experiments.}
\end{remark}
\sa{We empirically observed that FRSD parameter tuning is pretty stable: once the parameters are tuned, the performance of the algorithm is robust to slight changes to the problem parameters, e.g., changes in the graph topology (through adding/deleting an edge), in the number of agents (increase/decrease), in convexity modulus and Lipschitz constants. For instance, in the Huber-loss minimization and logistic regression problems we tested in the numerical section, fixing \fin{$\alpha,\beta>0$ such that} $\alpha\beta=0.05$ worked very well; thus, tuning hyper-parameters can be treated as one-dimensional search as in Xi-row, $\mathcal{AB}$, Push-DIGing, which are using a single parameter $\alpha>0$. Furthermore, if the problem in~\eqref{eqz1} will be solved repetitively with slightly changing data, then having an additional parameter might be helpful as it gives the practitioner an extra degree of freedom to optimize the performance given 
%robustness of parameters --
that \fin{one time parameter tuning would work fairly well under slight changes.}} %Thus, FRSD, having an additional parameter, is more flexible compared to other methods that can only tune a step size parameter, e.g., Xi-Row, $\mathcal{AB}$, Push-Pull.}

\sa{In the rest, we focus on establishing asymptotic convergence guarantees for FRSD; therefore, we skip providing a termination condition as designing a locally implementable stopping mechanism for \emph{decentralized optimization} algorithms is itself a complicated task, attracting recent research interest~\cite{xie2017stop,prakash2019distributed}.}
\subsection{Related Methods}
\label{sec:related_work}
Next, we discuss the existing distributed optimization methods 
for a directed graph $\cG$ satisfying Assumption~\ref{assu1}. In particular, \sa{Push-DIGing using column-stochastic weights, $\mathcal{AB}$, $\mathcal{AB}$m, $\mathcal{ABN}$ and Push-Pull using both row- and column-stochastic weights, and \rev{Xi-row}, FROZEN, D-DGNT using only row-stochastic weights are closely related to our FRSD method and are described below in detail.}
\subsubsection{Push-DIGing}
 \rev{The} Push-DIGing algorithm, proposed in \cite{nedic2017achieving}, achieves a linear convergence rate 
{for solving \eqref{eqz1} over directed graphs (possibly time-varying)} with a constant step-size under Assumptions~\ref{assu1}-\ref{assu3}. 
{Given $\cG$,} Push-DIGing %is performed with updating the 
updates four variables ${\bm{x}_{i}(k), \bm{y}_{i}(k), \bm{z}_{i}(k)\in\reals^p}$ and $ v_{i}(k)\in{\reals}$ for each agent $i\in\cV$  
as follows:
\begin{align*}
v_{i}(k+1)& = \sum\limits_{\sa{j\in\cN_i^{in}}} b_{ij}  v_{j}(k),\\
\bm{x}_{i}(k+1)& = \sum\limits_{\sa{j\in\cN_i^{in}}} b_{ij} \left( \bm{x}_{j}(k) - \alpha\; \bm{y}_{j}(k) \right), \\
\bm{z}_{i}(k+1)& =\bm{x}_{i}(k+1)/ v_{i}(k+1),\\
\bm{y}_{i}(k+1)& = \sum\limits_{\sa{j\in\cN_i^{in}}} b_{ij} \bm{y}_{j}(k) + \nabla f_{i}(\bm{z}_{i}(k+1))-\nabla f_{i}(\bm{z}_{i}(k)),
\end{align*}
where $\alpha > 0$ is the step size and $\sa{\overline{{B}}}=\left[ b_{ij}\right] \in \mathbb{R}^{n\times n}$ is {a} column-stochastic %weights
\sa{matrix} {compatible with $\cG$}. The Push-DIGing algorithm {is} initialized with {$v_i(0)=1$}, $\bm{y}_{i}(0)=\nabla f_{i}(\bm{z}_{i}(0))$ and {from an arbitrary $\bm{x}_{i}(0)$ for each $i\in\cV$}. 
{Since directed graphs are not balanced in general}, Push-DIGing 
{adopts} a push-sum strategy, %along with 
\sa{which utilizes column-stochastic weights, requiring  
each agent} to know its out-degree  
--this may not be practical within broadcast-based communication systems. \sa{Compared to using column-stochastic weights, adopting row-stochastic weights might be preferred in such a distributed environment 
%that
\rev{where} each agent only manages the weights on information 
pertaining \rev{to} its in-neighbors.}
\subsubsection{$\mathcal{AB}$/$\mathcal{AB}$m/Push-Pull}
In contrast to Push-DIGing, \sa{$\mathcal{AB}$~\cite{xin2018linear} and  $\mathcal{AB}$m~\cite{xin2019distributed} algorithms} 
could get away with the nonlinear update due to 
eigenvector estimation. 
\sa{The $\mathcal{AB}$ and  $\mathcal{AB}$m\footnote{\rev{To present $\mathcal{AB}$ and $\mathcal{AB}$m in a unified manner, we use an adapt-then-combine update for $\bm{y}_{i}$ in $\mathcal{AB}$m similar to $\mathcal{AB}$~\cite{xin2018linear}. In~\cite{xin2019distributed}, it is mentioned that the $\mathcal{AB}$m method also works with this update; but, the originally stated version of $\mathcal{AB}$m uses an combine-then-adapt update: ${\bm{y}_{i}}(k+1) = \sum_{j\in{\cN_i^{in}}} b_{ij} \bm{y}_{j}(k) + \nabla f_{j}(\bm{x}_{j}(k+1))-\nabla f_{j}(\bm{x}_{j}(k))$ --we also considered the original version for our numerical tests.}} methods use} both row-stochastic and column-stochastic weights simultaneously. %to stay feasible in directed graphs. 
{At each iteration $k\geq 0$, %these methods 
they update two variables $\bm{x}_{i}(k),\bm{y}_{i}(k)\in{\reals^p}$ 
for each agent $i\in\cV$:}
{
\begin{align*}
{\bm{x}_{i}}(k+1)& = \sum\limits_{j\in\sa{\cN_i^{in}}} r_{ij}  \bm{x}_{j}(k) - \alpha \bm{y}_{i}(k)\sa{+\beta \left(\bm{x}_{i}(k)-\bm{x}_{i}(k-1)\right),}  \\
{\bm{y}_{i}}(k+1)& = \sum\limits_{j\in\sa{\cN_i^{in}}} b_{ij} (\bm{y}_{j}(k) + \nabla f_{j}(\bm{x}_{j}(k+1))-\nabla f_{j}(\bm{x}_{j}(k))),
\end{align*}}%
where $\alpha > 0$ is the step-size, \sa{$\beta\geq 0$ is the momentum parameter}, $\sa{\overline{R}}=[r_{ij}]\in \mathbb{R}^{n\times n}$ {and $\sa{\overline{B}}=\left[ b_{ij}\right] \in \mathbb{R}^{n\times n}$ denote the row-stochastic and column-stochastic weights, respectively, compatible with $\cG$}. {For the  $\mathcal{AB}$ method. \sa{the momentum parameter} $\beta=0$, and the iterate sequence, initialized with an arbitrary $\bm{x}_{i}(0)$ and $\bm{y}_{i}(0)=\nabla f_{i}(\bm{x}_{i}(0))$ for each $i\in\cV$, %AB algorithm 
converges %linearly
\rev{with a linear rate}} to the optimal solution under Assumptions~\ref{assu1}-\ref{assu3}.  
{%There is a variant of the AB algorithm, 
\sa{%On the other hand, 
\fin{Setting $\beta>0$,}}  $\mathcal{AB}$m~\cite{xin2019distributed} combines the gradient tracking with a momentum term and can deal with} nonuniform step-sizes, \sa{i.e., each agent-$i$ can pick $\alpha_i$ and $\beta_i$.}
 
{Push-Pull, proposed in \cite{pu2020push},   
is related to  $\mathcal{AB}$, it is only different in its $\bm{x}_i(k+1)$ update:
\begin{align*}
\bm{x}_{i}(k+1) = \sum\limits_{j\in\sa{\cN_i^{in}}} r_{ij} \big(\bm{x}_{j}(k) - \alpha \sa{\bm{y}_{j}(k)} \big),
%y_{i}(k+1)& = \sum\limits_{j=1}^{n} b_{ij} \big(y_{j}(k) + \nabla f_{j}(x_{j}(k+1))-\nabla f_{j}(x_{j}(k))\big),
\end{align*}
while $\bm{y}_i(k+1)$ update is the same with  $\mathcal{AB}$. 
 $\mathcal{AB}$ approach is based on the Combine-And-Adapt based scheme; on the other hand, Push-Pull method 
can be considered as an Adapt-Then-Combine based approach --for more details see~\cite{sayed2014adaptation}.} 

\subsubsection{Xi-row}
{The method proposed in \cite{xi2018linear}, which we call it as Xi-row in this paper, can solve \eqref{eqz1} over directed networks with a linear convergence rate using a %fixed 
\fin{constant step-size uniform across the agents.} Similar to our FRSD method, it %also %with utilizing 
only employs row-stochastic weights. Each agent $i\in\cV$ updates three variables $\bm{x}_{i}(k),\bm{y}_{i}(k),\bm{v}_{i}(k)\in\reals^p$ 
as follows:}
\begin{align*}
\bm{x}_{i}(k+1)& = \sum\limits_{j\in\sa{\cN_i^{in}}} r_{ij}  \bm{x}_{j}(k) - \alpha \bm{y}_{i}(k), \\
\bm{v}_{i}(k+1)& = \sum\limits_{j\in\sa{\cN_i^{in}}} r_{ij}  \bm{v}_{j}(k),\\
\bm{y}_{i}(k+1)& = \sum\limits_{j\in\sa{\cN_i^{in}}} r_{ij} \bm{y}_{i}(k)+ \dfrac{ \nabla f_{i}(\bm{x}_{i}(k+1))}{[\bm{v}_{i}(k+1)]_{i}}-\dfrac{\nabla f_{i}(\bm{x}_{i}(k))}{[\bm{v}_{i}(k)]_{i}},
\end{align*}
where $\alpha >0$ is the step-size and $\sa{\overline{R}}=\left[ r_{ij}\right] \in \mathbb{R}^{n\times n}$ is \sa{a row-stochastic %weights
matrix} {compatible with $\cG$}. The Xi-row iterates are initialized with arbitrary $\bm{x}_{i}(0)$, $\bm{v}_{i}(0)=\bm{e}_{i}$ and $\bm{y}_{i}(0)=\nabla f_{i}(\bm{x}_{i}(0))$ for each $i\in\cV$. \sa{A variant of the Xi-row method, FROST~\cite{xin2019frost} extends Xi-row to handle nonuniform step-sizes}.
\subsubsection{ $\mathcal{ABN}$/FROZEN/D-DNGT}  
{$\mathcal{ABN}$ and FROZEN, proposed in \cite{xin2019distributedNEST}, extend $\mathcal{AB}$ and Xi-row to incorporate Nesterov’s momentum term. Similar to $\mathcal{AB}$, $\mathcal{ABN}$ uses both row-stochastic and column-stochastic weights, %On the other hand, 
\fin{while FROZEN only requires row-stochastic weights.} At each iteration $k\geq 0$,  $\mathcal{ABN}$ updates three variables $\bm{x}_{i}(k),\bm{y}_{i}(k), \bm{s}_{i}(k)\in{\reals^p}$ for each agent $i\in\cV$:
\begin{subequations}
\begin{align}
\bm{s}_{i}(k+1)& = \sum\limits_{j\in\sa{\cN_i^{in}}} r_{ij}  \bm{x}_{j}(k) - \alpha \bm{y}_{i}(k) \label{eq:ABN-s}\\
\bm{x}_{i}(k+1)&=\bm{s}_{i}(k+1)+\beta \left(\bm{s}_{i}(k+1)-\bm{s}_{i}(k)\right), \label{eq:ABN-x}\\
\bm{y}_{i}(k+1)& = \sum\limits_{j\in\sa{\cN_i^{in}}} b_{ij} \bm{y}_{j}(k) + \nabla f_{j}(\bm{x}_{j}(k+1))-\nabla f_{j}(\bm{x}_{j}(k)),\nonumber
\end{align}
\end{subequations}
\fin{and} FROZEN updates %four variables, 
\sa{ $\bm{x}_{i}(k),\bm{y}_{i}(k),\bm{s}_{i}(k)\in{\reals^p}$, and $\bm{v}_{i}(k)\in\reals^n$, such that for each agent $i\in\cV$, $\bm{s}_{i}(k+1)$ and $\bm{x}_{i}(k+1)$ are updated according to \eqref{eq:ABN-s} and \eqref{eq:ABN-x}, respectively, \fin{$\bm{v}_{i}(k+1 )$ and $\bm{y}_{i}(k+1)$ are} updated according to}
\begin{align*}
\bm{v}_{i}(k+1)& = \sum\limits_{j\in\sa{\cN_i^{in}}} r_{ij}  \bm{v}_{j}(k),\\
%\bm{s}_{i}(k+1)& = \sum\limits_{j\in\cV} r_{ij}  \bm{x}_{j}(k) - \alpha \bm{y}_{i}(k)\\
%\bm{x}_{i}(k+1)&=\bm{s}_{i}(k)+\beta \left(\bm{s}_{i}(k+1)-\bm{s}_{i}(k)\right),  \\
\bm{y}_{i}(k+1)& = \sum\limits_{j\in\sa{\cN_i^{in}}} r_{ij} \bm{y}_{i}(k)+ \dfrac{ \nabla f_{i}(\bm{x}_{i}(k+1))}{[\bm{v}_{i}(k+1)]_{i}}-\dfrac{\nabla f_{i}(\bm{x}_{i}(k))}{[\bm{v}_{i}(k)]_{i}},
\end{align*}
where $\alpha > 0$ is the step-size, $\beta\geq 0$ is the momentum parameter in both methods, $\overline{R}=\left[ r_{ij}\right]\in \mathbb{R}^{n\times n}$ and $\overline{B}=\left[ b_{ij}\right] \in \mathbb{R}^{n\times n}$ denote row-stochastic and column-stochastic weight matrices. For both methods, $\bm{x}_{i}(0)$ and $\bm{s}_{i}(0)$ are arbitrary, and the other variables are initialized as $\bm{v}_{i}(0)=\bm{e}_{i}$, $\bm{y}_{i}(0)=\nabla f_{i}(\bm{x}_{i}(0))$ for each $i\in\cV$.} 

\sa{Another momentum-based method is D-DNGT, proposed in  \cite{lu2020nesterov}. D-DNGT is related to both FROZEN and $\mathcal{AB}$m:
\rev{
\begin{align*}
\bm{s}_{i}(k+1)& = \sum\limits_{j\in{\cN_i^{in}}} r_{ij}  \bm{x}_{j}(k)+\beta \left(\bm{s}_{i}(k)-\bm{s}_{i}(k-1)\right) - \alpha \bm{y}_{i}(k)\\
\bm{x}_{i}(k+1)&=\bm{s}_{i}(k+1)+\beta \left(\bm{s}_{i}(k+1)-\bm{s}_{i}(k)\right),  \\
\bm{v}_{i}(k+1)& = \sum\limits_{j\in{\cN_i^{in}}} r_{ij}  \bm{v}_{j}(k),\\
\bm{y}_{i}(k+1)& = \sum\limits_{j\in{\cN_i^{in}}} r_{ij} \bm{y}_{i}(k)+ \dfrac{ \nabla f_{i}(\bm{x}_{i}(k+1))}{[\bm{v}_{i}(k+1)]_{i}}-\dfrac{\nabla f_{i}(\bm{x}_{i}(k))}{[\bm{v}_{i}(k)]_{i}},
\end{align*}}%
where the initial variables are \rev{set as in}
%taken similar to 
the FROZEN method.}%
% {All the methods we have reviewed that use a uniform step-size also employ column-stochastic weights, except for FRSD,  Xi-row, \diab{FROZEN and D-DNGT} which only use row-stochastic weights. Therefore, FRSD, Xi-row, \diab{FROZEN and D-DNGT} are the method of choice for the broadcast-based distributed computational setting. On the other hand, comparing FRSD and Xi-row, FRSD has additional momentum parameter $\beta>0$; thus, it is natural to expect that it can be tuned to converge faster than Xi-row --indeed, we observed this expected behavior empirically in our numerical experiments -- see Section~\ref{sec:numerical}.}
\subsubsection{Comparison of different dynamics} \sa{Relaxing the assumption that all nodes need to know their out-degree (a requirement for column-stochastic methods) comes at a cost. Indeed, \fin{to be able to implement FRSD or any other row-stochastic method, e.g., \rev{Xi-row}, FROZEN, D-DNGT, one needs} each agent $i\in\cV$ to know the total number of agents in the network as well as its own rank in order to construct $\bm{v}_i\in\reals^n$. Finally, while push-sum based methods only require an extra scalar to be stored for ``debiasing," row-stochastic methods, including FRSD, require storing an $n$-dimensional vector, which grows linearly with the number of agents in the network.}

\sa{Next, to get a better insight about \rev{the} FRSD update rule, we write $\bm{\mathrm{x}}(k+2)$ in a recursive manner for FRSD and compare it against $\mathcal{AB}$ and \sa{$\mathcal{AB}$m} methods, which use both row- and column-stochastic matrices, and also with Xi-row which uses row-stochastic weights.}\\[3mm]
\textbf{ $\mathcal{AB}$/\sa{ $\mathcal{AB}${\rm m}}:} {For $k\geq 0$,}
{
\begin{flalign*}
\bm{\mathrm{x}}(k+2)= &(R+B)\bm{\mathrm{x}}(k+1) -{BR {\bf x}(k)} - \alpha B {\big(\nabla f(\bm{\mathrm{x}}(k+1))-\nabla f(\bm{\mathrm{x}}(k))\big)}\\
 &\sa{+\beta\big(\bx(k+1)-\bx(k)\big)-\beta B\big(\bx(k)-\bx(k-1)\big).}
 \vspace*{-3mm}
\end{flalign*}}%
\sa{$\mathcal{AB}$/\sa{ $\mathcal{AB}$m} recursion using column-stochastic weights for the gradient tracking term, $B\big(\nabla f(\bm{\mathrm{x}}(k+1))-\nabla f(\bm{\mathrm{x}}(k))\big)$, is fundamentally different %than
\fin{from} the row-stochastic methods.}\\[0.5mm]

\noindent\textbf{Xi-row:} {For $k\geq 0$,}%
{
\begin{align*}
\bm{\mathrm{x}}(k+2)= &2R\bm{\mathrm{x}}(k+1)-R^{2} \bm{\mathrm{x}}(k)\\
 &- \alpha  \big(\widetilde{V}^{-1}(k+1)\nabla f(\bm{\mathrm{x}}(k+1))-\widetilde{V}^{-1}(k)\nabla f(\bm{\mathrm{x}}(k))\big),
\end{align*}}%
{where for $k\geq 0$, $\widetilde{V}(k)\triangleq\diag(V(k))$ and $V(k)\triangleq\left[\bm{v}_{1}(k),...,\bm{v}_{n}(k)\right]^\top \in \mathbb{R}^{n\times n}$ --see Definition~\ref{def:Vk}. \sa{Except for the difference in how gradient tracking is handled, $\mathcal{AB}$ and Xi-row are closely related in terms of consensus dynamics, i.e., say $R=B=W$ for some doubly-stochastic mixing matrix $W$ compatible with $\cG$, then both $\mathcal{AB}$ ($\beta=0$) and Xi-row updates take the same form: $\bx(k+2)=2W\bx(k)-W^2\bx(k)+G(k)$, where $G(k)$ is the term related to gradient tracking. In contrast, FRSD has different consensus dynamics.}}\\[3mm]
\textbf{FRSD:} {For $k\geq 0$,}%
{
\begin{align*}
\bm{\mathrm{x}}(k+2)=&\big((1+\alpha \beta)R +(1-\alpha \beta) I_{n\sa{p}}\big)\bm{\mathrm{x}}(k+1)-R \bm{\mathrm{x}}(k) \\
&- \alpha   \big(\widetilde{V}^{-1}(k+1)\nabla f(\bm{\mathrm{x}}(k+1))-\widetilde{V}^{-1}(k)\nabla f(\bm{\mathrm{x}}(k))\big).
\end{align*}}%
\sa{For FRSD, we can set $\beta=c/\alpha$ for any $c\in(0,1)$, and using this choice, FRSD updates reduces to}% 
\sa{
\begin{align*}
\bm{\mathrm{x}}(k+2)= &2R\bm{\mathrm{x}}(k+1)-R^2 \bm{\mathrm{x}}(k)+\Delta_c(k) \\
&-\alpha  \big(\widetilde{V}^{-1}(k+1)\nabla f(\bm{\mathrm{x}}(k+1))-\widetilde{V}^{-1}(k)\nabla f(\bm{\mathrm{x}}(k))\big),
\end{align*}}%
\sa{where $\Delta_c(k)\triangleq (1-c)(I-R)\bm{x}(k+1)-R(I-R)\bm{x}(k)$ is the difference term between FRSD and the other two recursion rules. \fin{Indeed,} for $R=W$ as above, FRSD recursion takes the form: $\bx(k+2)=2W\bx(k)-W^2\bx(k)+G(k)+\Delta_c(k)$, where $G(k)$ is the FRSD gradient tracking term same with Xi-row.}

\sa{Clearly, for arbitrary $R$ and $B$ \rev{that are} compatible with a non-trivial directed graph $\cG$, $\mathcal{AB}$, $\mathcal{AB}$m, Xi-row and FRSD \rev{are not the same, they generate distinct iterate sequences.}}
% {Note that for $R=B=I_{n\diab{p}}$, all of them generate the same sequence. That said, for arbitrary $R$ and $B$ compatible with a non-trivial directed graph $\cG$,  $\mathcal{AB}$, Xi-row and FRSD are all different. Compared to  $\mathcal{AB}$ and Xi-row, FRSD is more flexible as it has an additional momentum parameter $\beta>0$ in addition to the constant step size $\alpha>0$ like the others.} 
\subsubsection{Implicit Gradient Tracking}
\sa{It is important to emphasize that the gradient tracking component $$\tilde V^{-1}(k+1) \nabla f(\bx(k+1)) - \tilde V^{-1}(k) \nabla f(\bx(k))$$ indeed appears in the FRSD recursion \fin{when} written using only in $\bx$ variables. That is why we are able to obtain the linear convergence for the FRSD iterate sequence. However, it is also worth mentioning that the implementation of the FRSD algorithm in practice does not require gradient tracking for computations at the node level, unlike the other methods in the literature -- as all other methods \emph{explicitly} use the gradient tracking, e.g., Xi-row,  $\mathcal{AB}$, Push-Pull, Push-DIGing. Using an
\emph{implicit} gradient tracking mechanism, FRSD does not need to store the previous iterates for neither $\bm{x}_i$ nor $\bm{y}_i$ variables, i.e., each agent-$i$ needs to store only $\boldsymbol{x}_i(k)$, $\boldsymbol{y}_i(k-1)$ and $\boldsymbol{v}_i(k)$ to be able to update these iterates to $\boldsymbol{x}_i(k+1)$, $\boldsymbol{y}_i(k)$ and $\boldsymbol{v}_i(k+1)$; hence, to implement FRSD, agent-$i$ needs to store a $2p+n$-dimensional vector. Moreover, the novel $\by$-update (see \texttt{line~\ref{eqz4c}} of Algorithm~\ref{algori1}), leading to
%use of %the momentum 
\emph{implicit} gradient-tracking, also result in a significant reduction in communication overhead; indeed, in order to implement FRSD, the agent-$j$ needs to only broadcast $\bx_j(k)\in\reals^p$ and $\bm{v}_j(k)\in\reals^n$; 
%to its out neighbor $i\in \cN_j^{out}$; 
thus, $j\in\cV$ needs to only transmit $n+p$-dimensional vector 
%to its out neighbor $i\in\cN_j^{out}$ 
-- note that for small networks, i.e., when $n$ is small, this is a significant reduction compared to $2p+1$ required by both Push-DIGing and $\mathcal{AB}$/Push Pull. Furthermore, comparing FRSD with the other row-stochastic methods \rev{Xi-row}, FROZEN, and D-DNGT communication requirement decreases from $2p+n$ to $p+n$. Therefore, for solving high-dimensional problems over small-to-medium size networks, i.e., $p\gg n$, FRSD becomes the method of choice -- see Table~\ref{table:comparison}.}
\begin{table}
\centering
\scriptsize
\begin{tabular}{ | c || c | c | c |c | c |  }
 \hline
Methods & Variables & Memory & Comm. & Row S. & Col. S.\\
 \hline
  \hline
    $\mathcal{AB}$ & $\bm{x},\bm{x}^p,\bm{y}$ & $3p$  &  $2p$  &   \cmark  &   \cmark\\
  \hline 
  Push-Pull  & $\bm{x},\bm{x}^p,\bm{y}$ & $3p$   & $2p$ &   \cmark  &   \cmark\\
    \hline
 $\mathcal{ABN}$ & $\bm{x},\bm{x}^p,\bm{y},\bm{s}$ &  $4p$ &  $2p$  &   \cmark  &   \cmark\\
 \hline
  $\mathcal{AB}$m & $\bm{x},\bm{x}^p,\bm{y}$ & $3p$ &  $2p$  &   \cmark  &   \cmark \\
  \hline
 Push-Ding   & $\bm{x},\bm{x}^p,\bm{y}, v$ &  $3p+1$    &   $2p+1$   &   \xmark &  \cmark \\
  \hline
 \rev{Xi-row}     & $\bm{x},\bm{x}^p,\bm{y},\bm{v}$ &  $3p+n$    &  $2p+n$  &  \cmark   &    \xmark         \\
  \hline
 D-DNGT     & \rev{$\bm{x},\bm{x}^p,\bm{y},\bm{s}, \bm{s}^p,\bm{v}$} &  $\rev{5}p+n$    &  $2p+n$  &  \cmark   &    \xmark         \\
  \hline
FROZEN      & $\bm{x},\bm{x}^p,\bm{y},\bm{s},\bm{v}$ &  $4p+n$    &  $2p+n$  &  \cmark   &    \xmark         \\
  \hline
FRSD        & $\bm{x},\bm{y},\bm{v}$ &  $2p+n$    &  $p+n$  &  \cmark   &      \xmark       \\
 \hline
 FRSD-CS        & $\bm{x},\bm{y},\bm{v}$ &  $2p+n$    &  $p+n$  &  \cmark   &      \xmark       \\
 \hline
\end{tabular}
 \caption{\sa{Comparison of methods for directed graphs in terms of storage and communication requirements (\rev{The ``Variables" column lists the variables %with minimal amount of memory 
 stored at each node %to be able 
 to carry out the computation} -- $\bm{x}^p$ denotes the previous iterate), and whether %their usage of 
 \rev{they use} row- and/or column-stochastic mixing matrices.}}
 \vspace*{-2mm}
\label{table:comparison}
\end{table}%
\subsection{Primal-Dual Algorithm Motivation}
\label{sec:motivation}
\sa{In a similar spirit with the discussion in~\cite{xu2020accelerated}, we can argue that FRSD is closely related to the primal-dual algorithms for saddle point problems studied within the optimization literature. More precisely, consider an equivalent formulation of the main problem in~\eqref{eqz1}: $$\min_{\{\bm{x}_i\}_{i\in\cV}}\{\sum_{i\in\cV}f_i(\bm{x}_i):\ \bx\triangleq[\bm{x}_i]_{i\in\cV}\in\cC\},$$ where $\cC\triangleq\{\bx:\ \bm{x}_1=\bm{x}_2=\ldots=\bm{x}_n\}$. Using the Fenchel duality, this problem can be written equivalently as
\begin{align}
\label{eq:SP_formulation}
    \min_{\bx}\max_{\by\in\cC^\perp}\sum_{i\in\cV}f_i(\bm{x}_i)+\bm{y}_i^\top\bm{x_i},
\end{align} where $\cC^\perp$ is the orthogonal complement of the subspace $\cC$, i.e., $\by\in\cC^\perp$ if and only if $\sum_{i\in\cV}\bm{y}_i=\mathbf{0}_p$. After swithcing the roles of $\bx$ and $\by$ through multiplying \eqref{eq:SP_formulation} with $-1$, if one naively implements a variant\footnote{The variant we discussed %below
\fin{here} is proposed in \cite{hamedani2021decentralized} %\url{https://arxiv.org/pdf/1908.11835.pdf} 
-- see Eq~(5) therein.} of the primal-dual algorithm proposed by Chambolle \& Pock~\cite{chambolle2016ergodic}, we get 
%the following updates:
{
\begin{subequations}
\label{eq:CP}
\begin{align}
    \label{eq:CP-x}
    \bm{m}_i(k)&\gets(1+\theta)\bm{y}_i(k)-\theta\bm{y}_i(k-1),\quad i\in\cV,\\
    \bm{x}_i(k+1)&\gets\bm{x}_i(k)-\alpha\Big(\grad f_i(\bm{x}_i(k))+\bm{m}_i(k)\Big),\quad i\in\cV,\\
    \by(k+1)&\gets\Pi_{\cC^\perp}\Big(\by(k)+\beta\bx(k+1)\Big),\label{eq:CP-projection}
\end{align}
\end{subequations}}%
where $\alpha,\beta>0$ are primal and dual step sizes, respectively, and $\theta\geq 0$ is the momentum parameter -- here, $\Pi_{\cC^\perp}(\cdot)$ denotes the Euclidean projection onto $\cC^\perp$, i.e., for simplicity of the notation, assume $p=1$; then, for any $\by$, $\Pi_{C^\perp}(\by)=\by-\mathbf{1}\mathbf{1}^\top\by/n$. 
%Since $\cC^\perp$ is a subspace
This explicit form of $\Pi_{\cC^\perp}(\cdot)$ and $\by(0)=\mathbf{0}_n$ initialization imply that \eqref{eq:CP-projection} is equivalent to $\by(k+1)\gets \by(k)+\Pi_{\cC^\perp}(\bx(k+1))$. This method is known to converge %linearly 
\rev{with a linear rate} for appropriately chosen $\alpha,\beta,\theta>0$; however, due to $\Pi_{\cC^\perp}(\cdot)$ in~\eqref{eq:CP-projection}, this algorithm is not %distributed
\fin{decentralized}.} 

\sa{To motivate FRSD through \rev{an} analogy, suppose the underlying network $\cG=(\cV,\cE)$ is undirected and we are given a doubly stochastic mixing matrix $W\in\reals^{n\times n}$ such that $W=W^\top$ and $W_{ij}>0$ if and only if $(i,j)\in\cE$. Let $W_\infty\triangleq\lim_{k\to\infty}W^k=\mathbf{1}_n\mathbf{1}_n^\top/n$. Note that $\Pi_{\cC^\perp}(\by)=(I-W_\infty)\by$. For decentralized implementation, consider approximating $\Pi_{\cC^\perp}(\cdot)$ with $(I-W)(\cdot)$. Thus, we will approximate the update in~\eqref{eq:CP-projection} with $\by(k+1)\gets \by(k)+(I-W)\bx(k+1)$, and with this approximation, the recursion in \eqref{eq:CP} takes the following form:
{
\begin{subequations}
\label{eq:CP-approx}
\begin{align}
    \by(k)&\gets\by(k-1)+\beta (I-W)\bx(k),\label{eq:CP-projection-approx}\\
    \bx(k+1)&\gets\bx(k)-\alpha\Big(\grad f(\bx(k))+\by(k)+\theta \beta(I-W)\bx(k)\Big),\label{eq:CP-approx-x}
\end{align}
\end{subequations}}%
where $\grad f(\bx)=[\grad f_1(\bm{x}_1)^\top, \ldots, \grad f_n(\bm{x}_n)^\top]^\top$.  %Note that
\fin{Hence,} adding and subtracting $W\bx(k)$ to \eqref{eq:CP-approx-x}, we get
{
\begin{align}
    \bx(k+1)\gets &W\bx(k)-\alpha\Big(\grad f(\bx(k))+\by(k)\Big)\nonumber\\
    &\mbox{ }+(1-\alpha\beta\theta)(I-W)\bx(k). \label{eq:CP-approx-x-2}
\end{align}}%
Therefore, given the primal, dual step sizes $\alpha,\beta>0$ such that $\alpha\beta<1$, setting the momentum parameter $\theta=(\alpha\beta)^{-1}>1$, the last term in~\eqref{eq:CP-approx-x-2} disappears as $1-\alpha\beta\theta=0$, the primal-dual algorithm with approximate averaging in \eqref{eq:CP-approx} reduces to
{
\begin{subequations}
\label{eq:CP-approx-FRSD}
\begin{align}
    \by(k)&\gets\by(k-1)+\beta \Big(\bx(k)-W\bx(k)\Big),\label{eq:CP-projection-approx}\\
    \bx(k+1)&\gets W\bx(k)-\alpha\Big(\grad f(\bx(k))+\by(k)\Big). \label{eq:CP-approx-x-FRSD}
\end{align}
\end{subequations}}%
It can be clearly seen that {the} FRSD algorithm for directed networks can be obtained from \eqref{eq:CP-approx-FRSD} by replacing \rev{the} doubly stochastic mixing matrix $W$ with a row stochastic $R$ and by ``debiasing" through introducing $\{\bv(k)\}_k\subset\reals^n$ sequence since $R_\infty=\lim_{k\to \infty}R^k=\mathbf{1}_n \bm{\pi}^\top$ for some $\bm{\pi}\in\reals^n$ such that $\bm{\pi}>\mathbf{0}_n$ and $\mathbf{1}_n^\top \bm{\pi}=1$.}
\section{Main Results}
In this section, we will show that {the iterate sequence generated by the algorithm FRSD as stated in \eqref{eqqz6} converges} to the optimal solution $\bm{\mathrm{x}}^{*}$ %linearly. 
\rev{with a linear rate. This result only applies to \fin{the FRSD method}, theoretical analysis of FRSD-CS is not considered in this paper.} 
%{Without loss of generality, we consider $p=1$;{hence, the local iterates 
%$\bm{x}_{i}(k),\bm{y}_{i}(k)\in\diab{\reals^p}$.} %one dimensional.} 
%\begin{remark}
%\label{rem:p1}
%{Since we assume $p=1$, $\bm{\mathrm{x}}=[\bm{x}_i]_{i=1}^n\in\reals^n$ and $f$ and $\grad f$ defined in Definition~\ref{def:f} become $f:\reals^n\to\reals$ and $\grad f:\reals^n\to\reals^n$ such that $f(\bm{\mathrm{x}})\triangleq\sum_{i=1}^nf_i(\bm{x}_i)$ and $\grad f(\bm{\mathrm{x}})\triangleq [\grad f_i(\bm{x}_i)]_{i=1}^n\in\reals^n$.}
%\end{remark}
\begin{remark}
\label{rem:muL-f}
{Assumptions~\ref{assu2} and~\ref{assu3} imply that $f$ is $L$-smooth, i.e., $\norm{\grad f(\bm{\mathrm{x}})-\grad f(\bm{\mathrm{x}}')}\leq L \norm{\bm{\mathrm{x}}-\bm{\mathrm{x}}'}$, and $\mu$-strongly convex.} 
%where $f= 1/n\sum_{i} f_{i}$ {as well.}
\end{remark}
\begin{remark}
\label{rem:stationary2}
{Since $\sa{\overline{R}}$ is row-stochastic, \sa{the} spectral radius of $\sa{\overline{R}}$ is 1, {$\rho(\sa{\overline{R}})=1$}; thus, $\lim_{k\to\infty}\sa{\overline{R}}^k$ exists.} {In particular, since $\sa{\overline{R}}$ corresponds to an ergodic Markov chain, 
we get $\lim_{k\rightarrow \infty} \sa{\overline{R}}^{k}=\bm{1}_{n} \bm{\pi}^\top$ -- see Remark~\ref{rem:stationary1}.} 
\end{remark}
\begin{defn}
\label{def:Vinf}
Define $V_{\infty}\triangleq\lim_{k\rightarrow \infty} V(k)$ and $\widetilde{V}_\infty\sa{\triangleq}\diag(V_{\infty})$. Furthermore, let $v\triangleq\sup\limits_{{k\geq 0}}\Vert V(k) \Vert$ and $\tilde{v}\triangleq\sup\limits_{{k\geq 0}}\Vert \widetilde{V}^{-1}(k) \Vert$.
\end{defn}
\begin{remark}
\label{rem:Vinf-bounds}
%$\widetilde{V}_\infty\sa{\triangleq}\lim_{k\to\infty}\widetilde{V}(k)$
\sa{Since $\sa{\overline{V}}(0)=I_{n}$, $\lim_{k\rightarrow \infty}{\overline{R}}^{k}=\bm{1}_{n} \bm{\pi}^\top$, we get $V_{\infty}=(\bm{1}_n\bm{\pi}^\top)\otimes I_p$ and $\widetilde{V}_\infty=\diag({V_{\infty}})=\sa{\bm{\pi}\otimes \bm{1}_{p}}$. Note $\{V(k)\}_k$ is convergent; hence, it is bounded, implying that $v\in\reals$ exists. Furthermore, Remark~\ref{rem:stationary1} shows that $\bm{\pi}>0$; therefore, $\tilde v\in\reals$ also exists.}
\end{remark}
\begin{remark}
\label{rem:Vinf}
{Since \sa{$\overline{R}$} corresponds to an Ergodic Markov chain, Remarks~\ref{rem:stationary1} and~\ref{rem:stationary2} imply that $V_\infty R=R V_\infty=V_\infty V_\infty=V_\infty$.} \sa{Moreover, the spectral radius \fin{of $R-V_\infty$} \rev{satisfies} $\rho(R-V_\infty)=\rho(\overline{R}-\bm{1}_n\bm{\pi}^\top)<1$.}
\end{remark}

Next, we define some auxiliary sequences that will be used %within 
\sa{in} the analysis. For $k\geq 0$, let $\hat{\bm{\mathrm{x}}}(k)\sa{\triangleq} V_{\infty}\bm{\mathrm{x}}(k)=\sa{(\bm{1}_{n}\otimes I_p)(\bm{\pi}^\top \otimes I_p) \bm{\mathrm{x}}(k)=(\bm{1}_{n}\otimes I_p)\hat{\bm{x}}(k)\in\reals^{np}}$, where ${\hat{\bm{x}}}(k)\sa{\triangleq}\sa{(\bm{\pi}^\top \otimes I_p)\bm{\mathrm{x}}(k){\in\reals^{{p}}}}$, \sa{i.e., $\hat{\bx}(k)=\bm{1}_n\otimes\hat{\bm{x}}(k)$.}
Let $\bm{\mathrm{x}}^{*}\sa{\triangleq\bm{1}_{n}\otimes}\bm{x}^{*}$ where
$\bm{x}^{*}\in\reals^{\sa{p}}$ is the unique optimal solution to \eqref{eqz1}. Thus, 
\sa{Definition~\ref{def:f}} implies that 
$\nabla f(\hat{\bm{\mathrm{x}}}(k))=\left[\nabla f_{1}(\hat{\bm{x}}(k))\sa{^\top},\ldots,\nabla f_{n}(\hat{\bm{x}}(k))\sa{^\top} \right]^\top{\in\reals^{{\sa{np}}}}$ and $\nabla f(\bm{\mathrm{x}}^*)=\left[\nabla f_{1}(\bm{x}^{*})\sa{^\top},\ldots,\nabla f_{n}(\bm{x}^{*})\sa{^\top} \right]^\top{\in\reals^{\sa{np}}}$.
\begin{remark}
\label{rem:optimality}
{From the optimality condition for \eqref{eqz1}, $\sa{(\bm{1}_{n}^\top\otimes I_p)}\grad f(\bm{\mathrm{x}}^*)=0$.}
\end{remark}

The structure of our proof was inspired by \cite{xi2018linear} and \cite{qu2017harnessing}.
In particular, we construct a linear system of inequalities {and use the deterministic version of the celebrated supermartingale convergence theorem~\cite{Robbins71} to prove the convergence results. We were able to show that FRSD iterates converge to the optimal consensus solution with a linear rate as in \cite{pu2020push, xin2018linear, lu2020nesterov}.} 

{In the rest of this section, we 
{establish the linear convergence; but, first, we state some preliminary results which will be used later.} 
\begin{defn}
\label{def:C}
Given $\alpha,\beta>0$ such that $\alpha\beta\in(0,1)$, let $\sa{C=\overline{C}\otimes I_p}$ and $\sa{\overline{C}}\triangleq (1-\alpha\beta)I_{n}+\alpha\beta \sa{\overline{R}}$, where $\sa{\overline{R}}=[r_{ij}]\in{\reals^{n\times n}}$ is the row-stochastic matrix as given in \eqref{eqz5}.
\end{defn}

\sa{The Markov chain associated with ${\overline{C}}$} is the lazy version of the Markov chain corresponding to $ \sa{\overline{R}}$; thus, it has the same stationary distribution, i.e., 
$\lim_{k\rightarrow \infty} \sa{\overline{C}}^{k}= \lim_{k\rightarrow \infty}  \sa{\overline{R}}^{k}=\bm{1}_{n} \bm{\pi}^\top$. Next, we state two technical results 
that will help us derive our main result.} 
\begin{lemma}\label{lemma1}
{Given 
$R$ and $C$ as defined above, there exist 
vector norms $\norm{\cdot}_R$, $\norm{\cdot}_C$ such that {$\norm{\cdot}\leq\norm{\cdot}_R$ and $\norm{\cdot}\leq\norm{\cdot}_C$}, and there exist constants $\sigma_R, \sigma_C\in (0,1)$ such that
\begin{align}
\Vert R\bm{\mathrm{x}} - \hat{ \bm{\mathrm{x}}} \Vert_{R} \leq \sigma_{R} \Vert \bm{\mathrm{x}} - \hat{ \bm{\mathrm{x}}} \Vert_{R} , \label{eqz7}\\
\Vert C\bm{\mathrm{x}} - \hat{ \bm{\mathrm{x}}} \Vert_{C}  \leq \sigma_{C} \Vert \bm{\mathrm{x}} - \hat{ \bm{\mathrm{x}}} \Vert_{C}, \label{eqz8}
\end{align}}%
for any $\bm{\mathrm{x}}{\in\reals^n}$ and $\hat{\bm{\mathrm{x}}}=V_\infty\bm{\mathrm{x}}$.
\end{lemma}
 %using the discussion in the paragraph above Lemma~\ref{lemma1}.

{Lemma~\ref{lemma1} directly follows from \eqref{eqz5} and Assumption~\ref{assu1} -- for the proof of \eqref{eqz7}, see  \cite[Lemma~2]{xi2018linear}, and \eqref{eqz8} can be shown similarly since $\lim_{k\rightarrow \infty}C^{k}= \lim_{k\rightarrow \infty} R^{k}= \sa{(\bm{1}_{n}\bm{\pi}^\top)\otimes I_p}$. Indeed, 
%one can argue that 
since $\rho(R-V_\infty)<1$ --see Remark~\ref{rem:Vinf},
%; thus, 
\cite[Lemma 5.6.10]{horn2012matrix} implies that there exists invertible \sa{$S\in\reals^{n{p}\times n{p}}$} such that $\norm{\bx}_R\triangleq\norm{S\bx}_1$; moreover, the matrix norm $\vertiii{\cdot}$ induced by $\norm{\cdot}_R$ satisfies $\vertiii{R-V_\infty}\in(0,1)$. Finally, through properly scaling $\norm{\cdot}_R$, we immediately get $\norm{\cdot}\leq\norm{\cdot}_R$, which does not affect $\vertiii{\cdot}$ since $\vertiii{B}=\max\{\norm{B\bx}_R/\norm{\bx}_R:\ \bx\neq {\bf 0}\}$ for any \sa{$B\in\reals^{n{p}\times n{p}}$}. Same arguments can be used for showing \eqref{eqz8} as  we also have $\rho(C-V_\infty)<1$.}

\begin{remark}
\label{rem:sigma_R}
{Let $\vertiii{\cdot}$ represent the matrix norm induced by $\norm{\cdot}_R$. According to \cite[Lemma 5.6.10]{horn2012matrix}, the constant $\sigma_R\in(0,1)$ in Lemma~\ref{lemma1} has an explicit form, 
%given by 
$\sigma_R=\vertiii{R-V_\infty}$.}
\end{remark}

First, we remark that all vector norms on a finite dimensional vector spaces are equivalent, i.e., there exist {$\kappa_1,\kappa_2,\kappa_3,\kappa_4>0$} such that
\begin{equation}
\label{eq:norm_equiv}
\begin{aligned}
\Vert \cdot \Vert_{R} & \leq {\kappa_1}\Vert \cdot \Vert_{C}, \quad&
\Vert \cdot \Vert_{C} & \leq {\kappa_2}\Vert \cdot \Vert_{R}, \\
\Vert \cdot \Vert_{R} & \leq {\kappa_3}\Vert \cdot \Vert, &
\Vert \cdot \Vert_{C} & \leq {\kappa_4}\Vert \cdot \Vert.
\end{aligned}
\end{equation}

{Similar to the results in~\cite{nedic2014distributed}, we also have $\Vert V(k)- V_{\infty}\Vert\leq \Lambda \lambda^k$ for some $0<\Lambda\in\reals$ and $\lambda\in(0,1)$. Below we analyze the dependence of $\lambda$ and $\Lambda$ on $R$.}
{
\begin{lemma}\label{lemma2}
Let $V(k)=R^k$ for $k\geq 0$ and $V_\infty=\lim_{k\to\infty}R^k$. 
Then, for $\kappa_3>0$ defined in~\eqref{eq:norm_equiv} and $\sigma_R\in(0,1)$ given in Remark~\ref{rem:sigma_R}, the following bound holds:
\begin{align}
\label{eq:Vk-Vinf-bound}
\sa{\Vert \overline{V}(k)- \bm{1}_n\bm{\pi}^\top\Vert=}\Vert V(k)- V_{\infty}\Vert  \leq \kappa_3 \sigma_R^{k},\quad \forall~k\geq 0.
\end{align}
\end{lemma}
\begin{proof}
It immediately follows from Remark~\ref{rem:Vinf} that
\begin{align*}
    \norm{V(k)- V_{\infty}}\leq\norm{(R-V_\infty)^k}\leq \kappa_3\vertiii{(R-V_\infty)^k}\leq \kappa_3\sigma_R^k,  
\end{align*}
holds for $k\geq 1$, where the second inequality follows from
\begin{align*}
    \norm{A}=\max_{\norm{\bm{v}}\leq 1}\norm{A\bm{v}}\leq\max_{\norm{\bm{v}}_R\leq \kappa_3}\norm{A\bm{v}}_R=\sa{\kappa_3}\vertiii{A},
    %\ \forall~A\in\reals^{n\diab{p}\times n\diab{p}};
\end{align*}
for \sa{all $A\in\reals^{n{p}\times n{p}}$} and the third inequality is due to $\vertiii{\cdot}$ being submultiplicative as it is an induced norm. \sa{Finally, the equality in \eqref{eq:Vk-Vinf-bound} follows from the fact that singular values of $\overline{V}(k)- \bm{1}_n\bm{\pi}^\top$ and $V(k)- V_{\infty}$ are the same.}\qed
\end{proof}}%
\begin{lemma}\label{lemma3}
The following inequalities hold for all $k\geq 0$:
\begin{subequations}
\begin{align}
&\Vert \widetilde{V}^{-1}(k)-\widetilde{V}^{-1}_{\infty}\Vert  \leq \widetilde{V}^{2} %\tilde{\Lambda}
\sa{\sqrt{n}}\kappa_3\sigma_R^{k} \label{eq:V-k-inf}\\
&\Vert \widetilde{V}^{-1}(k)-\widetilde{V}^{-1}(k-1)\Vert \leq 2 \tilde{v}^{2} %\tilde{\Lambda}
\sa{\sqrt{n}}\kappa_3\sigma_R^{k}. \label{eq:V-k-k}
\end{align}
\end{subequations}
\end{lemma}
\begin{proof}
\sa{The proof %Lemma~\ref{lemma3} 
follows from \cite[Lemma~3]{xi2018linear}. Indeed, note that $\widetilde{V}^{-1}(k)-\widetilde{V}^{-1}_{\infty}=\widetilde{V}^{-1}(k) (\widetilde{V}(k)-\widetilde{V}_{\infty}) \widetilde{V}^{-1}_{\infty}$; hence, $\norm{\widetilde{V}^{-1}(k)-\widetilde{V}^{-1}_{\infty}}\leq \tilde{v}^2 \norm{\diag\left({V}(k)-{V}_{\infty}\right)}\leq \tilde{v}^2 \sqrt{n} \Vert \bar{V}(k)- \bm{1}_n\bm{\pi}^\top\Vert$, where we used $\norm{A}_F\leq \sa{\sqrt{n}}\norm{A}_2$ for any $A\in\reals^{n\times n}$. Thus, the result follows from Lemma~\ref{lemma2}.} \qed
\end{proof}
\begin{lemma}\label{lemma4}
The following inequality holds for all \sa{$k\geq 0$}:
\begin{flalign*}
&\begin{aligned}
\text{(a)} \quad \Vert &   V_{\infty}\widetilde{V}^{-1}(k) \nabla f(\bm{\mathrm{x}}(k))\Vert
 \leq  v \tilde{v}^{2} {\sqrt{n}}\kappa_3{\sigma_R^{k}}  \Vert \nabla f(\bm{\mathrm{x}}(k)) \Vert 
 +{n L \Vert \bm{\mathrm{x}}(k)- \hat{\bm{\mathrm{x}}}(k)\Vert}_{C}  +nL \Vert \hat{\bm{\mathrm{x}}}(k)- {\bf x}^{*}\Vert\\
\text{(b)} \quad \Vert & V_{\infty}\widetilde{V}^{-1}(k-1) \nabla f(\bm{\mathrm{x}}(k)) \Vert
\leq  3v\tilde{v}^{2} {\sqrt{n}\kappa_3}{\sigma_R^{k}}  \Vert \nabla f(\bm{\mathrm{x}}(k)) \Vert 
+{nL \Vert \bm{\mathrm{x}}(k)- \hat{\bm{\mathrm{x}}}(k)\Vert{_{C}}  +nL \Vert \hat{\bm{\mathrm{x}}}(k)- {\bf x}^{*}\Vert}\\
\text{(c)} \quad \Vert & \hat{\bm{\mathrm{x}}}(k) - \hat{\bm{\mathrm{x}}}(k-1)  \Vert \leq \alpha v \tilde{v} L~\Vert \bm{\mathrm{x}}(k)-\bm{\mathrm{x}}(k-1) \Vert{_{R}} + \alpha 3v\tilde{v}^{2} {\sqrt{n}\kappa_3}{\sigma_R^{k}}  \Vert \nabla {f}(x(k)) \Vert\\  & \qquad\qquad\qquad\qquad+\alpha n L \Vert \bm{\mathrm{x}}(k)- \hat{\bm{\mathrm{x}}}(k)\Vert{_{C}}  +\alpha n L \Vert \hat{\bm{\mathrm{x}}}(k)- {\bf x}^{*}\Vert\\
\text{(d)} \quad \Vert & \widetilde{V}^{-1}(k)\nabla {f}(\bm{\mathrm{x}}(k)) -\widetilde{V}^{-1}(k-1)\nabla {f}(\bm{\mathrm{x}}(k-1) ) \Vert \leq \tilde{v} L \Vert \bm{\mathrm{x}}(k) - \bm{\mathrm{x}}(k-1)  \Vert{_{R}} + 2\tilde{v}^{2}{\sqrt{n}\kappa_3}{\sigma_R^{k}} \Vert \nabla {f}(\bm{\mathrm{x}}(k))  \Vert. \\
\end{aligned}&&
\end{flalign*}
\end{lemma}

\begin{proof}
First, we prove the part $(a)$.
\begin{align*}
&\quad \Vert V_{\infty}\widetilde{V}^{-1}(k) \nabla f(\bm{\mathrm{x}}(k))\Vert\\
\leq &\quad \Vert V_{\infty}\widetilde{V}^{-1}(k) \nabla f(\bm{\mathrm{x}}(k)) -V_{\infty}\widetilde{V}_{\infty}^{-1} \nabla f(\bm{\mathrm{x}}(k)) \Vert +   \Vert V_{\infty}\widetilde{V}_{\infty}^{-1} \nabla f(\bm{\mathrm{x}}(k))  \Vert\\
\leq &\quad\Vert V_{\infty}\Vert \Vert \widetilde{V}^{-1}(k) - \widetilde{V}_{\infty}^{-1} \Vert \Vert \nabla f(\bm{\mathrm{x}}(k)) \Vert + \Vert V_{\infty}\widetilde{V}_{\infty}^{-1} \nabla f(\bm{\mathrm{x}}(k)) - \sa{(\bm{1}_n \otimes I_p)(\bm{1}_n^\top \otimes I_p)} \nabla f(\bm{\mathrm{x}}^{*})\Vert\\
\leq &\quad {v \tilde{v}^{2} {\sqrt{n}}\kappa_3{\sigma_R^{k}}  \Vert \nabla f(\bm{\mathrm{x}}(k)) \Vert +n L \Vert \bm{\mathrm{x}}(k)- {\bf x}^{*}\Vert,}
% &\leq v \tilde{v}^{2} \tilde{\Lambda}\lambda^{k}  \Vert \nabla \bm{f}(\bm{\mathrm{x}}(k)) \Vert +n L \Vert \bm{\mathrm{x}}(k)- \hat{\bm{\mathrm{x}}}(k)\Vert \\
% &\qquad\qquad+nL \Vert \hat{\bm{\mathrm{x}}}(k)- \bm {x}^{*}\Vert
\end{align*}
{which together with triangular inequality implies $(a)$, {where \sa{$\widetilde{V}_\infty=\diag(V_\infty)$ %is defined in
--see Definition~\ref{def:Vinf}}}. In the second inequality, we use Remark~\ref{rem:optimality}, and the third inequality follows from \eqref{eq:V-k-inf} in Lemma~\ref{lemma3} and we also use Remark~\ref{rem:muL-f} along with $V_{\infty}\widetilde{V}_{\infty}^{-1}=  \sa{(\bm{1}_n \otimes \bm{I_p})(\bm{1}_n^\top \otimes \bm{I_p})=(\bm{1}_n\bm{1}_n^\top)\otimes I_p}$; hence, \sa{$\norm{(\bm{1}_n\bm{1}_n^\top)\otimes I_p}=n$}.} Next, we prove part $(b)$:
\begin{align*}
&\quad\Vert V_{\infty}\widetilde{V}^{-1}(k-1) \nabla {f}(\bm{\mathrm{x}}(k))\Vert\\
\leq &\quad \Vert V_{\infty}\widetilde{V}^{-1}(k-1) \nabla {f}(\bm{\mathrm{x}}(k)) -V_{\infty}\widetilde{V}^{-1}(k)  \nabla {f}(\bm{\mathrm{x}}(k)) \Vert +   \Vert V_{\infty}\widetilde{V}^{-1}(k) \nabla {f}(\bm{\mathrm{x}}(k)) \Vert\\
\leq &\quad \Vert V_{\infty}\Vert \Vert \widetilde{V}^{-1}(k) - \widetilde{V}^{-1}(k-1) \Vert
\Vert \nabla {f}(\bm{\mathrm{x}}(k)) \Vert + \Vert V_{\infty}\widetilde{V}^{-1}(k) \nabla {f}(\bm{\mathrm{x}}(k)) \Vert;
% &\leq 3v\tilde{v}^{2} \tilde{\Lambda}\lambda^{k}  \Vert \nabla {f}(\bm{\mathrm{x}}(k)) \Vert \\
% &\qquad \quad +k_{3}nL \Vert \bm{\mathrm{x}}(k)- \hat{\bm{\mathrm{x}}}(k)\Vert_{C} +nL \Vert \hat{\bm{\mathrm{x}}}(k)- \bm {x}^{*}\Vert.
\end{align*}
{hence, the part $(b)$ follows from \eqref{eq:V-k-k} in Lemma~\ref{lemma3} and from part $(a)$ of Lemma~\ref{lemma4}.}

{Now we consider part $(c)$. Since $\bm{\mathrm{y}}(0)={\bf 0}_{n{p}}$, it follows from \eqref{eqz6c} that $\bm{\mathrm{y}}(k)=\beta(\sa{I_{np}}-R)\sum_{\ell=1}^k\bm{\mathrm{x}}(\ell)$. %Note that $V_{\infty}R={\bf 1}_n\bm{\pi}^\top R= 1_n\bm{\pi}^\top=V_{\infty}$
{Since $V_{\infty}R=V_{\infty}$ -- see Remark~\ref{rem:Vinf},} we have $V_{\infty}\bm{\mathrm{y}}(k)=\sa{\mathbf{0}_{np}}$ for all $k\geq 0$ as $V_{\infty}(\sa{I_{np}}-R)=\sa{\mathbf{0}_{np\times np}}$. Hence, using $\hat{\bm{\mathrm{x}}}(k)=V_\infty\bm{\mathrm{x}}(k)$ for $k\geq 0$, when we multiply $V_{\infty}$ on both side of \eqref{eqz6b}, we get}
\begin{align*}
%V_{\infty}\bm{\mathrm{x}}(k)& =V_{\infty}\left(  R\bm{\mathrm{x}}(k-1) - \alpha \left( \bm{y}(k-1) + \widetilde{V}^{-1}(k-1) \nabla \bm{f}(\bm{\mathrm{x}}(k-1)) \right)\right)\\
%& = V_{\infty}\bm{\mathrm{x}}(k-1) - \alpha V_{\infty}\widetilde{V}^{-1}(k-1) \nabla \bm{f}(\bm{\mathrm{x}}(k-1)) \\
\hat{\bm{\mathrm{x}}}(k)& =\hat{\bm{\mathrm{x}}}(k-1)- \alpha V_{\infty}\widetilde{V}^{-1}(k-1) \nabla {f}(\bm{\mathrm{x}}(k-1)).
\end{align*}
{Therefore, the part $(c)$ immediately follows from using Remark~\ref{rem:muL-f} and the part $(b)$ of Lemma~\ref{lemma4} on}
\begin{align*}
&\quad \Vert \hat{\bm{\mathrm{x}}}(k) - \hat{\bm{\mathrm{x}}}(k-1)  \Vert \\
%&= \Vert \alpha V_{\infty} \widetilde{V}^{-1}(k-1)\nabla {f}(\bm{\mathrm{x}}(k-1))  \Vert \\
\leq &\quad {\alpha\Vert V_{\infty} \widetilde{V}^{-1}(k-1) \Big(\nabla {f}(\bm{\mathrm{x}}(k-1)) - \nabla {f}(\bm{\mathrm{x}}(k))\Big) \Vert}  + \alpha\Vert V_{\infty} \widetilde{V}^{-1}(k-1)\nabla {f}(\bm{\mathrm{x}}(k)) \Vert.
\end{align*}
Finally, we consider the part $(d)$.
\begin{align*}
&\quad \Vert \widetilde{V}^{-1}(k)\nabla {f}(\bm{\mathrm{x}}(k)) -\widetilde{V}^{-1}(k-1)\nabla {f}(\bm{\mathrm{x}}(k-1))  \Vert \nonumber \\
\leq &\quad \Vert  \widetilde{V}^{-1}(k-1)\nabla {f}(\bm{\mathrm{x}}(k)) -\widetilde{V}^{-1}(k-1)\nabla \bm{f}(\bm{\mathrm{x}}(k-1))  \Vert + \Vert  \widetilde{V}^{-1}(k)\nabla {f}(\bm{\mathrm{x}}(k)) -\widetilde{V}^{-1}(k-1)\nabla {f}(\bm{\mathrm{x}}(k)) \Vert\nonumber \\
\leq &\quad \Vert \widetilde{V}^{-1}(k-1) \Vert \Vert \nabla {f}(\bm{\mathrm{x}}(k)) -\nabla {f}(\bm{\mathrm{x}}(k-1))  \Vert +  \Vert \widetilde{V}^{-1}(k)- \widetilde{V}^{-1}(k-1) \Vert  \Vert \nabla {f}(\bm{\mathrm{x}}(k))  \Vert.
\end{align*}
{Hence, the part $(d)$ follows from \eqref{eq:V-k-k} of Lemma~\ref{lemma3} and Remark~\ref{rem:muL-f}.} 
\end{proof}

{For the sake of completeness we provide another technical result --for its proof, see \cite[Lemma 10]{qu2017harnessing}.}
\begin{lemma}\label{lemma5}
{%Under 
\sa{Given $\alpha \in(0,\frac{2}{nL})$, let $\eta \triangleq \max \left\lbrace  \vert 1-n L  \alpha \vert, \vert 1-n\mu  \alpha \vert \right\rbrace$. If} Assumptions~\ref{assu2} and~\ref{assu3} \sa{hold, then}  
for all ${\bm{x}}\in \mathbb{R}^p$, one has}
\begin{equation*}
{\Vert  \bm{x}-{\alpha \sum_{i\in\cV} \nabla {f}_i(\bm{x})} - \bm{x}^{*}\Vert \leq {\eta} \Vert \bm{x}-\bm{x}^{*}  \Vert}.
\end{equation*}
\end{lemma}

{Next, we will obtain bounds on $\Vert \bm{\mathrm{x}}(k+1)-\hat{\bm{\mathrm{x}}}(k+1)\Vert_C$, $\Vert \hat{\bm{\mathrm{x}}}(k+1) - {{\bf x}^*} \Vert $ and $\Vert \bm{\mathrm{x}}(k+1)-\bm{\mathrm{x}}(k) \Vert_R$. Combining these results will help us establish the linear rate for FRSD.}
\begin{lemma}\label{lemma6}
The following inequality holds for all {$k\geq 0$:}
\begin{align*}
\Vert  \bm{\mathrm{x}}&(k+1)- \hat{\bm{\mathrm{x}}}(k+1)\Vert{_{C}} \nonumber\\
 \leq &\quad ( \sigma_{C}+\alpha {{\kappa_{4}}}  n L) \Vert \bm{\mathrm{x}}(k) -\hat{\bm{\mathrm{x}}}(k)  \Vert{_{C}} + ( {\kappa_{2}\vertiii{R}}+  \alpha {\kappa_{4}} \tilde{v} L)\Vert\bm{\mathrm{x}}(k)-\bm{\mathrm{x}}(k-1)\Vert{_{R}}+ \alpha {\kappa_{4}} n L \Vert \hat{\bm{\mathrm{x}}}(k)- {{\bf x}^*}\Vert\\
&\quad +\alpha {\kappa_{4}}(2+v)  \tilde{v}^{2}{\sqrt{n}\kappa_3}{\sigma_R^{k}} \Vert \nabla {f}(\bm{\mathrm{x}}(k))  \Vert,
\end{align*}
{where $\vertiii{\cdot}$ denotes the induced matrix norm corresponding to the vector norm $\norm{\cdot}_R$.}
\end{lemma}
\begin{proof} {Using \eqref{eqz6b} twice, one for $\bm{\mathrm{x}}(k+1)$ and one for $\hat{\bm{\mathrm{x}}}(k+1)=V_\infty \bm{\mathrm{x}}(k+1)$, and using $V_\infty R=V_\infty$ together with $V_\infty\bm{\mathrm{y}}(k)=0$, we get the first equality below:}
\begin{align}
&\quad \Vert \bm{\mathrm{x}}(k+1)- \hat{\bm{\mathrm{x}}}(k+1)\Vert{_{C}} \nonumber\\
= &\quad \Vert R\bm{\mathrm{x}}(k)-\alpha \bm{\mathrm{y}}(k)-\alpha \widetilde{V}^{-1}(k)\nabla {f}(\bm{\mathrm{x}}(k))  {- \hat{\bm{\mathrm{x}}}(k)
%+ \alpha \bm{\hat{y}}(k)
+ \alpha V_{\infty} \widetilde{V}^{-1}(k)\nabla {f}(\bm{\mathrm{x}}(k)) \Vert}{_{C}} \nonumber\\
%&= \Vert R\bm{x}(k)-\alpha\left( y(k-1)+\beta (I_{n}-R)\bm{x}(k)\right) \nonumber \\
%&\qquad-\alpha \widetilde{V}^{-1}(k)\nabla \bm{f}(\bm{x}(k)) - \hat{\bm{x}}(k) \nonumber \\
%&\qquad + \alpha V_{\infty} \widetilde{V}^{-1}(k)\nabla \bm{f}(\bm{x}(k)) \Vert \nonumber\\
= &\quad \Vert R\bm{\mathrm{x}}(k)-R\bm{\mathrm{x}}(k-1)+\bm{\mathrm{x}}(k)- \hat{\bm{\mathrm{x}}}(k)-\alpha\beta (I_{n}-R)\bm{\mathrm{x}}(k) +\alpha \widetilde{V}^{-1}(k-1)\nabla {f}(\bm{\mathrm{x}}(k-1)) \nonumber \\
&\quad-\alpha \widetilde{V}^{-1}(k)\nabla {f}(\bm{\mathrm{x}}(k)) + \alpha V_{\infty} \widetilde{V}^{-1}(k)\nabla {f}(\bm{\mathrm{x}}(k)) \Vert{_{C}} \nonumber \\
\leq &\quad {\Vert \big((1-\alpha\beta)I_{n}+\alpha\beta R\big)\bm{\mathrm{x}}(k) - \hat{\bm{\mathrm{x}}}(k) \Vert}{_{C}} + {\kappa_{2}} \Vert R\bm{\mathrm{x}}(k)-R\bm{\mathrm{x}}(k-1)\Vert_{{R}} \nonumber \nonumber\\
&\quad+ \alpha {\kappa_{4}} \Vert \widetilde{V}^{-1}(k)\nabla {f}(\bm{\mathrm{x}}(k))-\widetilde{V}^{-1}(k-1)\nabla {f}(\bm{\mathrm{x}}(k-1))   \Vert + \alpha {\kappa_{4}}\Vert V_{\infty} \widetilde{V}^{-1}(k)\nabla {f}(\bm{\mathrm{x}}(k)) \Vert, \label{eqz11}
\end{align}
{where in the second equality we first use \eqref{eqz6c} to represent $\bm{\mathrm{y}}(k)$ in terms of $\bm{\mathrm{x}}(k)$ and $\bm{\mathrm{y}}(k-1)$, and next we use \eqref{eqz6b} to get rid of the term $-\alpha\bm{\mathrm{y}}(k-1)$.}

{Next, using \eqref{eqz8} of Lemma~\ref{assu1}, we can bound the first term on the right-hand-side of \eqref{eqz11} as follows:}
\begin{align*} %\label{eqz12}
\Vert \big((1-\alpha\beta)I_{n}+\alpha\beta R\big)\bm{\mathrm{x}}(k) - \hat{\bm{\mathrm{x}}}(k) \Vert{_{C}} &=\Vert C\bm{\mathrm{x}}(k) - \hat{\bm{\mathrm{x}}}(k) \Vert{_{C}} \leq  \sigma_{C} \Vert \bm{\mathrm{x}}(k) -\hat{\bm{\mathrm{x}}}(k)  \Vert{_{C}},
\end{align*}
{where $C$ is given in Definition~\ref{def:C}. Clearly, we can also bound the second term in \eqref{eqz11} with $\vertiii{R}\norm{\bm{\mathrm{x}}(k)-\bm{\mathrm{x}}(k-1)}_R$. Finally, using the parts (d) and (a) of Lemma~\ref{lemma4} for the third and the fourth terms, respectively, we get the desired result.} 
\end{proof}
\begin{remark}
{\fin{For FRSD the step size} $\alpha=\cO(\frac{1}{nL})$, \sa{i.e., there exists $n_0$ and $C>0$ such that $\alpha \leq C/(nL)$ for all $n\geq n_0$,}
% compares similarly 
\sa{\fin{this bound} is comparable to the step size bounds} used in other related works, e.g., the $\mathcal{AB}$, Push-DIGing, Xi-row methods.}
\end{remark}
\begin{lemma}\label{lemma7}
{When $0<\alpha<\dfrac{2}{nL}$, it holds that for $k\geq 0$:}
\begin{align*}
\Vert \hat{\bm{\mathrm{x}}}&(k+1) -\bm{\mathrm{x}}^{*}\Vert \nonumber\\
\leq &\quad \eta \Vert \hat{\bm{\mathrm{x}}}(k) -   \bm{\mathrm{x}}^{*} \Vert + \alpha n L {\Vert \bm{\mathrm{x}}(k) -\hat{\bm{\mathrm{x}}}(k)\Vert{_{C}}} +\alpha v\tilde{v}^{2}{\sqrt{n}\kappa_3}{\sigma_R^{k}}\Vert \nabla {f}(\bm{\mathrm{x}}(k)) \Vert.
\end{align*}
\begin{proof}
 Using \eqref{eqz6b} for $\hat{\bm{\mathrm{x}}}(k+1)=V_\infty \bm{\mathrm{x}}(k+1)$ together with $V_\infty R=V_\infty$ {and $V_\infty\bm{\mathrm{y}}(k)=0$, we get}
\begin{align}
&\quad \Vert \hat{\bm{\mathrm{x}}}(k+1) -  \bm{\mathrm{x}}^{*}\Vert\nonumber \\
= &\quad \Vert \hat{\bm{\mathrm{x}}}(k) -\alpha V_{\infty}\widetilde{V}^{-1}(k) \nabla {f}(\bm{\mathrm{x}}(k)) - \bm{\mathrm{x}}^{*} \Vert\nonumber\\
\leq &\quad \alpha\Vert \sa{\left((\bm{1}_n\bm{1}_n^\top)\otimes I_p\right)} \nabla {f}(\hat{\bx}(k))
 -  V_{\infty}\widetilde{V}^{-1}(k)  \nabla {f}(\bm{\mathrm{x}}(k))\Vert  +\Vert \hat{\bm{\mathrm{x}}}(k) -\alpha \sa{\left((\bm{1}_n\bm{1}_n^\top)\otimes I_p\right)} \nabla {f}(\hat{\bx}(k)) - \bm{\mathrm{x}}^{*} \Vert. \label{eqz12}
% \Vert \diab{ (\bm{1}_n \otimes {I_p})\Big(( \bm{\pi}^\top\otimes {I}_p) \bm{\mathrm{x}}(k) -\bm{x}^{*}  }\nonumber\\
% &\qquad \qquad \diab{-\alpha  \sum_{i=1}^n  \nabla {f_i}(( \bm{\pi}^\top\otimes \bm{I}_p) \bm{\mathrm{x}}(k))\Big)\Vert} \nonumber \\
 % + \alpha\Vert\diab{ (\bm{1}_n \otimes \bm{I_p}) \sum_{i=1}^n \nabla {f_i}(( \bm{\pi}^\top\otimes \bm{I}_p) \bm{\mathrm{x}}(k))}\nonumber \\
%  & \hspace{3cm}
\end{align}
%\sa{where we used $V_{\infty}=(\bm{1}_n\bm{\pi}^\top)\otimes I_p$ --see Remark~\ref{rem:Vinf-bounds}.}
\sa{Now, we bound the first term on the right-hand-side of \eqref{eqz12}:}
\begin{align}
&\quad \Vert \sa{\left((\bm{1}_n\bm{1}_n^\top)\otimes I_p\right)} \nabla {f}(\hat{\bx}(k))
%\diab{(\bm{1}_n \otimes \bm{I_p}) \sum_{i=1}^n \nabla {f_i}(( \bm{\pi}^\top\otimes \bm{I}_p) \bm{\mathrm{x}}(k)))} \nonumber\\ &\hspace{2cm}
- V_{\infty}\widetilde{V}^{-1}(k) \nabla {f}(\bm{\mathrm{x}}(k))   \Vert\nonumber\\
\leq &\quad \Vert \sa{\left((\bm{1}_n\bm{1}_n^\top)\otimes I_p\right)} \nabla {f}(\hat{\bx}(k))
%\diab{(\bm{1}_n \otimes \bm{I_p}) \sum_{i=1}^n \nabla {f_i}(( \bm{\pi}^\top\otimes \bm{I}_p) \bm{\mathrm{x}}(k)))}\nonumber \\ &\hspace{2cm}
- {V_{\infty}\widetilde{V}^{-1}_{\infty}} %(k)
\nabla {f}(\bm{\mathrm{x}}(k))   \Vert +\Vert V_{\infty}\widetilde{V}^{-1}(k) \nabla {f}(\bm{\mathrm{x}}(k))  -  {V_{\infty}\widetilde{V}^{-1}_{\infty}} %(k)
\nabla {f}(\bm{\mathrm{x}}(k)) \Vert\nonumber\\
\leq &\quad n L \Vert \bm{\mathrm{x}}(k) -\hat{\bm{\mathrm{x}}}(k)\Vert{_C} + v\tilde{v}^{2}{\sqrt{n}\kappa_3}{\sigma_R^{k}}\Vert \nabla {f}(\bm{\mathrm{x}}(k))\Vert, %\label{eqz14}
\end{align}
where we used \sa{${V_{\infty}\widetilde{V}^{-1}_{\infty}}= (\bm{1}_n\bm{1}_n^\top)\otimes I_p$},
%\diab{(\bm{1}_n \otimes \bm{I_p})(\bm{1}_n^\top \otimes \bm{I_p})}$,
Assumption~\ref{assu2} and Lemma~\ref{lemma3}.
Next, the second term on the right-hand-side of \eqref{eqz12} can be bounded using Lemma~\ref{lemma5}:
\begin{align}
\Vert \hat{\bm{\mathrm{x}}}(k) -\alpha \sa{\left((\bm{1}_n\bm{1}_n^\top)\otimes I_p\right)} \nabla {f}(\hat{\bx}(k)) - \bm{\mathrm{x}}^{*} \Vert =& \sa{\Vert \bm{1}_n \otimes \Big(\hat{\bm{x}}(k)-\bm{x}^{*} -\alpha  \sum_{i=1}^n  \nabla {f_i}(\hat{\bm{x}}(k))\Big)\Vert}\nonumber \\
\leq &\sa{\eta\Vert\bm{1}_n \otimes \Big(\hat{\bm{x}}(k)-\bm{x}^*\Big)\Vert} =\eta \Vert \hat{\bm{\mathrm{x}}}(k) -\bm{\mathrm{x}}^{*} \Vert, \label{eqz13}
\end{align}
where $\eta=\max \left\lbrace  \vert 1-nL  \alpha  \vert, \vert 1-n\mu  \alpha \vert \right\rbrace$,  \sa{$\bm{\mathrm{x}}^{*}=\bm{1}_{n}\otimes\bm{x}^{*},$ and $\hat{\bm{x}}(k)\triangleq (\bm{\pi}^\top\otimes I_p)\bx(k)$, i.e., $\hat{\bx}(k)=\bm{1}_n\otimes \hat{\bm{x}}(k)$.}
\sa{Finally, Lemma~\ref{lemma7} follows from \eqref{eqz12}-\eqref{eqz13}.} 
\end{proof}
\end{lemma}
\begin{lemma}\label{lemma8}
The following inequality holds for all {$k\geq 0$:}
\begin{align*}
\Vert \bm{\mathrm{x}}&(k+1) -\bm{\mathrm{x}}(k)\Vert{_{R}} \\
 \leq &\quad \big( \sigma_{R}+\alpha (1+v)\kappa_3\tilde{v} L\big) \Vert \bm{\mathrm{x}}(k+1) -\bm{\mathrm{x}}(k)  \Vert{_{R}}+  \alpha (\beta {\kappa_{1}  \vertiii{I_{n}-R}}+\kappa_3 nL)\Vert \bm{\mathrm{x}}(k) -\hat{\bm{\mathrm{x}}}(k)\Vert{_{C}}\\
&\quad+ \alpha \kappa_3 n L \Vert \hat{\bm{\mathrm{x}}}(k)- {{\bf x}^{*}}\Vert +{\kappa_3^2}  \alpha(3v+2)   \tilde{v}^{2}{\sqrt{n}}{\sigma_R^{k}}~\Vert \nabla {f}(\bm{\mathrm{x}}(k))  \Vert.
\end{align*}
% \begin{proof}
% See Appendix.
% \end{proof}
\end{lemma}
\begin{proof}
{We use \eqref{eqz6b} and \eqref{eqz6c} for rewriting $\bm{\mathrm{x}}(k+1)$ and $\bm{\mathrm{y}}(k)$ respectively, to derive the first two equations:}
\begin{align}
&\quad \Vert \bm{\mathrm{x}}(k+1) -\bm{\mathrm{x}}(k)\Vert{_{R}} \label{eqz15}\\
= &\quad\Vert R\bm{\mathrm{x}}(k)-\alpha \bm{\mathrm{y}}(k) -\alpha \widetilde{V}^{-1}(k) \nabla{f}(\bm{\mathrm{x}}(k)) -\bm{\mathrm{x}}(k)\Vert{_{R}} \nonumber\\
= &\quad \Vert R\bm{\mathrm{x}}(k)-\alpha\bm{\mathrm{y}}(k-1)-\alpha\beta(I_{n}-R)\bm{\mathrm{x}}(k)-\alpha \widetilde{V}^{-1}(k) \nabla{f}(\bm{\mathrm{x}}(k)) -\bm{\mathrm{x}}(k)\Vert{_{R}} \nonumber\\
%&= \Vert R\bm{x}(k)+\bm{x}(k)-R\bm{x}(k-1)\nonumber\\
%&\qquad\qquad\quad+\alpha \widetilde{V}^{-1}(k-1) \nabla \bm{f}(\bm{x}(k-1))-\alpha\beta(I_{n}-R)\bm{x}(k)\nonumber\\
%&\qquad\qquad\quad -\alpha \widetilde{V}^{-1}(k) \nabla\bm{f}(\bm{x}(k)) -\bm{x}(k)\Vert \nonumber\\
= &\quad \Vert R(\bm{\mathrm{x}}(k)-\bm{\mathrm{x}}(k-1))+\alpha \widetilde{V}^{-1}(k-1) \nabla {f}(\bm{\mathrm{x}}(k-1)) -\alpha\beta(I_{n}-R)\bm{\mathrm{x}}(k)\nonumber -\alpha \widetilde{V}^{-1}(k) \nabla{f}(\bm{\mathrm{x}}(k)) \Vert{_{R}} \nonumber\\
\leq &\quad \Vert R\bm{\mathrm{x}}(k)-R\bm{\mathrm{x}}(k-1) -\hat{\bm{\mathrm{x}}}(k)+\hat{\bm{\mathrm{x}}}(k-1) \Vert{_{R}}   + \kappa_3\Vert \hat{\bm{\mathrm{x}}}(k)-\hat{\bm{\mathrm{x}}}(k-1) \Vert+ \alpha\beta \Vert(I_{n}-R)\bm{\mathrm{x}}(k) \Vert{_R}  \nonumber\\
&\quad +\alpha \kappa_3\Vert \widetilde{V}^{-1}(k) \nabla {f}(\bm{\mathrm{x}}(k)) -\widetilde{V}^{-1}(k-1) \nabla{f}(\bm{\mathrm{x}}(k-1)) \Vert  \nonumber
\end{align}
where in the third equation, we use \eqref{eqz6b} {to get rid of the term $-\alpha\bm{\mathrm{y}}(k-1)$ as we did previously to derive \eqref{eqz11}}.
{We bound the first term above using Remark~\ref{rem:sigma_R}, i.e.,}
\begin{align}
&\quad\Vert R\bm{\mathrm{x}}(k)-R\bm{\mathrm{x}}(k-1) -\hat{\bm{\mathrm{x}}}(k)+\hat{\bm{\mathrm{x}}}(k-1) \Vert{_{R}}  \nonumber\\
 = &\quad\Vert (R-V_{\infty})(\bm{\mathrm{x}}(k)-\bm{\mathrm{x}}(k-1)) \Vert{_{R}}\leq \sigma_{R}\Vert \bm{\mathrm{x}}(k)-\bm{\mathrm{x}}(k-1)  \Vert{_{R}}.  \label{eqz16}
\end{align}
We can use Lemma~\ref{lemma4}~(c) to bound $\Vert \hat{\bm{\mathrm{x}}}(k)-\hat{\bm{\mathrm{x}}}(k-1)\Vert$,
and Lemma~\ref{lemma4}~(d) to bound the fourth term. {Then, the remaining third term in \eqref{eqz15} can be bounded as}
\begin{align*}
 \Vert(I_{n}-R)\bm{\mathrm{x}}(k) \Vert_{R}\quad
 &= \quad \Vert(I_{n}-R)(\bm{\mathrm{x}}(k)-\hat{\bm{\mathrm{x}}}(k)) \Vert_{R} \nonumber\\
&\leq \quad  \vertiii{I_{n}-R} \Vert\bm{\mathrm{x}}(k)-\hat{\bm{\mathrm{x}}}(k) \Vert_{R} \nonumber\\
&\leq \quad  \kappa_{1}  \vertiii{I_{n}-R} \Vert\bm{\mathrm{x}}(k)-\hat{\bm{\mathrm{x}}}(k)  \Vert_{C}, %\label{eqz18}
\end{align*}
where the fist equality {follows from $(I_{n}-R)V_{\infty}=\mathbf{0}$ due to $RV_{\infty}=V_{\infty}$; hence, we can add $(I_{n}-R)\hat{\bm{\mathrm{x}}}(k)$ to $(I_{n}-R)\bm{\mathrm{x}}(k)$. Combining all bounds {gives} the desired result.} 
\end{proof}

Combining the results of Lemmas~\ref{lemma6},~\ref{lemma7} and~\ref{lemma8}, we will construct a linear dynamical system prove the linear convergence of the proposed algorithm.
{For the sake of notational simplicity, we define some constants below:}
\sa{
\begin{align*}
s_{1} &{\triangleq} \kappa_{4}n L, &
s_{2} &{\triangleq} \kappa_{4}\tilde{v} L, &
s_{3} &{\triangleq} n L,\\
s_{4} &{\triangleq} \kappa_3 nL, &
%s_{5}& {\triangleq} \kappa_3 n L, &
s_{5} &{\triangleq} \kappa_3 (1+v)\tilde{v} L, &
s(\beta) &\triangleq \beta\kappa_1\vertiii{I_{n}-R},\\
\sa{p_{1}} &{\triangleq} {{\kappa_3}\kappa_{4}(2+v)\tilde{v}^{2}{\sqrt{n}}}, &
\sa{p_{2}} &{\triangleq} {\kappa_3}{v\tilde{v}^{2}{\sqrt{n}}}, &
\sa{p_{3}} &{\triangleq} {{\kappa_3^2} (3v+2)\tilde{v}^{2}{\sqrt{n}}}.
\end{align*}}%
{For $\alpha\in(0,\frac{2}{nL})$ and $\beta>0$ such that $\alpha\beta<1$, FRSD sequence $\{\bm{\mathrm{x}}(k)\}_{k\geq 0}$ satisfies the following system:}
\begin{align}\label{eqz19}
\theta_{k+1}\leq  \sa{\Upsilon_{\alpha,\beta}}~\theta_{k}+ \Phi_{k} \Psi_{k},\quad \forall~k\geq 0,
\end{align}
{where $\theta_k,\Phi_k$, $\Psi_{k}$ and \sa{$\Upsilon_{\alpha,\beta}\triangleq\Upsilon_1(\alpha,\beta)+\alpha\Upsilon_2$} are defined as}
{
\begin{align*}
&\theta_{k}=
\begin{bmatrix}
\Vert \bm{\mathrm{x}}(k) -\hat{\bm{\mathrm{x}}}(k)  \Vert_{C} \\
 \Vert \hat{\bm{\mathrm{x}}}(k)- \bm {x}^{*}\Vert  \\
\Vert\bm{\mathrm{x}}(k)-\bm{\mathrm{x}}(k-1)\Vert_{R}
\end{bmatrix},\qquad
\sa{\Phi_{k}= \sigma_R^{k}\alpha
\begin{bmatrix}
p_{1} & 0 & 0\\
p_{2} & 0 & 0\\
p_{3} & 0 & 0
\end{bmatrix},}\qquad
\Psi_{k}=
\begin{bmatrix}
\Vert  \nabla {f}(\bm{\mathrm{x}}(k)) \Vert \\
 0  \\
0
\end{bmatrix},\\
&\sa{\Upsilon_1(\alpha,\beta)=
\begin{bmatrix}
\sigma_{C}& 0 & \kappa_{2}\vertiii{R}\\
0 & 1 & 0 \\
\alpha s(\beta) & 0 & \sigma_{R}
\end{bmatrix},}\qquad 
\sa{\Upsilon_2=
\begin{bmatrix}
s_{1}& s_{1} & s_{2}\\
s_{3} & -n\mu & 0 \\
s_{4} & s_{4} & s_{5}
\end{bmatrix}.}\\
\end{align*}}%
\begin{theorem}\label{theorem1}
Suppose Assumptions~\ref{assu1}-\ref{assu3} holds. Let $\alpha,\beta>0$ such that 
{$\alpha\in(0,\bar{\alpha})$} and $\alpha\beta<1$, where \sa{$\bar{\alpha}>0$ is defined as}
{
\begin{align}%\label{eqz20}
&\bar{\alpha}\triangleq 
{\sup_{\delta_1,\delta_2}}\min\left\lbrace {\frac{(1-\sigma_{C})-\kappa_{2}\vertiii{R}\delta_{2}}{s_{1}(1+\delta_{1})+s_{2}\delta_{2}}},{\frac{(1-\sigma_{R})\delta_{2}}{s_{4}(1+\delta_{1})+s_{5}\delta_{2}\sa{+s(\beta)}}},\frac{1}{nL}\right\rbrace \nonumber\\
&\quad\hbox{s.t.}\qquad \dfrac{L}{\mu}<\delta_{1}, \qquad 0 <\delta_{2} < \dfrac{1-\sigma_{c}}{\kappa_{2}\vertiii{R}}. \label{eq:gamma_condition}
\end{align}}%
{Then, the spectral radius \sa{satisfies} $\rho(\sa{\Upsilon_{\alpha,\beta}})<1$.} % holds.}
\end{theorem}
\begin{proof}
{Given $\alpha\in(0,\frac{2}{nL})$ and $\beta>0$ such that $\alpha\beta<1$, it follows from Lemmas~\ref{lemma6}-\ref{lemma8} that \eqref{eqz19} holds for $k\geq 0$.} Next, we show $\rho(\Upsilon_{\alpha,\beta})<1$.
{Since $\Upsilon_{\alpha,\beta}$ has all non-negative entries, 
it is sufficient to} show that $\Upsilon_{\alpha,\beta}~ \gamma< \gamma$ {for some positive $\gamma=\left[\gamma_{1},\gamma_{2},\gamma_{3} \right]^{\top} \in \mathbb{R}^{3}_+$ --}see \cite[Corollary~8.1.29]{horn2012matrix}. Since {$L \geq \mu$}, {according to the definition of $\eta$ %given 
in Lemma~\ref{lemma5},} $\eta = 1-\alpha n \mu$ for $\alpha \in (0,\frac{1}{nL})$.
Hence, $\Upsilon_{\alpha,\beta}~\gamma< \gamma$ {is equivalent to}
\begin{subequations}\label{eqqz22}
\begin{align}
&(s_{1}\gamma_{1}+s_{1}\gamma_{2}+s_{2}\gamma_{3}) \alpha < \gamma_{1}(1-\sigma_{C}) - \kappa_{2}\vertiii{R}\gamma_{3}, \label{eqqz22a}\\
&s_{3}\gamma_{1}\alpha -\gamma_{2}n \mu \alpha <0, \label{eqqz22b}\\
&\big((s_{4}\sa{+s(\beta)})\gamma_{1}+\sa{s_{4}}\gamma_{2}+\sa{s_{5}}\gamma_{3}\big) \alpha < \gamma_{3}(1-\sigma_{R}). \label{eqqz22c}
\end{align}
\end{subequations}
{Clearly, \eqref{eqqz22} holds for all $\alpha\in(0,\bar{\alpha})$ and $\gamma\in\reals^3$ such that $\gamma_2=\delta_1\gamma_1$ and $\gamma_3=\delta_2\gamma_1$ for any $\gamma_1>0$ and $\delta_1,\delta_2>0$ satisfying \eqref{eq:gamma_condition}; thus, we get $\rho(\Upsilon_{\alpha,\beta})<1$.}\qed
\end{proof}
\begin{remark}
\label{rem:alpha-bound}
\sa{Note {$\delta_{1}$ and $\delta_{2}$} 
%free parameters that need to satisfy only 
are only required to satisfy \eqref{eq:gamma_condition}.}
{To provide a lower bound on an admissible $\alpha$, we compute a lower bound on $\bar{\alpha}$ by setting $\delta_2=\frac{1-\sigma_C}{2\kappa_2 \vertiii{R}}$ satisfying \eqref{eq:gamma_condition}. The supremum over $\delta_1$ subject to  \eqref{eq:gamma_condition} is achieved at $\delta_1=\frac{L}{\mu}$. For this particular choice we get $\bar{\alpha}<\frac{1}{nL}$ and $\bar{\alpha}\geq \min\{\alpha_1, \alpha_2\}$, where}
{
\begin{align*}
    \alpha_1\triangleq& \Big[\frac{2\kappa_4}{1-\sigma_C}(\frac{L}{\mu}+1)nL+\frac{\kappa_4}{\kappa_2}\tilde{v}L\Big]^{-1},\\
    \alpha_2\triangleq &(1-\sigma_R)\Big[\frac{\kappa_2\kappa_3\vertiii{R}}{1-\sigma_C}(\frac{L}{\mu}+1)n L+\frac{\kappa_1\kappa_2\vertiii{R}\vertiii{I-R}\beta}{1-\sigma_C}
   +\kappa_3(1+v)\tilde{v}L\Big]^{-1},
\end{align*}}%
{where we used $1=\rho(R)\leq \vertiii{R}$.}
\end{remark}

Finally, in the next theorem, we prove {that FRSD iterate sequence converges \rev{with a linear rate} %convergence of our algorithm 
through showing a linear decay for $\{\Phi_{k}\}$. First, we state a classic result that will be useful in our analysis; for its proof, see \cite{Robbins71,polyak1987introduction}.}
\begin{lemma}\label{lemma9}
Let $\lbrace a_{k}\rbrace$, $\lbrace b_{k}\rbrace$, $\lbrace c_{k}\rbrace$ and $\lbrace d_{k}\rbrace$ be non-negative sequences such that $\sum\limits_{k=0}^{\infty}c_{k}<\infty$, $\sum\limits_{k=0}^{\infty}d_{k}<\infty$, and
\begin{align*}
a_{k+1}\leq (1+c_{k})a_{k}-b_{k}+ d_{k},\quad \forall~k\geq 0.
\end{align*}
Then $\{a_{k}\}$ converges and $\sum\limits_{k=0}^{\infty}b_{k}< \infty$. 
\end{lemma}
\begin{theorem}\label{theorem2}
Let Assumptions~\ref{assu1}-\ref{assu3} hold. \sa{For any %sufficiently small 
step-size $\alpha\in(0, \bar{\alpha})$, {the sequence $\{\bm{\mathrm{x}}(k)\}$} %in \eqref{eqqz6}
%\sa{linearly converges} 
\rev{converges with a linear rate} to $\bx^{*}=\bm{1}_n\otimes\bm{x}^*$ with a rate arbitrarily close to $\varphi_{\alpha,\beta}\triangleq\max\{\rho(\Upsilon_{\alpha,\beta}), \sigma_R\}\in(0,1)$,} where $\bar{\alpha}$ is defined in the Theorem \ref{theorem1}, \sa{and $\sigma_R=\vertiii{\Big(\overline{R}-(\bm{1}_n\bm{\pi}^\top)\Big)\otimes I_p}<1$.}
\end{theorem}
\begin{proof}
\sa{The proof follows from similar arguments as in the proof of \cite[Lemma 5]{xi2018linear}. 
Theorem \ref{theorem1} shows that 
%the spectral radius 
$\rho(\Upsilon_{\alpha,\beta})<1$; hence, \rev{from \cite[Lemma 5.6.10]{horn2012matrix},} given any positive $\zeta<1-\varphi_{\alpha,\beta}$, %\max\{\rho(\Upsilon_{\alpha,\beta}),\sigma_R\}$, 
there exists a matrix norm $\norm{\cdot}_{(\zeta)}$ such that $\norm{\Upsilon_{\alpha,\beta}}_{(\zeta)}\leq \rho(\Upsilon_{\alpha,\beta})+\tfrac{\zeta}{2}$. Since norms are equivalent in finite-dimensional spaces, we conclude that there
exists some $\Gamma>0$ such that we have
%the linear convergence rate of $\Upsilon^{k}$ and $\Phi_{k}$, 
\begin{align}\label{eqz26}
\Vert \Upsilon_{\alpha,\beta}^{k} \Vert \leq \Gamma {\tilde\lambda^{k}},  \qquad  \Vert \Upsilon_{\alpha,\beta}^{{k-j-1}}{\Phi_{j}} \Vert \leq \Gamma  \tilde\lambda^{k},
\end{align}
for all $0\leq j\leq k-1$, where $\tilde{\lambda}\triangleq\max\{\sigma_R, \rho(\Upsilon_{\alpha,\beta})+\tfrac{\zeta}{2}\}<1$.}
By writing \eqref{eqz19} recursively, we get, for all $k\geq 0$,
\begin{align}
\label{eq:recursive}
\theta_{k} \leq \Upsilon_{\alpha,\beta}^{k} \theta_{0} + \sum\limits_{j=0}^{k-1} \Upsilon_{\alpha,\beta}^{k-j-1}\Phi_{j}\Psi_{j}.
\end{align}
{Since all the terms in~\eqref{eq:recursive} have non-negative entries, using \eqref{eqz26}, we get for all $k\geq 0$,} 
 \begin{align}
\Vert\theta_{k}\Vert \quad & \leq \quad\Vert\Upsilon_{\alpha,\beta}^{k}\Vert \Vert \theta_{0}\Vert + \sum\limits_{j=0}^{k-1} \Vert\Upsilon_{\alpha,\beta}^{k-j-1}\Phi_{j}\Vert \Vert\Psi_{j}\Vert
 \leq \quad \Gamma  \tilde\lambda^{k} \big(\Vert \theta_{0}\Vert + \sum\limits_{j=0}^{k-1} \Vert\Psi_{j}\Vert\big). \label{eqz28}
\end{align}
{For any $k\geq 0$, we can bound $\Vert\Psi_{k}\Vert$ as follows:}
\begin{align}
\Vert\Psi_{k}\Vert \quad & \leq \quad \Vert \nabla {f}(\bm{\mathrm{x}}(k)) -\nabla {f}(\bm{\mathrm{x}}^{*})\Vert + \Vert \nabla {f}(\bm{\mathrm{x}}^{*})\Vert
\leq \quad L\Vert \bm{\mathrm{x}}(k) -\hat{\bm{\mathrm{x}}}(k)\Vert +L\Vert \hat{\bm{\mathrm{x}}}(k) -\bm{\mathrm{x}}^{*}\Vert+ \Vert \nabla {f}(\bm{\mathrm{x}}^{*})\Vert\nonumber\\
&\leq \quad {2} L\Vert \theta_{k}\Vert  + \Vert \nabla {f}(\bm{\mathrm{x}}^{*})\Vert. \label{eqz29}
\end{align}
Thus, {for all $k\geq 0$, combining \eqref{eqz28} and \eqref{eqz29} we get}
\begin{align*}
\Vert\theta_{k}\Vert \leq \Biggl( \Vert \theta_{0}\Vert + {2} L \sum\limits_{j=0}^{k-1} \Vert \theta_{j}\Vert 
 + k\Vert \nabla {f}(\bm{\mathrm{x}}^{*})\Vert \Biggl) \Gamma \tilde\lambda^{k}. %\label{eqz30}
\end{align*}
{For $k\geq 0$, let $a_{k}\triangleq\sum\limits_{j=0}^{k-1} \Vert \theta_{j}\Vert $, $b_k\triangleq 0$, $\tilde{c}\triangleq{2}L\Gamma$, and $\tilde{d}_{k}\triangleq \Gamma \Vert \theta_{0}\Vert + k\Gamma \Vert \nabla {f}(\bm{\mathrm{x}}^{*})\Vert$; hence, we get}
%from \eqref{eqz30} we get
\begin{align}\label{eqz31}
{\Vert\theta_{k}\Vert= a_{k+1}-a_{k}\leq (\tilde{c} a_{k}+\tilde{d}_{k})\tilde{\lambda}^k,\quad \forall~k\geq 0.}
\end{align}
{Define $c_k\triangleq \tilde{c}\tilde\lambda^k\geq 0$ and $d_k\triangleq \tilde{d}_k\tilde\lambda^k\geq 0$ for $k\geq 0$. Since $\tilde{\lambda}\in(0,1)$, we have $\sum_{k=0}^{\infty} c_k+d_k< \infty$; therefore, Lemma \ref{lemma9} implies that $\{a_{k}\}$ converges. \sa{Furthermore, our choice of $\zeta>0$ and the definition of $\tilde\lambda$ imply that $\tilde\lambda+\tfrac{\zeta}{2}\in(0,1)$; therefore, since $\{a_{k}\}$ is bounded, \eqref{eqz31} leads to %implies that %for all $\xi\in(0,1-\tilde\lambda)$, we get
\begin{align}
\lim\limits_{k\rightarrow \infty}\frac{\Vert \theta_{k} \Vert}{(\tilde\lambda+\tfrac{\zeta}{2})^{k}}\leq \frac{(\tilde{c} a_{k}+\tilde{d}_{k})\tilde\lambda^{k}}{(\tilde\lambda+\tfrac{\zeta}{2})^{k}}=0.
\end{align}}
Thus, there exist \sa{$c>0$} such that 
\begin{align}
\Vert \theta_{k} \Vert \leq \sa{c}(\tilde\lambda+\tfrac{\zeta}{2})^{k},\quad  
\forall k\geq 0, 
\end{align}
Thus, we get the desired result since for all $k\geq 0$,}
\begin{align*}
\Vert \bm{\mathrm{x}}(k)-\bm{\mathrm{x}}^{*} \Vert \quad &\leq \quad\Vert \bm{\mathrm{x}}(k)-\hat{\bm{\mathrm{x}}}(k)\Vert+ \Vert\hat{\bm{\mathrm{x}}}(k)-\bm{\mathrm{x}}^{*} \Vert 
\leq \quad {2}  \Vert  \theta_{k} \Vert
\leq  {2} c(\tilde\lambda+\tfrac{\zeta}{2})^{k}
\sa{\leq   2c (\varphi_{\alpha,\beta}+\zeta)^k,} 
\end{align*}
\sa{and we clearly have $\varphi_{\alpha,\beta}+\zeta<1$.}\qed
\end{proof}
\section{Numerical Results}
\label{sec:numerical}
In this section, we {provide some numerical %experiments for 
results %verifying 
to demonstrate the performance of the proposed method against the state-of-the-art competitive algorithms designed for \fin{consensus optimization over} directed graphs}. \sa{In our experiments}, we considered two types of distributed regression problems of the form given in \eqref{eqz1} \fin{to test FRSD method: \textit{(a)} regression with Huber loss, %and the other is 
\textit{(b)} the logistic regression, which are} described in Sections~\ref{sec:Huber} and~\ref{sec:logistic}, respectively. {Throughout the experiments, we use the} uniform %weighting strategy
\fin{weights %to set up 
for the row-stochastic %weights
matrix defined} in \eqref{eqz5}, {i.e., $r_{ij}=1/|\cN_{i}^{in}|$ for all $i\in\cV$, \sa{and \rev{we %consider 
use coordinated step-size and momentum parameters} for $\mathcal{AB}$m, $\mathcal{ABN}$, FROZEN and D-DNGT to have a fair comparison with other methods.}}

\sa{For both distributed regression problems,} we compare FRSD with Xi-row \cite{xi2018linear}, \sa{FROZEN \cite{
xin2019distributedNEST} and D-DNGT \cite{lu2020nesterov},} which use only \emph{row-stochastic} weights similar to our method, with Push-DIGing~\cite{nedic2017achieving}, which utilizes \emph{column-stochastic} weights, {and also with}  $\mathcal{AB}$~\cite{xin2018linear}, \sa{$\mathcal{AB}$m \cite{xin2019distributed}, $\mathcal{ABN}$~\cite{
xin2019distributedNEST}} and Push-Pull~\cite{pu2020push}, which use both \emph{row-stochastic} and \emph{column-stochastic} weights \sa{over six {different} time-invariant directed graphs {with $n = 10 , 30$ , $50$ , $100$ and $200$ nodes (agents)}, see Figure~\ref{fig1}}.

\sa{Furthermore, in Section~\ref{sec:random}, we also conducted a numerical test over the random graphs comparing the proposed method with the other row-stochastic methods, i.e., \rev{Xi-row}, FROZEN and D-DNGT.}

\subsection{
{Distributed Regression with Huber Loss}}
\label{sec:Huber}
{Suppose $\tilde{\bm{x}}\in\mR^p$ is the \emph{unknown} linear model, and for each $i\in\cV$, let $\bm{b}_i\in\reals^{m_i}$ be the corresponding noisy measurement vector, i.e., $\bm{b}_i=M_i \tilde{\bm{x}}+\bm{n}_i$ where $\bm{n}_i\in\reals^{m_i}$ is the measurement noise vector. Given parameter $\xi>0$, the Huber loss function $H_{\xi}:\reals\to\reals_+$ is defined as
{
\begin{align*} 
  H_{\xi}(z)=\left\{
                \begin{array}{ll}
              \dfrac{1}{2} z^{2} , &\text{if} \quad\vert z \vert \leq \xi;\\ %\quad(L_{1}\quad \text{zone})\\
                  \xi\left( \vert z \vert - \dfrac{1}{2} \xi\right)  \quad &\text{otherwise}. %\quad(L_{2}^{2} \quad\text{zone})
                \end{array}
              \right.
\end{align*}}%
For any $m\in\integers_+$, we also define $\mathbf{H}_{\xi}:\reals^m\to\reals^m$ such that $\mathbf{H}_{\xi}(\bm{z})=[H_{\xi}(z_j)]_{j=1}^m$ where $\bm{z}=[z_j]_{j=1}^m\in\reals^m$.}

{In this experiment, the goal is to estimate $\tilde{\bm{x}}$ with an optimal solution $\bm{x}^*$ to the regression problem with Huber loss:
\begin{align}
\label{eq:huber-problem}
%x^*\in\argmin_{x\in\mathbb{R}^{p}}\sum_{i=1}^{n}\sum_{j=1}^{m}H_{\xi}(M_{i,j}x-y_{i,j}),
\bm{x}^*\in\argmin_{\bm{x}\in\mathbb{R}^{p}}\bar{f}(\bm{x})\triangleq\frac{1}{n}\sum_{i\in\cV}\mathbf{H}_{\xi}(M_i \bm{x}-\bm{b}_i).
\end{align}
\sa{In the experiments, we solve \eqref{eq:huber-problem} over the set of directed graphs shown in Figure~\ref{fig1}. %following a similar setup 
We generate data as in \cite{nedic2017achieving} using $p=5$ and $m_i=10$} for $i\in\cV$. We set the Huber loss parameter $\xi=2$, and   
for each $i\in\cV$, we generated $f_i(\bm{x})=\mathbf{H}_{\xi}(M_i \bm{x}-\bm{b}_i)$ as described in \cite[Sec.~6]{nedic2017achieving} such that $L_i=1$.  Moreover, we also initialized all the methods \sa{we tested} from $\bm{x}_i(0)=\mathbf{0}$ for all $i\in\cV$. %As noted in~\cite{nedic2017achieving}, since 
In our experiments, $n\in\{10,30,50,100, 200\}$, $m_i=10$ for all $i\in\cV$ and $p=5$; therefore, $\bar{f}$ and $f_i$ for $i\in\cV$
%is merely convex 
are restricted strongly convex when the regression error is small.}  
\begin{figure}[!h]
  \begin{subfigure}[b]{0.2\textwidth}
    \includegraphics[width=\textwidth]{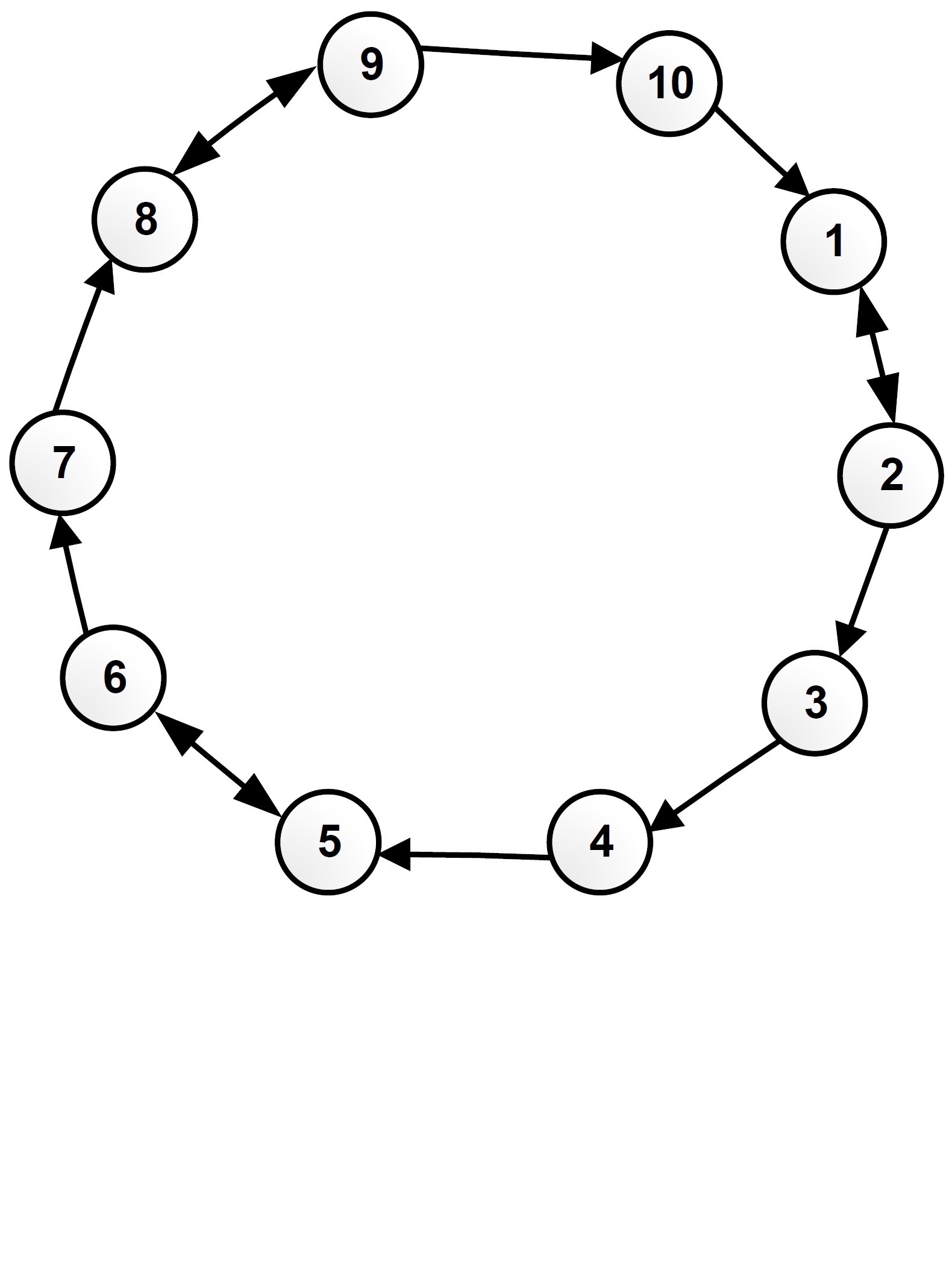}
    \caption{}
    \label{fig1a}
  \end{subfigure}
  \begin{subfigure}[b]{0.3\textwidth}
    \includegraphics[width=\textwidth]{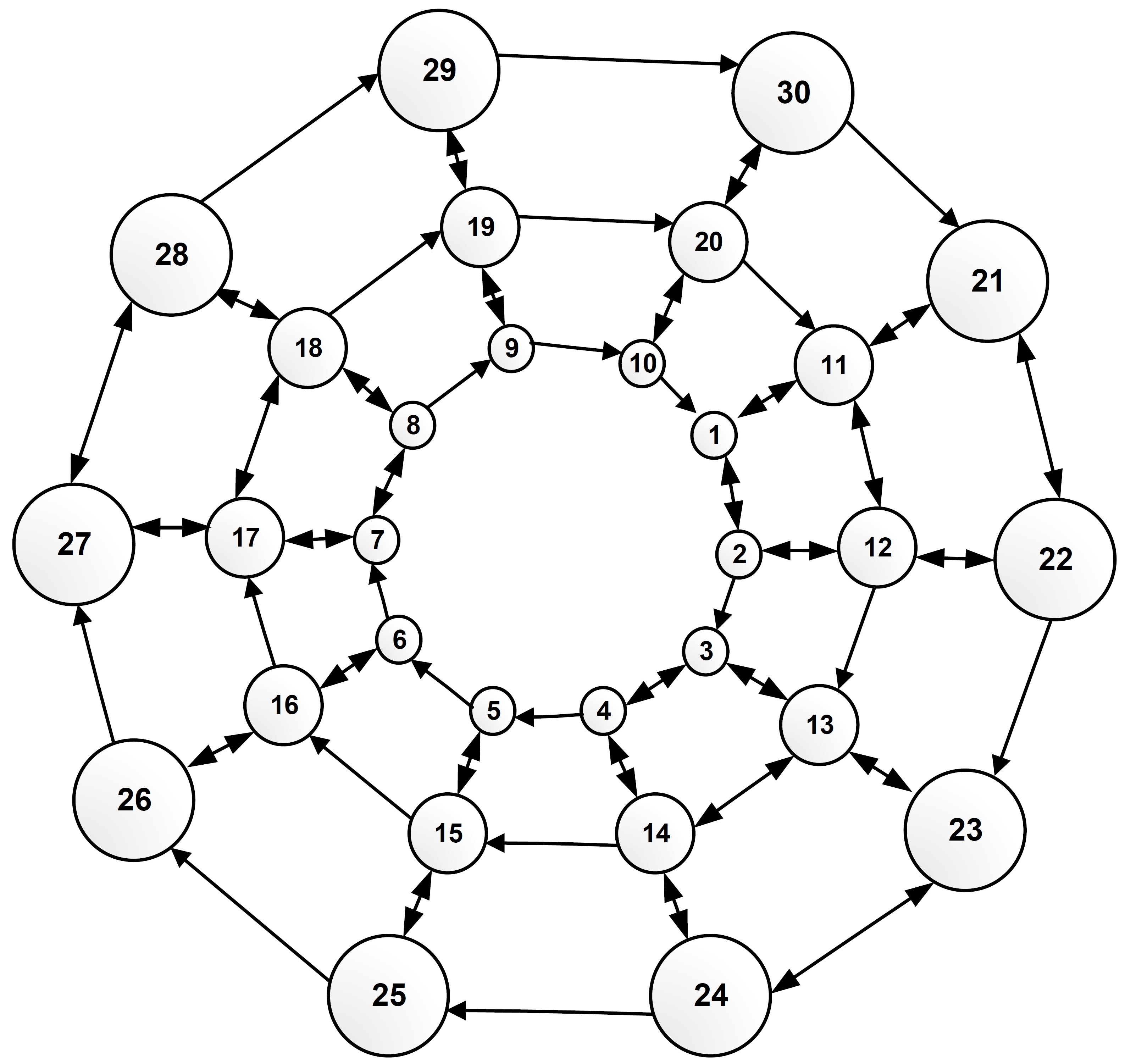}
    \caption{}
    \label{fig1b}
  \end{subfigure}
   \begin{subfigure}[b]{0.3\textwidth}
    \includegraphics[width=\textwidth]{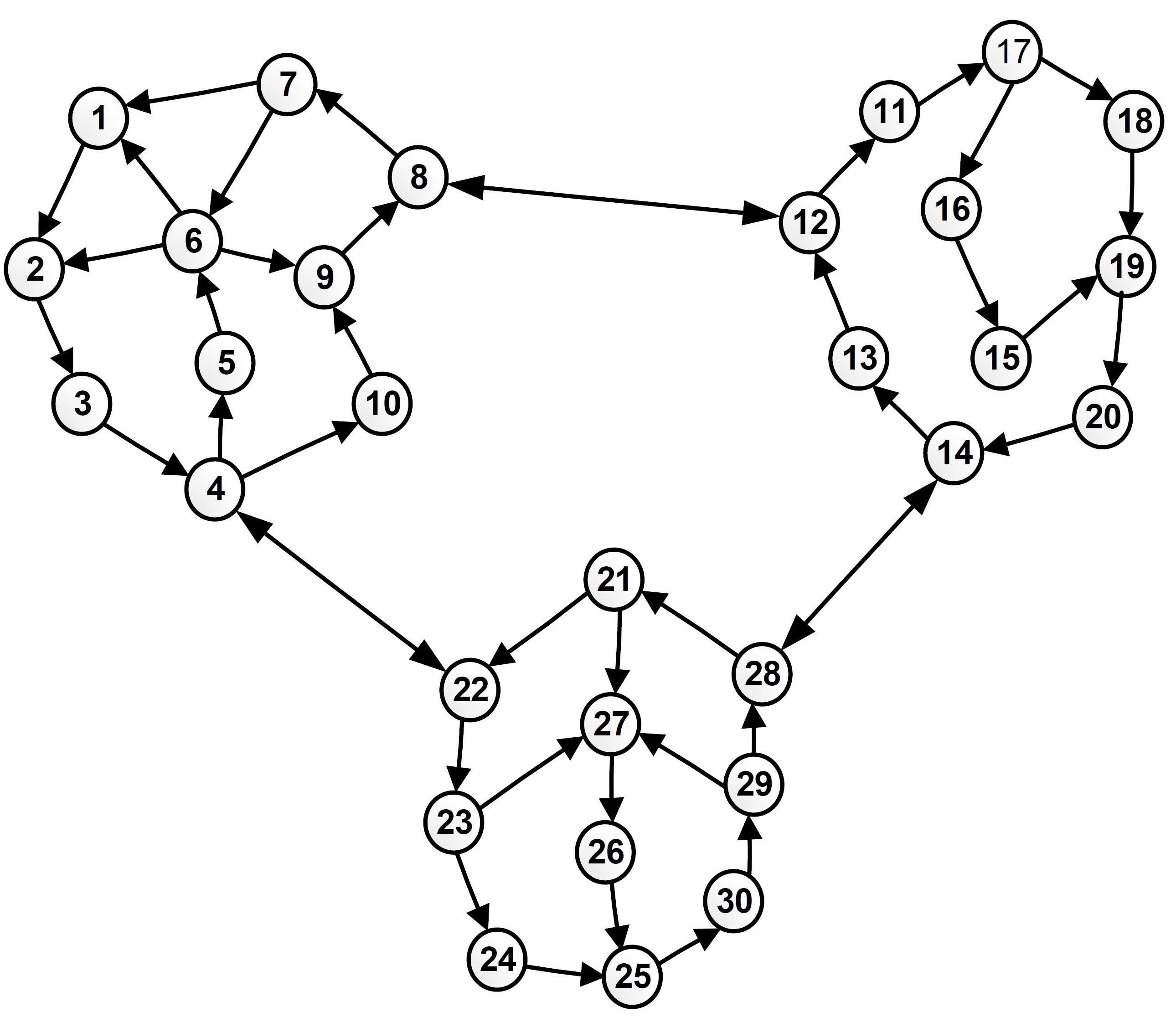}
    \caption{}
    \label{fig1c}
  \end{subfigure}
   \hfill
  \begin{subfigure}[b]{0.3\textwidth}
    \includegraphics[width=\textwidth]{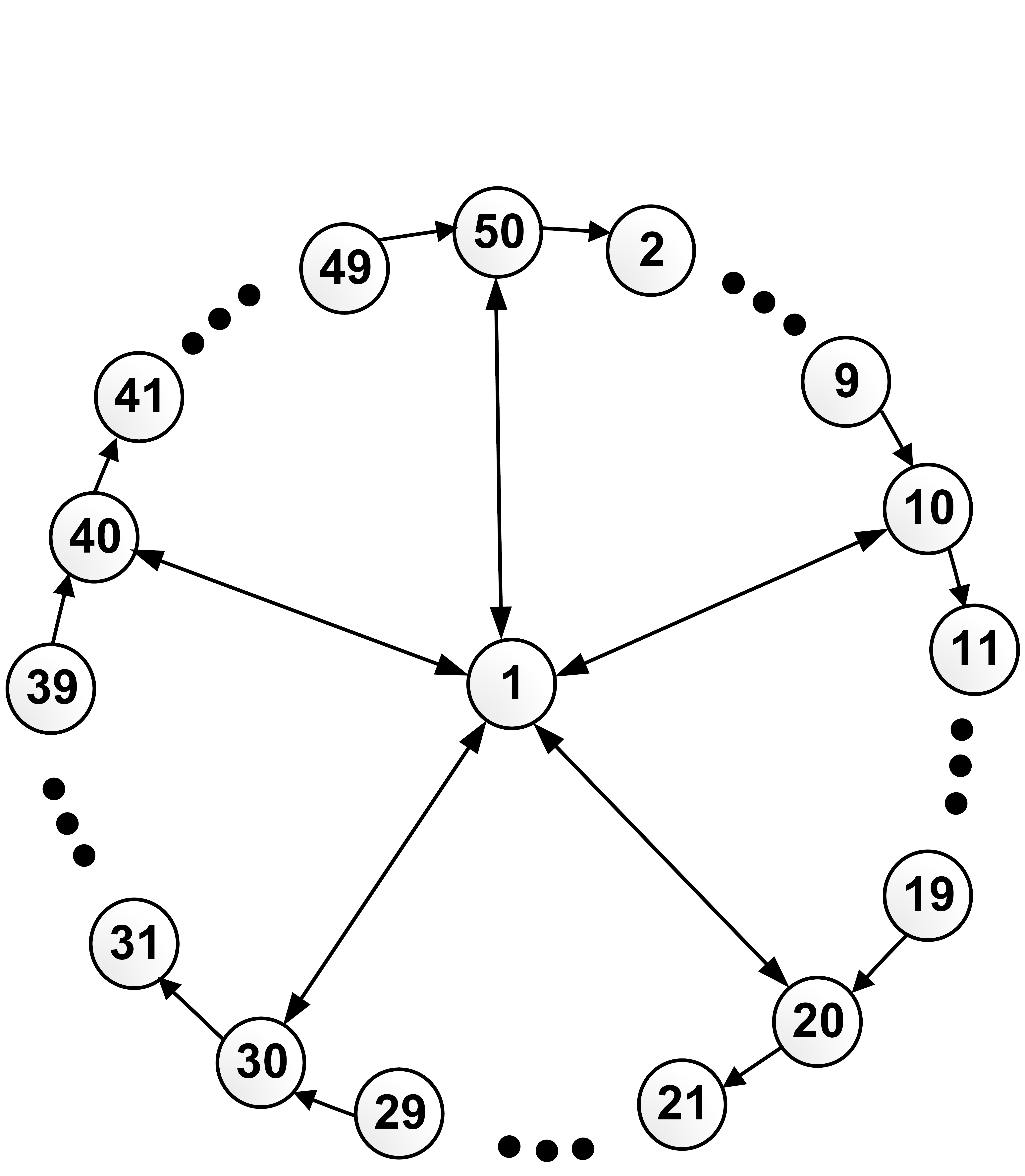}
    \caption{}
    \label{fig1d}
  \end{subfigure}
   \hfill
  \begin{subfigure}[b]{0.3\textwidth}
    \includegraphics[width=\textwidth]{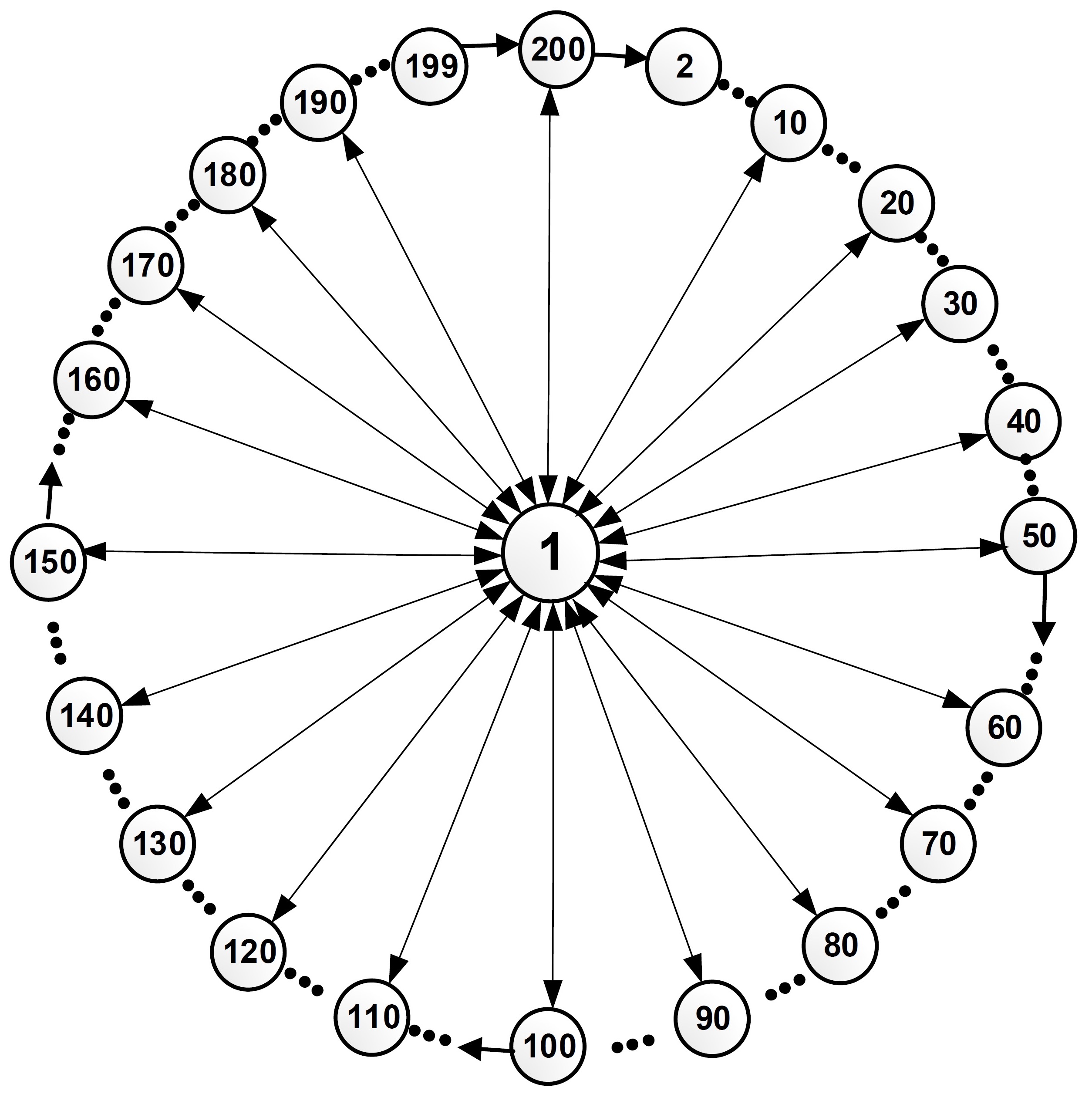}
    \caption{}
    \label{fig1e}
  \end{subfigure}
   \hfill
  \begin{subfigure}[b]{0.3\textwidth}
    \includegraphics[width=\textwidth]{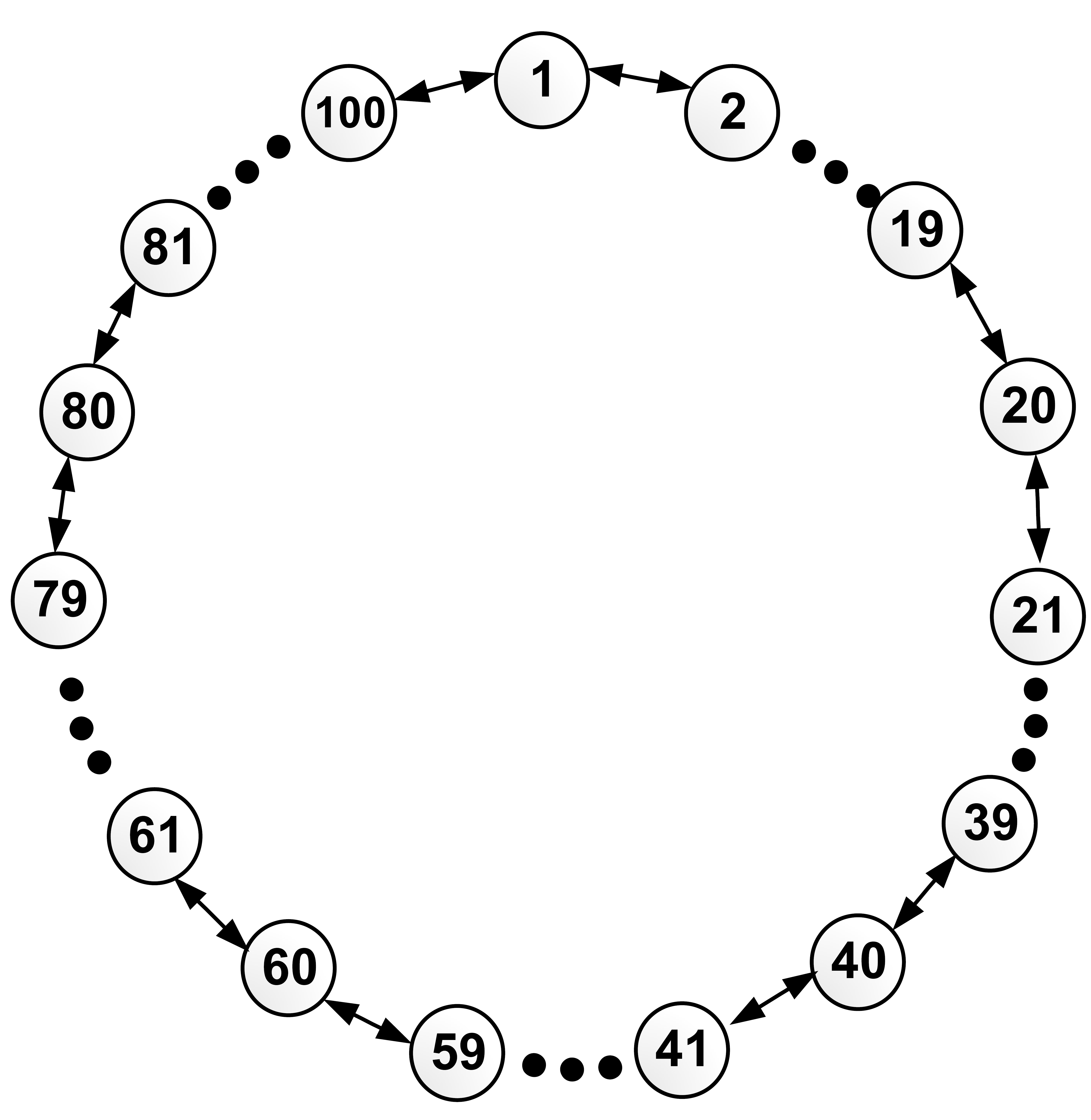}
    \caption{}
    \label{fig1f}
  \end{subfigure}
  \caption{\sa{Strongly-connected digraphs tested in our experiments.}}\label{fig1}
\end{figure}
\begin{figure}[!h] 
  \begin{subfigure}[b]{0.32\textwidth}
    \includegraphics[width=\textwidth]{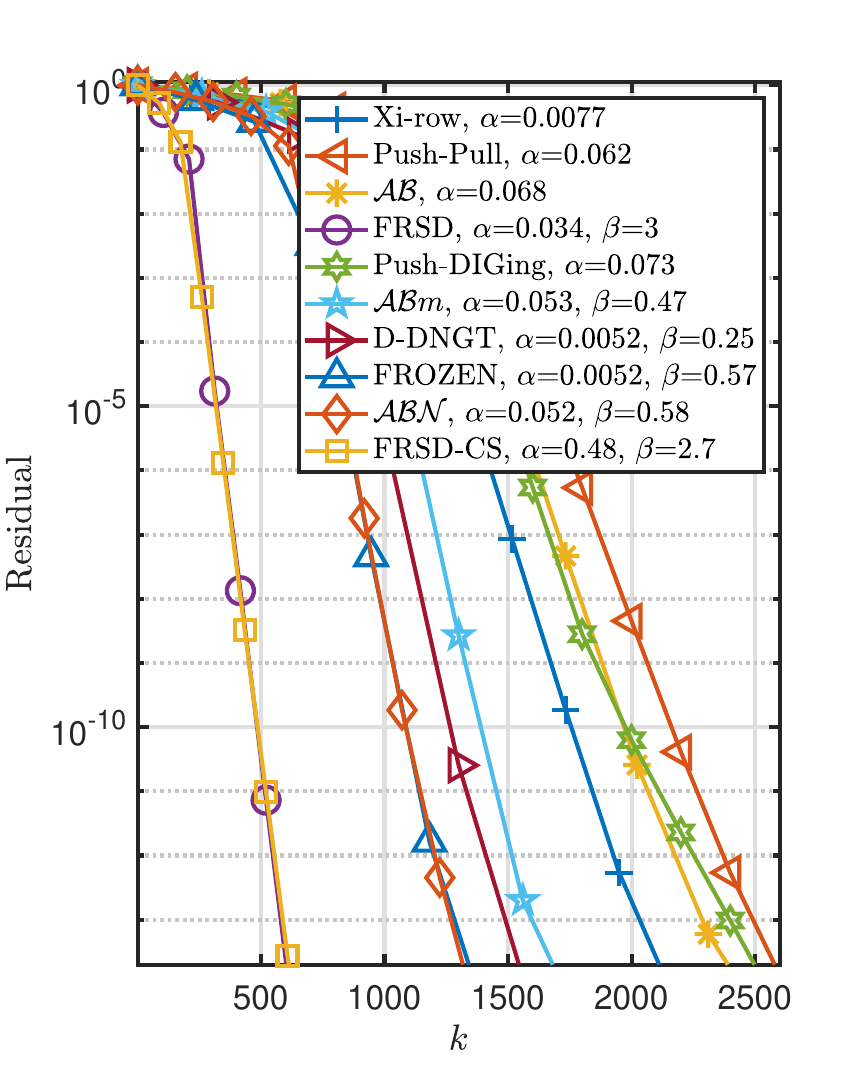}
    \caption{%The residuals 
    {$\{r(k)\}_k$ for  Fig.\ref{fig1}(a)}}
    \label{fig2a}
  \end{subfigure}
  \hfill
  \begin{subfigure}[b]{0.32\textwidth}
    \includegraphics[width=\textwidth]{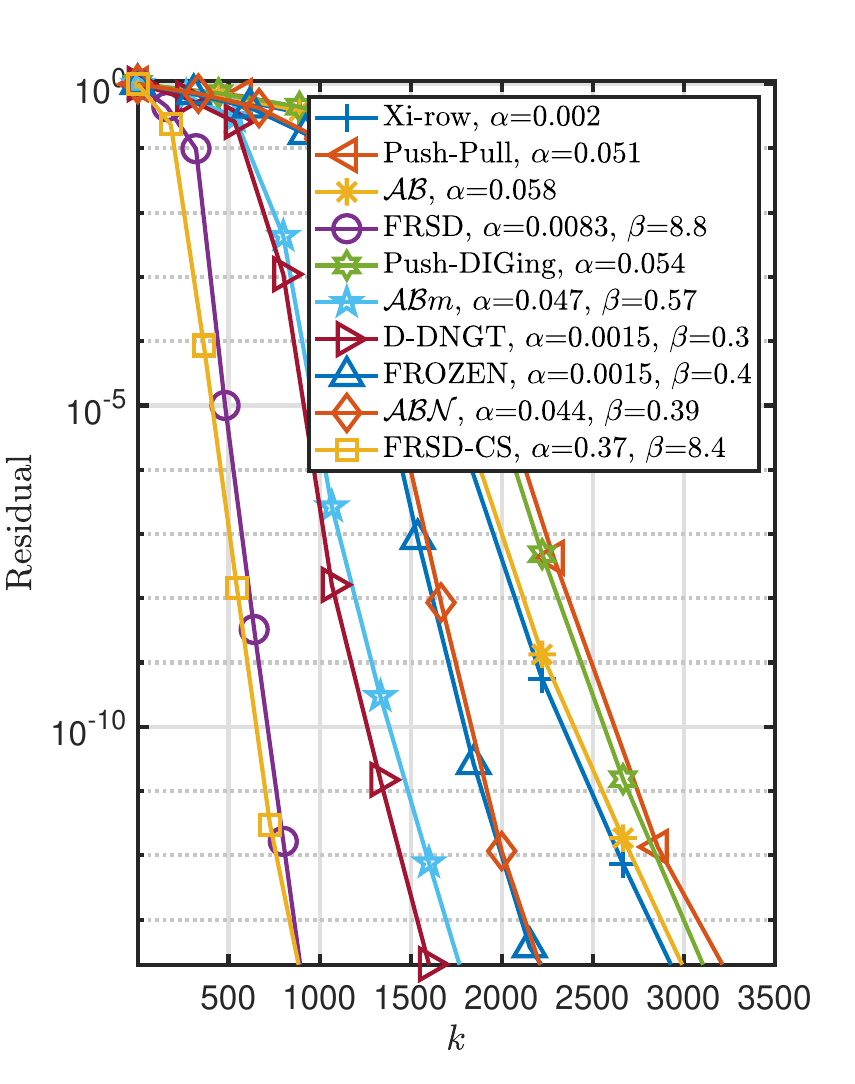}
    \caption{%The residuals 
    {$\{r(k)\}_k$ for  Fig.\ref{fig1}(b)}}
    \label{fig2b}
  \end{subfigure}
   \begin{subfigure}[b]{0.32\textwidth}
    \includegraphics[width=\textwidth]{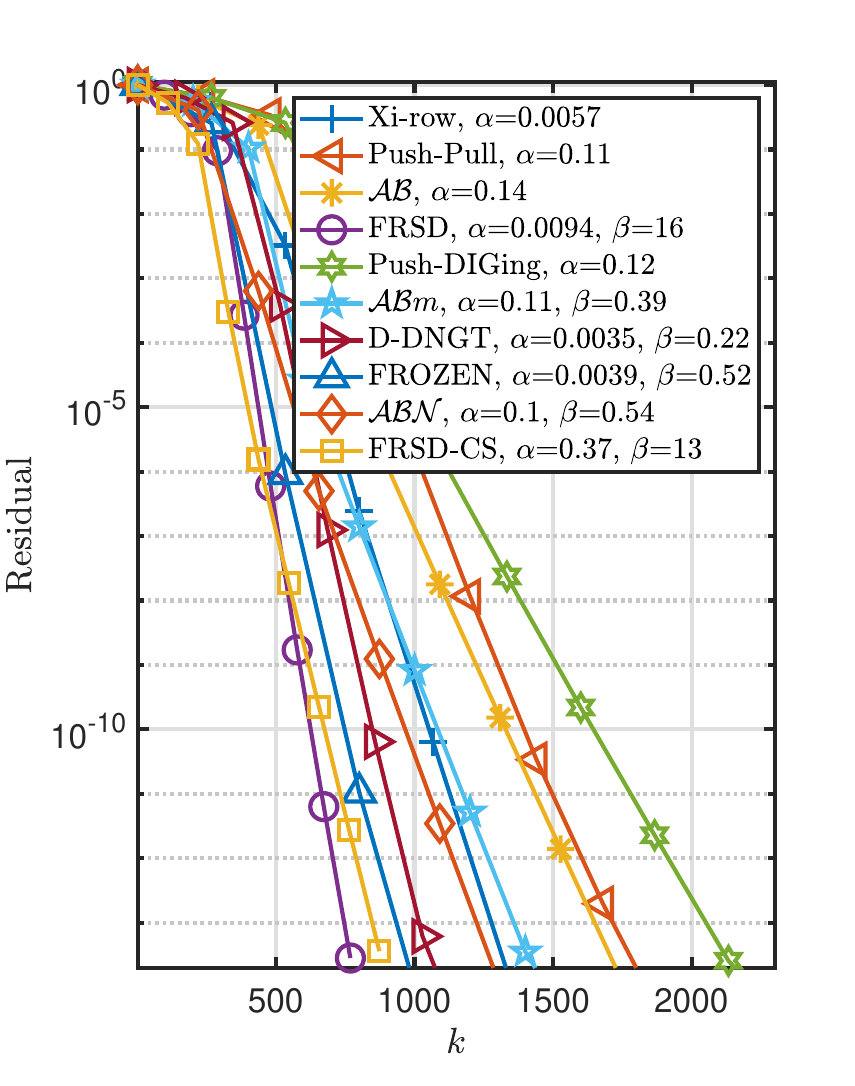}
    \caption{%The residuals 
    {$\{r(k)\}_k$ for  Fig.\ref{fig1}(c)}}
    \label{fig2c}
  \end{subfigure}
  \hfill
  \begin{subfigure}[b]{0.32\textwidth}
    \includegraphics[width=\textwidth]{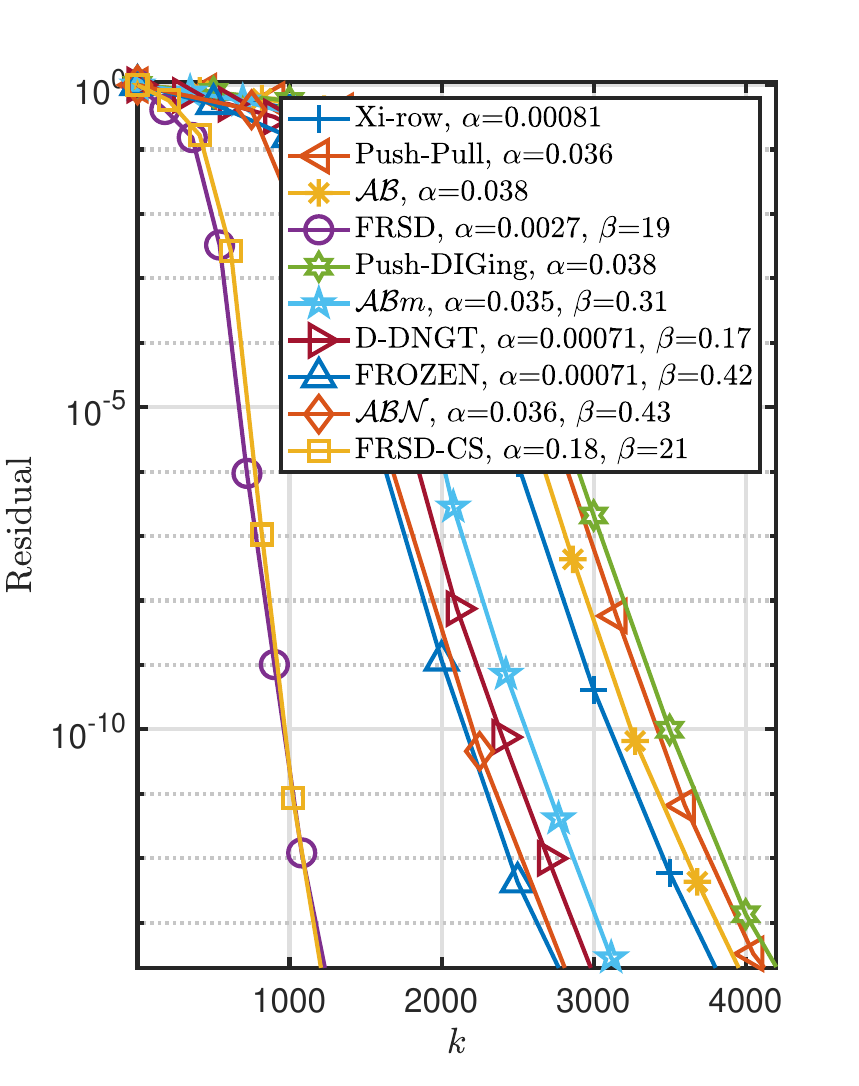}
    \caption{%The residuals 
    {$\{r(k)\}_k$ for Fig.\ref{fig1}(d)}}
    \label{fig2d}
  \end{subfigure}
  \begin{subfigure}[b]{0.32\textwidth}
    \includegraphics[width=\textwidth]{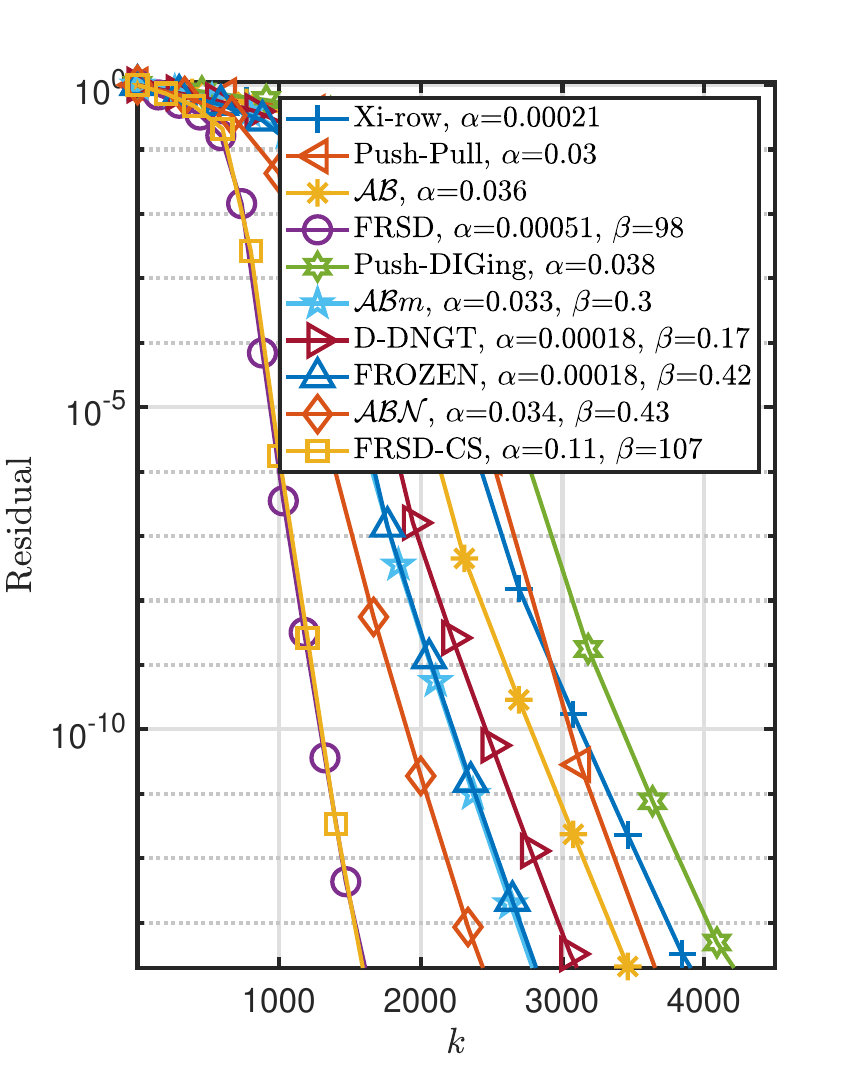}
    \caption{%The residuals 
    {$\{r(k)\}_k$ for  Fig.\ref{fig1}(e)}}
    \label{fig2e}
  \end{subfigure}
  \hfill
  \begin{subfigure}[b]{0.32\textwidth}
    \includegraphics[width=\textwidth]{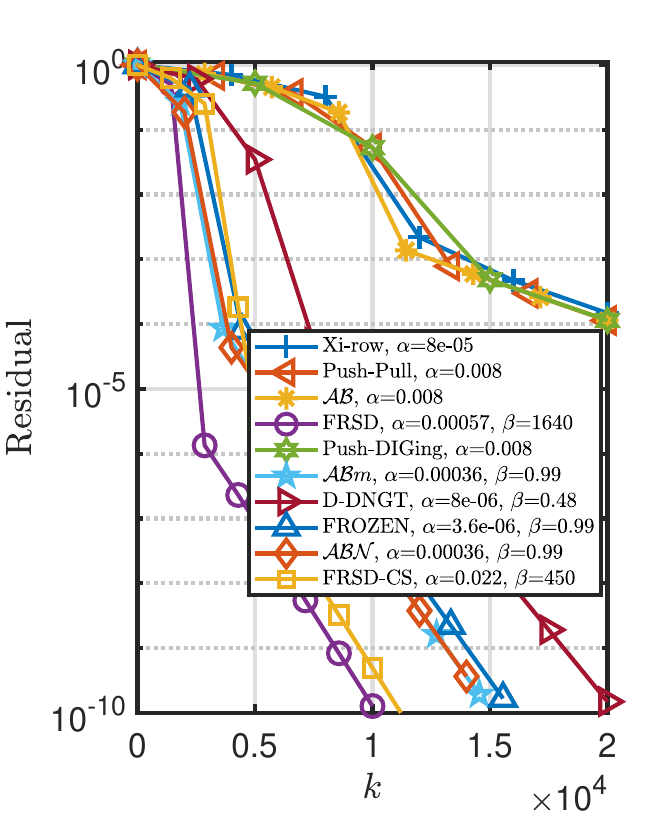}
    \caption{%The residuals 
    {$\{r(k)\}_k$ for Fig.\ref{fig1}(f)}}
    \label{fig2f}
  \end{subfigure}
  \caption{%{New, Huber Loss}
  Distributed Regression with Huber Loss}\label{fig2}
  \vspace*{-5mm}
\end{figure}
{In Fig.~\ref{fig2}, we plot %the results and show 
the residual sequence $\{r(k)\}_{k\geq 0}$ for all the methods where $r(k)\triangleq\dfrac{\Vert\bm{\mathrm{x}}(k)-\bm{\mathrm{x}}^{*}\Vert}{\Vert\bm{\mathrm{x}}(0)-\bm{\mathrm{x}}^{*}\Vert}$. To optimize the convergence rate, we tuned %the step size, $\alpha$, 
parameters for all algorithms.} 
% It is worth emphasizing that as FRSD has an additional parameter, $\beta$, while the other only has $\alpha$ to tune; therefore, we were able to tune $(\alpha,\beta)$ so that for both graphs in Fig.\ref{fig1}, FRSD exhibits a faster convergence compared to the others methods.
\subsection{
{Distributed} Logistic Regression} 
\label{sec:logistic}
{We now consider the distributed binary classification problem %, where we utilize 
using the logistic regression to train a linear classifier. Suppose each node (agent) $i\in\cV$ has access to $(M_i,\bm{b}_i)\in \mathbb{R}^{m_i\times p}\times\left\lbrace-1,+1 \right\rbrace^{m_i}$. Let $L:\reals\times\{-1,1\}\to\reals_+$ such that $L(u,v)=\ln(1+\exp(-u v))$; and for any $m\in\integers_+$, we also define $\mathbf{L}:\reals^m\times\{-1,1\}^m\to\reals^m_+$ such that $\mathbf{L}(\mathbf{u},\mathbf{v})=[L(u_j,v_j)]_{j=1}^m$ where $\mathbf{u}=[u_j]_{j=1}^m$ and $\mathbf{v}=[v_j]_{j=1}^m$.}
{The linear classifier $\bm{x}^*$ is computed by solving the regularized logistic regression problem:
\begin{align}
\label{eq:logistic}
% (u^*,v^*)\in\argmin_{u\in\mathbb{R}^{p}, v \in\mathbb{R}}\bar{f}(u,d)&\triangleq \sum\limits_{i=1}^{n} \sum\limits_{j=1}^{m_{i}} \ln  \big[ 1 + \exp\big(- \big( u^{T}M_{ij} \nonumber \\
% &+ v  \big)y_{ij} \big) \big]
% + \dfrac{n\lambda}{2}\Vert u \Vert_{2}^{2},
\bm{x}^*=\argmin_{\bm{x}\in\mathbb{R}^{p}}\bar{f}(x)&\triangleq \frac{1}{n}\sum\limits_{i\in\cV} \left(\mathbf{L}(M_i\bm{x},\bm{b}_i)+\dfrac{\lambda}{2}\Vert x \Vert_{2}^{2}\right).
\end{align}
where using regularization parameter $\lambda>0$ improves the ststistical properties of $\bm{x}^*$ -- see~\cite{sridharan2008fast}.}
\subsubsection{Tests on Specific Graph Topologies}
\label{sec:log_reg_test_deterministic}
In the \sa{first set of} experiments, we use the Australian-scale data set~\cite{chang2011libsvm} with $790$ data points where each data point {consists of a $14$-dimensional feature vector, i.e., $p=15$ to model the intercept, and the corresponding binary label. Suppose} each agent $i\in\cV$ samples $m_i=10$ {data points} uniformly at random from the training set {with replacement}. 
\sa{Hence, for each $i\in\cV$, we form $M_i\in\reals^{m_i\times p}$ using $m_i$ data points with $p-1$ features and set the last column of $M_i$ to $\bm{1}_{m_i}$ %the vector of all ones 
in order to model the intercept.} %Next, we normalize $M_i$ so that each of its columns become a unit vector, i.e., $\frac{M_i(:,j)}{\norm{M_i(:,j)}}$ for all $j=1,...,p$}. 
{We test the proposed method FRSD against the same methods that we compared with in Section~\ref{sec:Huber}. The residual sequence $\{r(k)\}_{k\geq 1}$ for all the methods are shown in Fig.~\ref{fig3}, where $r(k)$ is defined in Section~\ref{sec:Huber}.}\\

\begin{figure}[h!] 
  \begin{subfigure}[b]{0.32\textwidth}
    \includegraphics[width=\textwidth]{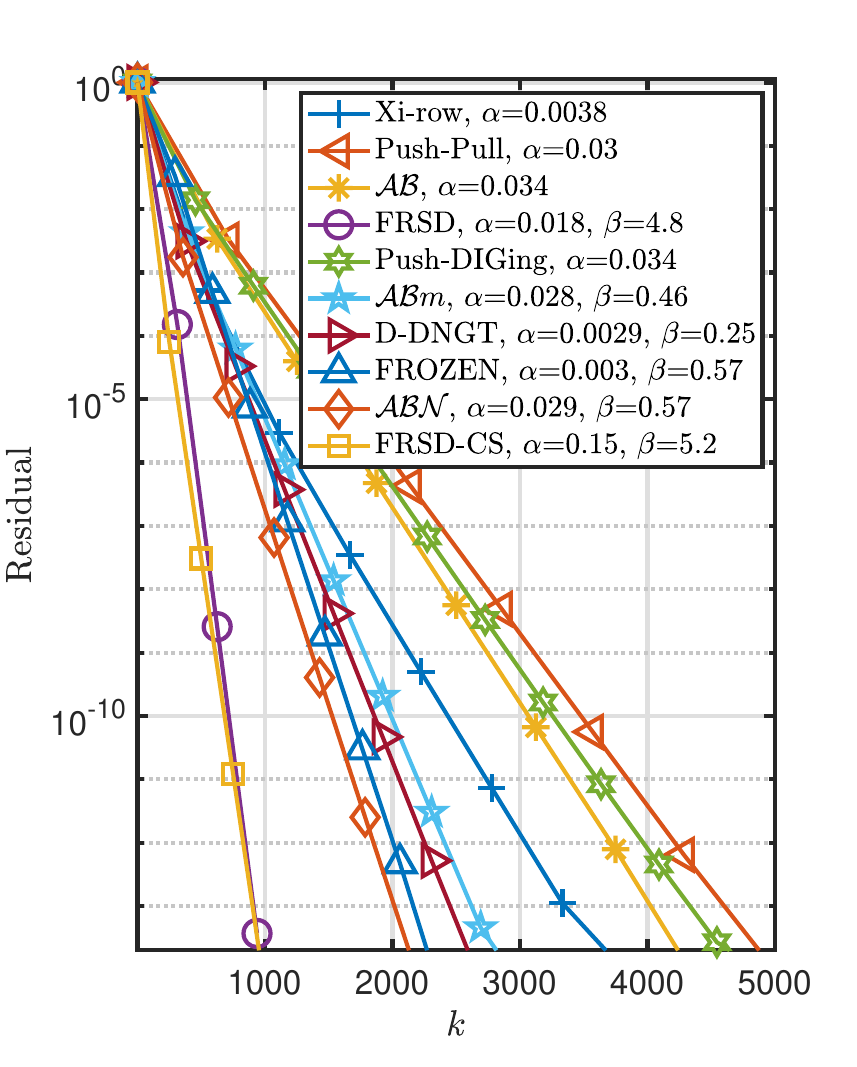}
    \caption{%The residuals 
    {$\{r(k)\}_k$ for  Fig.\ref{fig1}(a)}}
    \label{fig3a}
  \end{subfigure}
  \hfill
  \begin{subfigure}[b]{0.32\textwidth}
    \includegraphics[width=\textwidth]{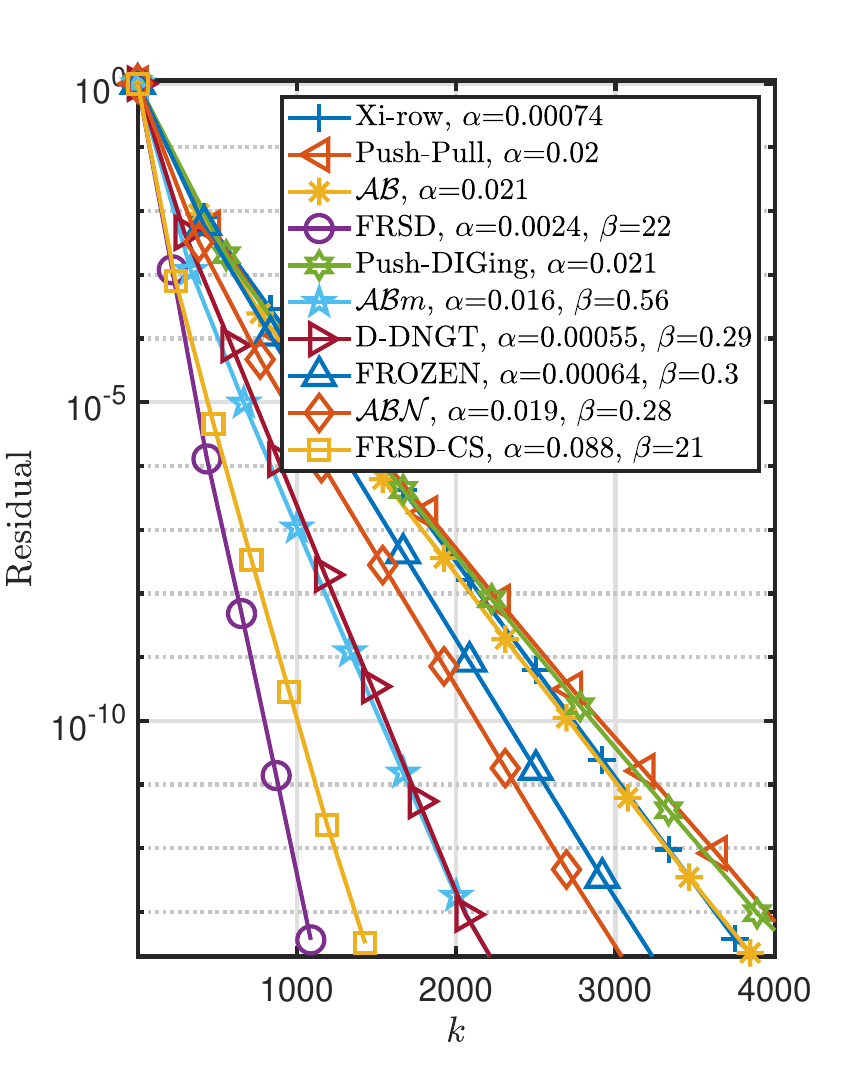}
    \caption{%The residuals 
    {$\{r(k)\}_k$ for Fig.\ref{fig1}(b)}}
    \label{fig3b}
  \end{subfigure}
    \begin{subfigure}[b]{0.32\textwidth}
    \includegraphics[width=\textwidth]{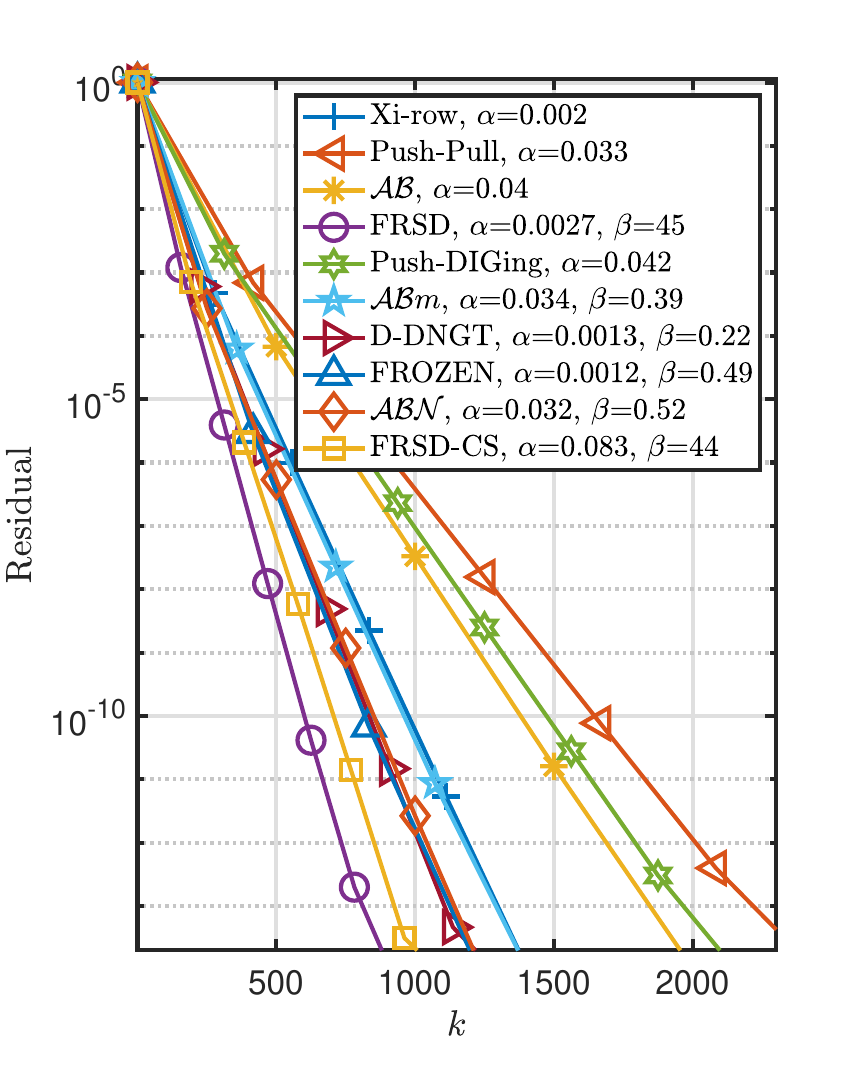}
    \caption{%The residuals 
    {$\{r(k)\}_k$ for Fig.\ref{fig1}(c)} }
    \label{fig3c}
  \end{subfigure}
  \hfill
  \begin{subfigure}[b]{0.32\textwidth}
    \includegraphics[width=\textwidth]{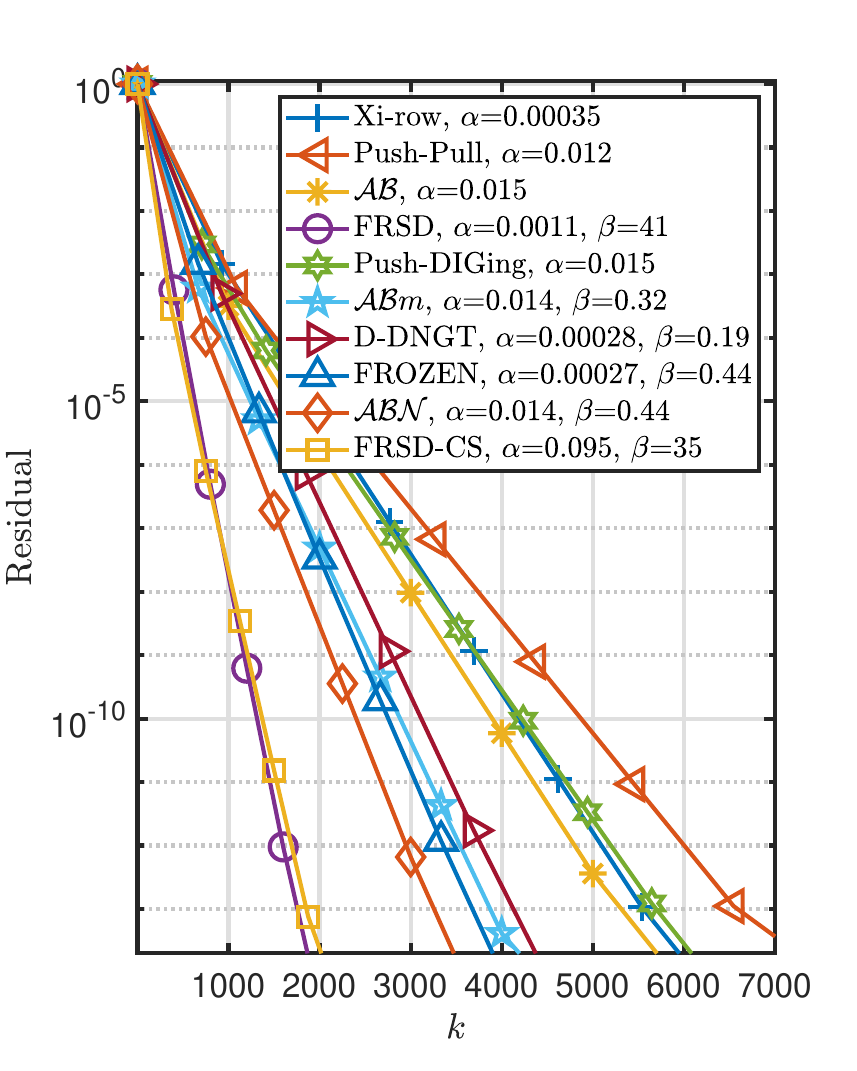}
    \caption{%The residuals 
    {$\{r(k)\}_k$ for Fig.\ref{fig1}(d)} }
    \label{fig3d}
  \end{subfigure}
  \hfill
  \begin{subfigure}[b]{0.32\textwidth}
    \includegraphics[width=\textwidth]{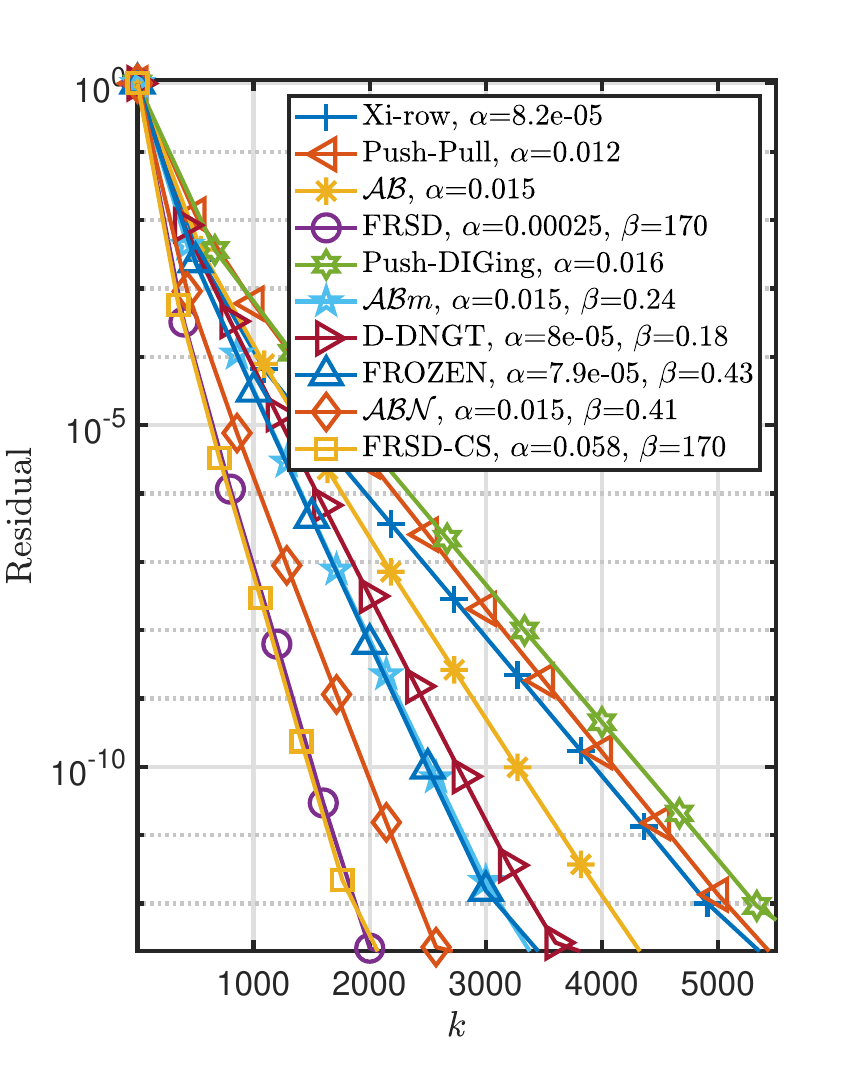}
    \caption{%The residuals 
    {$\{r(k)\}_k$ for Fig.\ref{fig1}(e)}}
    \label{fig3e}
  \end{subfigure}
  \begin{subfigure}[b]{0.32\textwidth}
    \includegraphics[width=\textwidth]{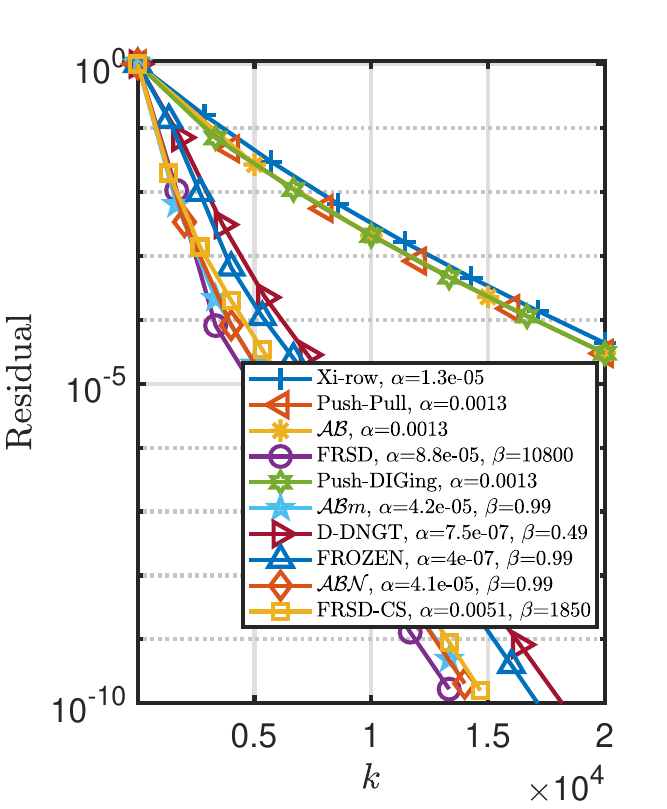}
    \caption{%The residuals 
    {$\{r(k)\}_k$ for Fig.\ref{fig1}(f)}}
    \label{fig3f}
  \end{subfigure}
  \caption{%New, Logistic Regression}\label{fig3}
  Distributed Logistic Regression}\label{fig3}
  \vspace*{-5mm}
\end{figure}
\subsubsection{Tests on Random Graphs}
\label{sec:random}
\sa{In the second set of experiments, we tested FRSD and FRSD-CS over random graphs against the other row-stochastic methods, i.e., \rev{Xi-row}, FROZEN and D-DGNT, to solve \rev{the} distributed logistic regression problem defined in Section~\ref{sec:logistic}. We considered two scenarios: Scenario I $n>p$ and Scenario II $p>n$ --recall that $n$ and $p$ denote the number of nodes in the network and the dimension of the decision variable, respectively. For Scenario I, i.e., $n>p$, we looked at two cases: low connectivity ratio (sparser graphs) and high connectivity ratio (denser graphs), where the connectivity ratio is defined as the ratio of the number edges to $n(n-1)$ \rev{--note that $n(n-1)$ is equal to the number of all possibles edges excluding self-loops.} For each scenario, we ran all $5$ algorithms on $20$ different randomly generated graphs. We reported the residual $r(k)\triangleq\norm{\bx(k)-\bx^*}/\norm{\bx(0)-\bx^*}$ against iteration counter $k$ and we also reported the residual against the amount communication per node by the end of iteration $k$ --recall that at each iteration FRSD and FRSD-CS require each node to broadcast only $n+p$-dimensional vector while the others, i.e., \rev{Xi-row}, FROZEN and D-DGNT, require each node to broadcast $n+2p$-dimensional vector. We have observed that both FRSD and FRSD-CS are competitive against the state of the art row stochastic methods, and the performance of our algorithms is %noticeable
\rev{superior} either when $n<p$ or when the graphs are sparse, which is \rev{indeed the case for most of the real-life networks.} %in practice.
Next we describe how we generated the random graphs.}

\paragraph{Random Graph Generation}
\label{sec:graph_gen}
\sa{We used \texttt{DGen} code\footnote{\texttt{DGen} code is written by Dr. W. Shi, see \url{https://sites.google.com/view/wilburshi/home/research/software-a-z-order/graph-tools/dgen}} to generate strongly connected random graphs. The algorithm \texttt{DGen}, implemented in MATLAB, receives two input: number of nodes $n$, and the connectivity ratio $\phi\in(0,1]$. Given $n$ and $\phi$, \texttt{DGen} generates a strongly connected random graph with $|\cE|=\lceil \phi n(n-1)\rceil\triangleq m_{n,\phi}$ edges. Let $\{p_i\}_{i=1}^n$ be a permutation of $[n]\triangleq\{1,\ldots,n\}$ chosen uniformly at random, and let $\cI=\{i\in[n]:\ p_i\neq i\}$. Then \texttt{DGen} creates a directed cycle $\cC$ using the nodes $\{p_i\}_{i\in\cI}$. Now consider a smaller dimensional graph with nodes $\{i:~i\in [n]\setminus\cI\}\cup\{c^*\}$ where $c^*$ is a ``super-node" representing the cycle $\cC$. Note that this new graph has $n-|\cI|+1$ nodes; one can repeat the above process by setting $n\gets n-|\cI|+1$, and whenever we connect a node from $[n]\setminus\cI$ with $c^*$, one randomly picks a node belonging to $\cC$. This process ends when the smaller dimensional graph has only a single super-node with no other nodes, which gives us a strongly connected graph. Say this graph has $\tilde{m}$ nodes, then the remaining $m_{n,\phi}-\tilde{m}$ edges are randomly added to obtain a strongly connected graph with connectivity ratio $\phi$.}

\rev{In the experiments with randomly generated strongly connected graphs, we only tested row-stochastic methods. Indeed, being able to get away with the eigenvector estimation through employing both row- and column-stochastic mixing matrices, AB-type methods, e.g., $\mathcal{AB}$~\cite{xin2018linear}, $\mathcal{AB}$m~\cite{xin2019distributed}, Push-Pull~\cite{pu2020push} and $\mathcal{ABN}$~\cite{xin2019distributedNEST}, perform better on these experiments than the first-order methods using row-stochastic weights alone. That is why we only reported the results for the row-stochastic methods to have a fair comparison among equivalent methods.}

\paragraph{Scenario I $(n>p)$} \sa{We set $n=200$ 
%and $p=15$, 
and generated $20$ random graphs for each connectivity ratio $\phi\in\{0.015, 0.15\}$. We use the same dataset and same problem setup with Section~\ref{sec:log_reg_test_deterministic}, i.e., $p=15$ and the number of data points $m_i=10$ for all $i\in\cV$. In Figures~\ref{fig:test-I-high} and~\ref{fig:test-I-low}, we report the results for \emph{high} connectivity ratio $\phi=0.15$ and the \emph{low} one $\phi=0.015$, respectively. For Scenario I, i.e., when $n>p$, we observe that FRSD and FRSD-CS are competitive against FROZEN and D-DGNT while performing better than \rev{Xi-row}. Furthermore, we also observe that the performances of FRSD and FRSD-CS improve as the random graphs get sparser, i.e., they perform significantly better for $\phi=0.015$ when compared to their performance for $\phi=0.15$. Finally, as expected, for $n>p$, the residual shows the same decay patterns with respect to increase in either iteration counter or the amount of data broadcast per node.}

% \begin{figure}[h!] 
%   \begin{subfigure}[a]{0.235\textwidth}
%     \includegraphics[width=\textwidth]{Fig_N200_015.eps}
%     \caption{%The residuals 
%   }
%     \label{fig:response-I-low-iter}
%   \end{subfigure}
%   \hfill
%   \begin{subfigure}[a]{0.24\textwidth}
%     \includegraphics[width=\textwidth]{Fig_C_N200_015.eps}
%     \caption{%The residuals 
%     }
%     \label{fig:response-I-low-comm}
%   \end{subfigure}
%   \caption{
%   Distributed logistic regression problem ($n=200$, $p=15$) over 20 random  directed graphs with a low connectivity ratio $\phi =  0.015 $. Solid curves represent the average and the shaded region represents the range statistics.}\label{fig:response-I-low}
% \end{figure}
%%%%%%%%%%%%%%%%%%%%%%%
\vspace*{-2mm}
\begin{figure}[h!]
  \begin{subfigure}[a]{0.45\textwidth}
    \includegraphics[width=\textwidth]{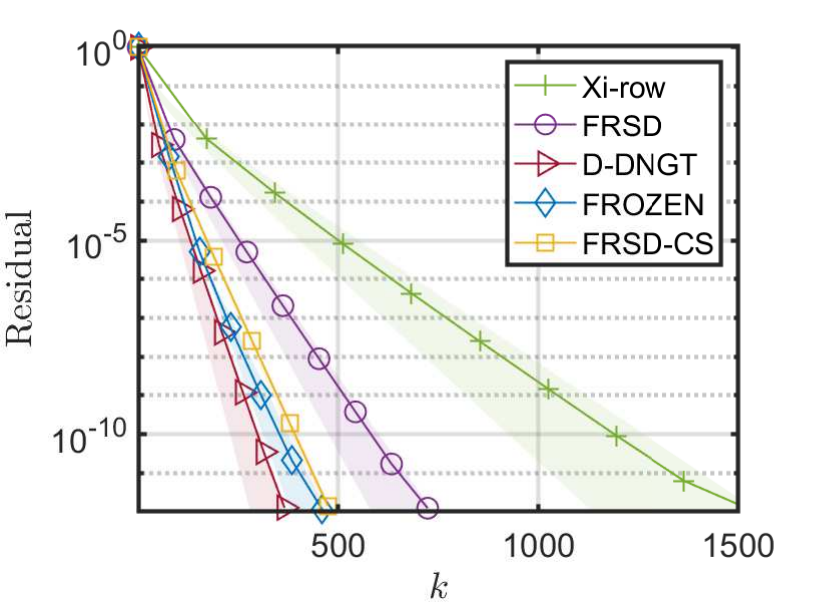}
    \caption{%The residuals 
   }
    \label{fig:test-I-high-iter}
  \end{subfigure}
  \hfill
  \begin{subfigure}[a]{0.45\textwidth}
    \includegraphics[width=\textwidth]{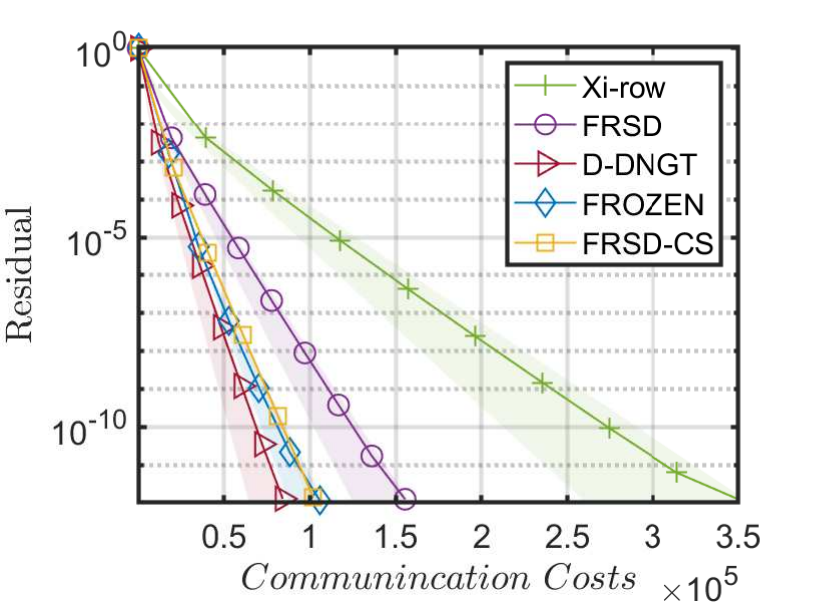}
    \caption{%The residuals 
    }
     \label{fig:test-I-high-comm}
  \end{subfigure}
   \caption{
  Distributed logistic regression problem ($n=200$, $p=15$) over 20 random  directed graphs with a high connectivity ratio $\phi =  0.15 $. Solid curves represent the average and the shaded region represents the range statistics.} \label{fig:test-I-high}
\end{figure}
\begin{figure}[h!] 
  \begin{subfigure}[a]{0.45\textwidth}
    \includegraphics[width=\textwidth]{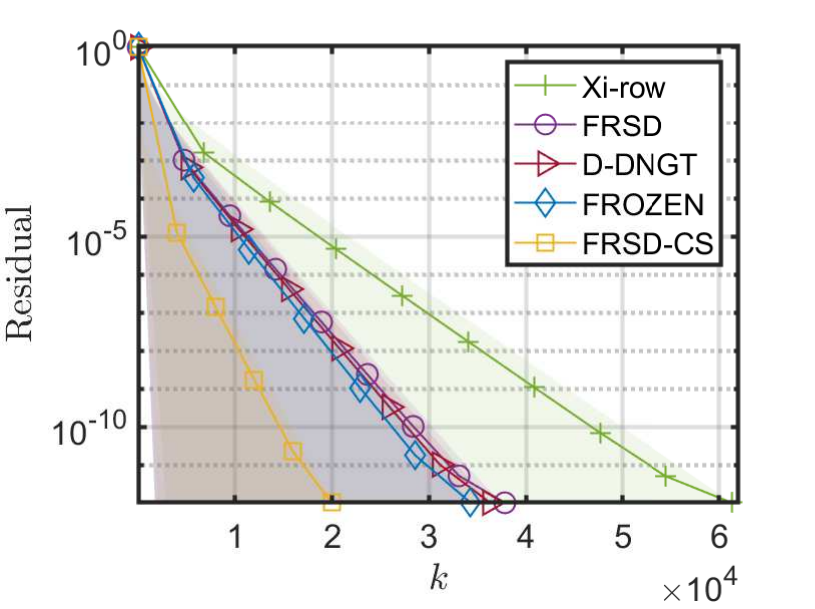}
    \caption{%The residuals 
   }
    \label{fig:test-I-low-iter}
  \end{subfigure}
  \hfill
  \begin{subfigure}[a]{0.45\textwidth}
    \includegraphics[width=\textwidth]{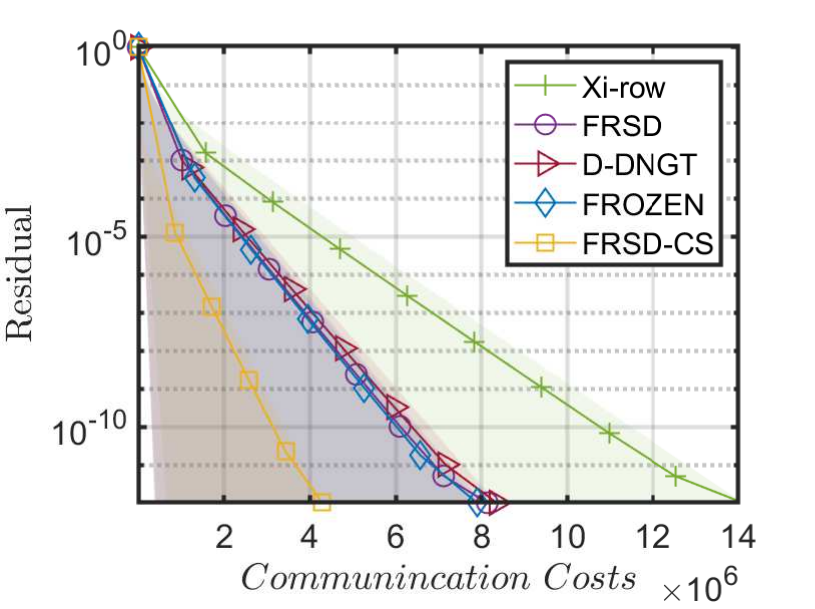}
    \caption{%The residuals 
    }
    \label{fig:test-I-low-comm}
  \end{subfigure}
   \caption{
  Distributed logistic regression problem ($n=200$, $p=15$) over 20 random  directed graphs with a low connectivity ratio $\phi =  0.015 $. Solid curves represent the average and the shaded region represents the range statistics.}
  \label{fig:test-I-low}
\end{figure}
% \begin{figure}[h!] 
%   \begin{subfigure}[a]{0.5\textwidth}
%     \includegraphics[width=\textwidth]{Fig_N25.pdf}
%     \caption{%The residuals 
%   }
%     \label{fig2a}
%   \end{subfigure}
%   \hfill
%   \begin{subfigure}[a]{0.5\textwidth}
%     \includegraphics[width=\textwidth]{Fig_C_N25.pdf}
%     \caption{%The residuals 
%     }
%   \end{subfigure}
%   \caption{
%   Distributed logistic regression: $\{r(k)\}_k$ for the random  directed graphs  $\rho =  0.1 $}  \label{fig_N25}
% \end{figure}

\paragraph{Scenario II ($n<p$)} \sa{We set $n=25$ 
%and $p=301$, 
and generated $20$ random graphs with connectivity ratio $\phi=0.1$, i.e., each random graph has $60$ edges. We used \texttt{w1a.t} (testing) dataset \cite{chang2011libsvm} with 47,272 data points with each data point consisting of 300 features vector -- implying $p=301$ to model the intercept. For classification, we again used the binary logistic regression model of Section~\ref{sec:log_reg_test_deterministic} with $p=301$ and $m_i= 400$ for all $i\in\cV$. In Figure~\ref{fig:test-II}, we report the results for this scenario. Indeed, when $n<p$, we observe that FRSD and FRSD-CS are competitive against FROZEN while performing better than both \rev{Xi-row} and D-DGNT. Finally, unlike Scenario I, when $n<p$, the advantage of FRSD and FRSD-CS over the other row-stocastic method in terms of lower communication overhead becomes more apparent: while the residuals for FRSD and FROZEN show the same decay patterns as the iteration counter increases, one can observe that FRSD performs better than FROZEN considering the amount of data broadcast per node since both FRSD and FRSD-CS need each node to broadcast $n+p$-dimensional vector, i.e., $\approx p$ as $p\gg n$, FROZEN requires broadcasting $n+2p$-dimensional vector, i.e., $\approx 2p$; hence, almost twice the
%communication burden
\rev{communication overhead} of FRSD.} 

\begin{figure}[h!] 
  \begin{subfigure}[a]{0.45\textwidth}
    \includegraphics[width=\textwidth]{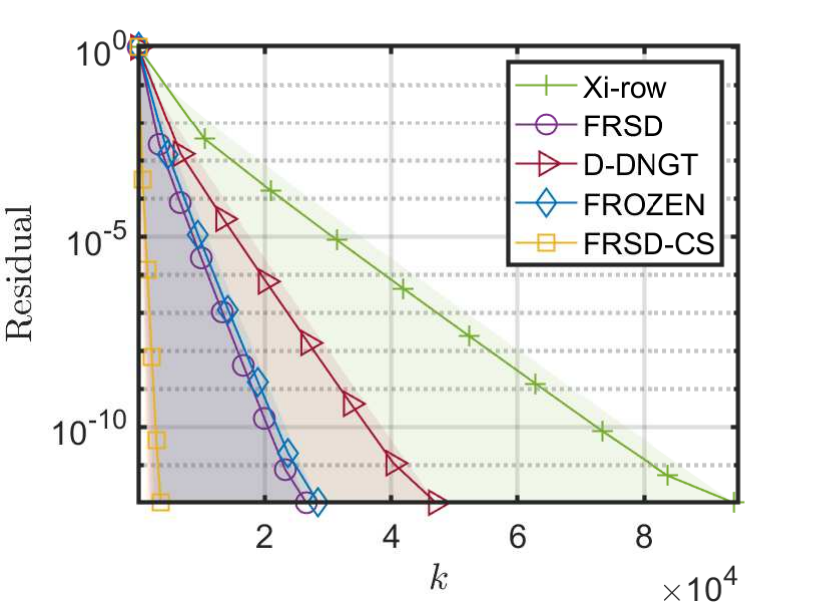}
    \caption{%The residuals 
   }
    \label{fig:test-II-iter}
  \end{subfigure}
  \hfill
  \begin{subfigure}[a]{0.45\textwidth}
    \includegraphics[width=\textwidth]{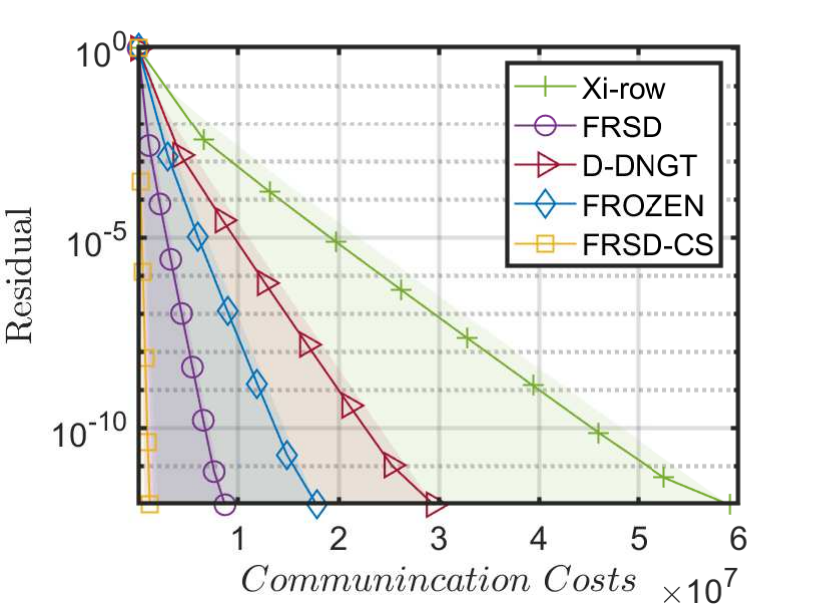}
    \caption{%The residuals 
    }
    \label{fig:test-II-comm}
  \end{subfigure}
   \caption{
  Distributed logistic regression problem ($n=25$, $p=301$) over 20 random  directed graphs with a low connectivity ratio $\phi =  0.1$. Solid curves represent the average and the shaded region represents the range statistics.}  
  \label{fig:test-II}
\end{figure}

% \begin{figure}[h!]
%   \begin{subfigure}[a]{0.24\textwidth}
%     \includegraphics[width=\textwidth]{Fig_N200_15_sha_1.pdf}
%     \caption{%The residuals 
%   }
%     \label{fig2a}
%   \end{subfigure}
%   \hfill
%   \begin{subfigure}[a]{0.24\textwidth}
%     \includegraphics[width=\textwidth]{Fig_Comu_N200_15_sha_1.pdf}
%     \caption{%The residuals 
%     }
%   \end{subfigure}
%   \caption{
%   Distributed logistic regression: $ \{r(k)\}_k $ for the random  directed graphs  $\rho =  0.15 $} \label{fig_rho_15}
% \end{figure}   
% \begin{figure}[h!] 
%   \begin{subfigure}[a]{0.24\textwidth}
%     \includegraphics[width=\textwidth]{Fig_N200_015_sha_1.pdf}
%     \caption{%The residuals 
%   }
%     \label{fig2a}
%   \end{subfigure}
%   \hfill
%   \begin{subfigure}[a]{0.24\textwidth}
%     \includegraphics[width=\textwidth]{Fig_Comu_N200_015_sha_1.pdf}
%     \caption{%The residuals 
%     }
%   \end{subfigure}
%   \caption{
%   Distributed logistic regression: $\{r(k)\}_k $ for the random  directed graphs  $\rho =  0.015 $}\label{fig_rho_015}
% \end{figure}
% \begin{figure}[h!] 
%   \begin{subfigure}[a]{0.24\textwidth}
%     \includegraphics[width=\textwidth]{Fig_N25_sha_1.pdf}
%     \caption{%The residuals 
%   }
%     \label{fig2a}
%   \end{subfigure}
%   \hfill
%   \begin{subfigure}[a]{0.24\textwidth}
%     \includegraphics[width=\textwidth]{Fig_Comu_N25_sha_1.pdf}
%     \caption{%The residuals 
%     }
%   \end{subfigure}
%   \caption{
%   Distributed logistic regression: $ \{r(k)\}_k $ for the random  directed graphs  $\rho =  0.1 $}  \label{fig_N25}
% \end{figure}
\section{Conclusion}
{In this paper, we proposed a distributed optimization algorithm, {FRSD, for decentralized consensus optimization} over directed graphs. {FRSD only employs a row-stochastic matrix for local messaging with neighbors, making it desirable for broadcast-based communication systems.} 
The proposed algorithm achieves a geometric convergence to the global optimal {when agents' cost functions are strongly convex with Lipschitz continuous gradients}. {Empirical results demonstrated the efficacy of the %novel momentum term 
implicit gradient tracking technique employed by FRSD, \rev{which led to: (i) reduction in the data stored, and (ii) reduction in the data broadcast, for each node. More precisely, FRSD does not need to store  $\bm{x}$ iterate from the previous iteration while it is needed for all other methods explicitly using the gradient tracking term; furthermore, FRSD also eliminates the need for broadcasting a variable related to gradient tracking.} %which performed better in practice than the other-state-of-the-art methods we compared. 
\sa{As a future research direction, we consider extending our results to the asynchronous computation setting over directed communication graphs.}}} 
%\bibliographystyle{IEEEtran}
%\bibliography{refe}

%\bibliography{refe}

\begin{thebibliography}{10}
\providecommand{\url}[1]{#1}
\csname url@samestyle\endcsname
\providecommand{\newblock}{\relax}
\providecommand{\bibinfo}[2]{#2}
\providecommand{\BIBentrySTDinterwordspacing}{\spaceskip=0pt\relax}
\providecommand{\BIBentryALTinterwordstretchfactor}{4}
\providecommand{\BIBentryALTinterwordspacing}{\spaceskip=\fontdimen2\font plus
\BIBentryALTinterwordstretchfactor\fontdimen3\font minus
  \fontdimen4\font\relax}
\providecommand{\BIBforeignlanguage}[2]{{%
\expandafter\ifx\csname l@#1\endcsname\relax
\typeout{** WARNING: IEEEtran.bst: No hyphenation pattern has been}%
\typeout{** loaded for the language `#1'. Using the pattern for}%
\typeout{** the default language instead.}%
\else
\language=\csname l@#1\endcsname
\fi
#2}}
\providecommand{\BIBdecl}{\relax}
\BIBdecl

\bibitem{rabbat2004distributed}
M.~Rabbat and R.~Nowak, ``Distributed optimization in sensor networks,'' in
  \emph{Proceedings of the 3rd international symposium on Information
  processing in sensor networks}, 2004, pp. 20--27.

\bibitem{khan2009diland}
U.~A. Khan, S.~Kar, and J.~M. Moura, ``{DILAND}: An algorithm for distributed
  sensor localization with noisy distance measurements,'' \emph{IEEE
  Transactions on Signal Processing}, vol.~58, no.~3, pp. 1940--1947, 2009.

\bibitem{bullo2009distributed}
F.~Bullo, J.~Cortes, and S.~Martinez, \emph{Distributed control of robotic
  networks: a mathematical approach to motion coordination algorithms}.\hskip
  1em plus 0.5em minus 0.4em\relax Princeton University Press, 2009, vol.~27.

\bibitem{cevher2014convex}
V.~Cevher, S.~Becker, and M.~Schmidt, ``Convex optimization for big data:
  Scalable, randomized, and parallel algorithms for big data analytics,''
  \emph{IEEE Signal Processing Magazine}, vol.~31, no.~5, pp. 32--43, 2014.

\bibitem{boyd2011distributed}
S.~Boyd, N.~Parikh, and E.~Chu, \emph{Distributed optimization and statistical
  learning via the alternating direction method of multipliers}.\hskip 1em plus
  0.5em minus 0.4em\relax Now Publishers Inc, 2011.

\bibitem{bekkerman2011scaling}
R.~Bekkerman, M.~Bilenko, and J.~Langford, \emph{Scaling up machine learning:
  Parallel and distributed approaches}.\hskip 1em plus 0.5em minus 0.4em\relax
  Cambridge University Press, 2011.

\bibitem{raja2015cloud}
H.~Raja and W.~U. Bajwa, ``Cloud k-svd: A collaborative dictionary learning
  algorithm for big, distributed data,'' \emph{IEEE Transactions on Signal
  Processing}, vol.~64, no.~1, pp. 173--188, 2015.

\bibitem{assran2019stochastic}
M.~Assran, N.~Loizou, N.~Ballas, and M.~Rabbat, ``Stochastic gradient push for
  distributed deep learning,'' in \emph{International Conference on Machine
  Learning}.\hskip 1em plus 0.5em minus 0.4em\relax PMLR, 2019, pp. 344--353.

\bibitem{tsianos2012push}
K.~I. Tsianos, S.~Lawlor, and M.~G. Rabbat, ``Push-sum distributed dual
  averaging for convex optimization,'' in \emph{2012 ieee 51st ieee conference
  on decision and control (cdc)}.\hskip 1em plus 0.5em minus 0.4em\relax IEEE,
  2012, pp. 5453--5458.

\bibitem{tsianos2012consensus}
------, ``Consensus-based distributed optimization: Practical issues and
  applications in large-scale machine learning,'' in \emph{2012 50th annual
  allerton conference on communication, control, and computing
  (allerton)}.\hskip 1em plus 0.5em minus 0.4em\relax IEEE, 2012, pp.
  1543--1550.

\bibitem{yuan2018exact}
K.~Yuan, B.~Ying, X.~Zhao, and A.~H. Sayed, ``Exact diffusion for distributed
  optimization and learning—part i: Algorithm development,'' \emph{IEEE
  Transactions on Signal Processing}, vol.~67, no.~3, pp. 708--723, 2018.

\bibitem{tsitsiklis1986distributed}
J.~Tsitsiklis, D.~Bertsekas, and M.~Athans, ``Distributed asynchronous
  deterministic and stochastic gradient optimization algorithms,'' \emph{IEEE
  Transactions on Automatic Control}, vol.~31, no.~9, pp. 803--812, 1986.

\bibitem{nedic2009distributed}
A.~Nedic and A.~Ozdaglar, ``Distributed subgradient methods for multi-agent
  optimization,'' \emph{IEEE Transactions on Automatic Control}, vol.~54,
  no.~1, pp. 48--61, 2009.

\bibitem{nedic2010constrained}
A.~Nedic, A.~Ozdaglar, and P.~A. Parrilo, ``Constrained consensus and
  optimization in multi-agent networks,'' \emph{IEEE Transactions on Automatic
  Control}, vol.~55, no.~4, pp. 922--938, 2010.

\bibitem{ram2010distributed}
S.~S. Ram, A.~Nedi{\'c}, and V.~V. Veeravalli, ``Distributed stochastic
  subgradient projection algorithms for convex optimization,'' \emph{Journal of
  optimization theory and applications}, vol. 147, no.~3, pp. 516--545, 2010.

\bibitem{duchi2011dual}
J.~C. Duchi, A.~Agarwal, and M.~J. Wainwright, ``Dual averaging for distributed
  optimization: Convergence analysis and network scaling,'' \emph{IEEE
  Transactions on Automatic control}, vol.~57, no.~3, pp. 592--606, 2011.

\bibitem{zhu2011distributed}
M.~Zhu and S.~Mart{\'\i}nez, ``On distributed convex optimization under
  inequality and equality constraints,'' \emph{IEEE Transactions on Automatic
  Control}, vol.~57, no.~1, pp. 151--164, 2011.

\bibitem{jakovetic2014fast}
D.~Jakoveti{\'c}, J.~Xavier, and J.~M. Moura, ``Fast distributed gradient
  methods,'' \emph{IEEE Transactions on Automatic Control}, vol.~59, no.~5, pp.
  1131--1146, 2014.

\bibitem{shi2015extra}
W.~Shi, Q.~Ling, G.~Wu, and W.~Yin, ``Extra: An exact first-order algorithm for
  decentralized consensus optimization,'' \emph{SIAM Journal on Optimization},
  vol.~25, no.~2, pp. 944--966, 2015.

\bibitem{wei20131}
E.~Wei and A.~Ozdaglar, ``On the $\mathcal{O}(1/k)$ convergence of asynchronous
  distributed alternating direction method of multipliers,'' in \emph{2013 IEEE
  Global Conference on Signal and Information Processing}.\hskip 1em plus 0.5em
  minus 0.4em\relax IEEE, 2013, pp. 551--554.

\bibitem{shi2014linear}
W.~Shi, Q.~Ling, K.~Yuan, G.~Wu, and W.~Yin, ``On the linear convergence of the
  {ADMM} in decentralized consensus optimization,'' \emph{IEEE Transactions on
  Signal Processing}, vol.~62, no.~7, pp. 1750--1761, 2014.

\bibitem{mokhtari2016dqm}
A.~Mokhtari, W.~Shi, Q.~Ling, and A.~Ribeiro, ``{DQM}: Decentralized
  quadratically approximated alternating direction method of multipliers,''
  \emph{IEEE Transactions on Signal Processing}, vol.~64, no.~19, pp.
  5158--5173, 2016.

\bibitem{mokhtari2016decentralized}
------, ``A decentralized second-order method with exact linear convergence
  rate for consensus optimization,'' \emph{IEEE Transactions on Signal and
  Information Processing over Networks}, vol.~2, no.~4, pp. 507--522, 2016.

\bibitem{aybat2017distributed}
N.~S. Aybat, Z.~Wang, T.~Lin, and S.~Ma, ``Distributed linearized alternating
  direction method of multipliers for composite convex consensus
  optimization,'' \emph{IEEE Transactions on Automatic Control}, vol.~63,
  no.~1, pp. 5--20, 2017.

\bibitem{aybat2015asynchronous}
N.~Aybat, Z.~Wang, and G.~Iyengar, ``An asynchronous distributed proximal
  gradient method for composite convex optimization,'' in \emph{International
  Conference on Machine Learning}.\hskip 1em plus 0.5em minus 0.4em\relax PMLR,
  2015, pp. 2454--2462.

\bibitem{aybat16}
N.~S. Aybat and E.~Yazdandoost~Hamedani, ``A primal-dual method for conic
  constrained distributed optimization problems,'' in \emph{Advances in Neural
  Information Processing Systems}, D.~Lee, M.~Sugiyama, U.~Luxburg, I.~Guyon,
  and R.~Garnett, Eds., vol.~29.\hskip 1em plus 0.5em minus 0.4em\relax Curran
  Associates, Inc., 2016.

\bibitem{aybat2019distributed}
N.~S. Aybat and E.~Y. Hamedani, ``A distributed admm-like method for resource
  sharing over time-varying networks,'' \emph{SIAM Journal on Optimization},
  vol.~29, no.~4, pp. 3036--3068, 2019.

\bibitem{kempe2003gossip}
D.~Kempe, A.~Dobra, and J.~Gehrke, ``Gossip-based computation of aggregate
  information,'' in \emph{44th Annual IEEE Symposium on Foundations of Computer
  Science, 2003. Proceedings.}\hskip 1em plus 0.5em minus 0.4em\relax IEEE,
  2003, pp. 482--491.

\bibitem{nedic2014distributed}
A.~Nedi{\'c} and A.~Olshevsky, ``Distributed optimization over time-varying
  directed graphs,'' \emph{IEEE Transactions on Automatic Control}, vol.~60,
  no.~3, pp. 601--615, 2014.

\bibitem{xi2017dextra}
C.~Xi and U.~A. Khan, ``Dextra: A fast algorithm for optimization over directed
  graphs,'' \emph{IEEE Transactions on Automatic Control}, vol.~62, no.~10, pp.
  4980--4993, 2017.

\bibitem{nedic2017achieving}
A.~Nedic, A.~Olshevsky, and W.~Shi, ``Achieving geometric convergence for
  distributed optimization over time-varying graphs,'' \emph{SIAM Journal on
  Optimization}, vol.~27, no.~4, pp. 2597--2633, 2017.

\bibitem{xi2017add}
C.~Xi, R.~Xin, and U.~A. Khan, ``{ADD-OPT}: Accelerated distributed directed
  optimization,'' \emph{IEEE Transactions on Automatic Control}, vol.~63,
  no.~5, pp. 1329--1339, 2017.

\bibitem{xin2018linear}
R.~Xin and U.~A. Khan, ``A linear algorithm for optimization over directed
  graphs with geometric convergence,'' \emph{IEEE Control Systems Letters},
  vol.~2, no.~3, pp. 315--320, 2018.

\bibitem{xin2019distributed}
------, ``Distributed heavy-ball: A generalization and acceleration of
  first-order methods with gradient tracking,'' \emph{IEEE Transactions on
  Automatic Control}, 2019.

\bibitem{pu2020push}
S.~Pu, W.~Shi, J.~Xu, and A.~Nedic, ``Push-pull gradient methods for
  distributed optimization in networks,'' \emph{IEEE Transactions on Automatic
  Control}, 2020.

\bibitem{xin2019distributedNEST}
R.~Xin, D.~Jakoveti{\'c}, and U.~A. Khan, ``Distributed nesterov gradient
  methods over arbitrary graphs,'' \emph{IEEE Signal Processing Letters},
  vol.~26, no.~8, pp. 1247--1251, 2019.

\bibitem{xi2018linear}
C.~Xi, V.~S. Mai, R.~Xin, E.~H. Abed, and U.~A. Khan, ``Linear convergence in
  optimization over directed graphs with row-stochastic matrices,'' \emph{IEEE
  Transactions on Automatic Control}, vol.~63, no.~10, pp. 3558--3565, 2018.

\bibitem{xin2019frost}
R.~Xin, C.~Xi, and U.~A. Khan, ``{FROST—Fast} row-stochastic optimization
  with uncoordinated step-sizes,'' \emph{EURASIP Journal on Advances in Signal
  Processing}, vol. 2019, no.~1, pp. 1--14, 2019.

\bibitem{lu2020nesterov}
Q.~L{\"u}, X.~Liao, H.~Li, and T.~Huang, ``A nesterov-like gradient tracking
  algorithm for distributed optimization over directed networks,'' \emph{IEEE
  Transactions on Systems, Man, and Cybernetics: Systems}, 2020.

\bibitem{tian2018asy}
Y.~Tian, Y.~Sun, and G.~Scutari, ``Asy-sonata: Achieving linear convergence in
  distributed asynchronous multiagent optimization,'' in \emph{2018 56th Annual
  Allerton Conference on Communication, Control, and Computing
  (Allerton)}.\hskip 1em plus 0.5em minus 0.4em\relax IEEE, 2018, pp. 543--551.

\bibitem{zhang2019asyspa}
J.~Zhang and K.~You, ``Asyspa: An exact asynchronous algorithm for convex
  optimization over digraphs,'' \emph{IEEE Transactions on Automatic Control},
  vol.~65, no.~6, pp. 2494--2509, 2019.

\bibitem{assran2020asynchronous}
M.~S. Assran and M.~G. Rabbat, ``Asynchronous gradient push,'' \emph{IEEE
  Transactions on Automatic Control}, vol.~66, no.~1, pp. 168--183, 2020.

\bibitem{xie2017stop}
P.~Xie, K.~You, and C.~Wu, ``How to stop consensus algorithms, locally?'' in
  \emph{2017 IEEE 56th Annual Conference on Decision and Control (CDC)}.\hskip
  1em plus 0.5em minus 0.4em\relax IEEE, 2017, pp. 4544--4549.

\bibitem{prakash2019distributed}
M.~Prakash, S.~Talukdar, S.~Attree, V.~Yadav, and M.~V. Salapaka, ``Distributed
  stopping criterion for consensus in the presence of delays,'' \emph{IEEE
  Transactions on Control of Network Systems}, vol.~7, no.~1, pp. 85--95, 2019.

\bibitem{sayed2014adaptation}
A.~H. Sayed, ``Adaptation, learning, and optimization over networks,''
  \emph{Foundations and Trends in Machine Learning}, vol.~7, no. ARTICLE, pp.
  311--801, 2014.

\bibitem{xu2020accelerated}
J.~Xu, Y.~Tian, Y.~Sun, and G.~Scutari, ``Accelerated primal-dual algorithms
  for distributed smooth convex optimization over networks,'' in
  \emph{International Conference on Artificial Intelligence and
  Statistics}.\hskip 1em plus 0.5em minus 0.4em\relax PMLR, 2020, pp.
  2381--2391.

\bibitem{hamedani2021decentralized}
E.~Y. Hamedani and N.~S. Aybat, ``A decentralized primal-dual method for
  constrained minimization of a strongly convex function,'' \emph{IEEE
  Transactions on Automatic Control}, vol.~67, no.~11, pp. 5682--5697, 2021.

\bibitem{chambolle2016ergodic}
A.~Chambolle and T.~Pock, ``On the ergodic convergence rates of a first-order
  primal--dual algorithm,'' \emph{Mathematical Programming}, vol. 159, no.~1,
  pp. 253--287, 2016.

\bibitem{qu2017harnessing}
G.~Qu and N.~Li, ``Harnessing smoothness to accelerate distributed
  optimization,'' \emph{IEEE Transactions on Control of Network Systems},
  vol.~5, no.~3, pp. 1245--1260, 2017.

\bibitem{Robbins71}
H.~Robbins and D.~Siegmund, \emph{Optimizing methods in statistics (Proc.
  Sympos., Ohio State Univ., Columbus, Ohio, 1971)}.\hskip 1em plus 0.5em minus
  0.4em\relax New York: Academic Press, 1971, ch. A convergence theorem for non
  negative almost supermartingales and some applications, pp. 233 -- 257.

\bibitem{horn2012matrix}
R.~A. Horn and C.~R. Johnson, \emph{Matrix analysis}.\hskip 1em plus 0.5em
  minus 0.4em\relax Cambridge university press, 2012.

\bibitem{polyak1987introduction}
B.~T. Polyak, ``Introduction to optimization. optimization software,''
  \emph{Inc., Publications Division, New York}, vol.~1, 1987.

\bibitem{sridharan2008fast}
K.~Sridharan, S.~Shalev-Shwartz, and N.~Srebro, ``Fast rates for regularized
  objectives,'' \emph{Advances in neural information processing systems},
  vol.~21, pp. 1545--1552, 2008.

\bibitem{chang2011libsvm}
C.-C. Chang and C.-J. Lin, ``{LIBSVM}: A library for support vector machines,''
  \emph{ACM transactions on intelligent systems and technology (TIST)}, vol.~2,
  no.~3, pp. 1--27, 2011.

\end{thebibliography}
%\bibliographystyle{plain}
\end{document}